\documentclass[reqno]{amsart}

\usepackage{hyperref}
\usepackage{amsmath,amsthm,amssymb}
\usepackage{tikz}

\usetikzlibrary{intersections,calc,angles,arrows.meta,patterns.meta}

\newtheorem{thm}{Theorem}
\newtheorem{cor}[thm]{Corollary}
\newtheorem{prop}[thm]{Proposition}
\newtheorem{lem}[thm]{Lemma}
\newtheorem{clm}[thm]{Claim}

\theoremstyle{definition}
\newtheorem{dfn}[thm]{Definition}
\newtheorem{rem}[thm]{Remark}
\newtheorem{ex}[thm]{Example}
\newtheorem{prob}[thm]{Problem}
\newtheorem{nota}[thm]{Notation}

\numberwithin{thm}{section}
\numberwithin{equation}{section}

\DeclareMathOperator{\im}{Im}
\DeclareMathOperator{\length}{length}

\newcommand{\step}[1]{\medskip\noindent\textit{#1.}}

\title[Topological regularity of Busemann spaces]{Topological regularity of Busemann spaces of nonpositive curvature}

\author[T. Fujioka]{Tadashi Fujioka}
\address[T. Fujioka]{Department of Applied Mathematics, Fukuoka University, Fukuoka 814-0180, Japan}
\email{tfujioka210@gmail.com}

\author[S. Gu]{Shijie Gu}
\address[S. Gu]{Department of Mathematics, Northeastern University, Shenyang, Liaoning, China, 110004}
\email{shijiegutop@gmail.com}

\date{\today}


\begin{document}

\begin{abstract}
We extend the topological results of Lytchak--Nagano and Lytchak--Nagano--Stadler for CAT($0$) spaces to the setting of Busemann spaces of nonpositive curvature, i.e., BNPC spaces.
We give a characterization of locally BNPC topological manifolds in terms of their links and show that the singular set of a locally BNPC homology manifold is discrete.
We also prove that any (globally) BNPC topological $4$-manifold is homeomorphic to Euclidean space.
Applications include a topological stability theorem for locally BNPC G-spaces.
Our arguments also apply to spaces admitting convex geodesic bicombings.
\end{abstract}

\maketitle
\tableofcontents

\section{Introduction}\label{sec:intro}

There are two important notions of globally nonpositive curvature defined for metric spaces. One arises from Alexandrov's pioneering comparison geometry and is nowadays called \textit{CAT(0)}, while the other goes back to Busemann and will be referred to here as \textit{BNPC} (also known in the literature as a Busemann space or a convex space).
A complete geodesic space $X$ is BNPC if, for any two constant-speed geodesics $\gamma,\eta:[0,1]\to X$, the restriction of the distance function
\[t\mapsto d(\gamma(t),\eta(t))\]
is convex on $[0,1]$.
Although CAT($0$) spaces are BNPC spaces, the class of BNPC spaces is much larger than the class of CAT($0$) spaces.
For example, an $L^p$ space is CAT($0$) only for $p=2$, yet it is BNPC for any $1<p<\infty$.
More generally, CAT($0$) spaces generalize complete, simply connected Riemannian manifolds of nonpositive sectional curvature, whereas BNPC spaces can be viewed as a generalization of such Finsler manifolds.
See Section \ref{sec:pre} for more background information.

The purpose of this paper is to extend recent topological results of Lytchak--Nagano \cite{LN:geo}, \cite{LN:top} and Lytchak--Nagano--Stadler \cite{LNS} (building on earlier work of Thurston \cite{Th:cat}) for CAT($0$) spaces to the setting of BNPC spaces.
This gives a positive answer to the problem raised in \cite[Section 1.6]{LNS}.
Although many of the original results were stated for more general CAT($\kappa$) spaces, we restrict
ourselves here to the most basic case of nonpositive curvature.
As mentioned above, the difference between Alexandrov and Busemann curvature bounds closely reflects the difference between Riemannian and Finsler geometries.
Therefore, the key to the extension is to capture the Finsler nature of BNPC spaces and deduce from it the same topological properties as in the CAT($0$) case.

More specifically, the main technical contribution of this paper is the development of a strainer theory for geodesically complete BNPC spaces (more precisely, for \textit{GNPC spaces} defined later).
Strainers are a classical tool in Alexandrov geometry, originally introduced in the study of
Alexandrov spaces with curvature bounded below \cite{BGP}, and are also the key ingredient in the aforementioned studies of CAT($0$) spaces \cite{LN:geo}, \cite{LN:top}, \cite{LNS}.
However, in contrast to the CAT($0$) setting, BNPC spaces do not have a canonical angle structure for introducing strainers, reflecting their Finsler-like nature.
We overcome this difficulty by separating the two roles played by angles in the CAT($0$) theory: one controlling the branching of geodesics and the other controlling
orthogonality.
These technical details will be discussed in Section \ref{sec:tool}.

\subsection{Main results}\label{sec:main}

The main results concern topological manifolds and homology manifolds with BNPC metrics.
A metric space is \textit{locally BNPC} if every point has a neighborhood that is BNPC with respect to the ambient metric, which we call a \textit{BNPC neighborhood}.

\begin{thm}\label{thm:reg}
Let $X$ be a locally compact, locally BNPC space.
Then $X$ is a topological $n$-manifold if and only if for any $p\in X$, every punctured ball around $p$ in its BNPC neighborhood is homotopy equivalent to the $(n-1)$-sphere.
\end{thm}

\begin{thm}\label{thm:sing}
Let $X$ be a locally BNPC homology manifold.
Then the set of non-manifold points in $X$ is discrete.
\end{thm}

\begin{thm}\label{thm:four}
Let $X$ be a (globally) BNPC topological manifold of dimension $4$.
Then $X$ is homeomorphic to Euclidean space.
\end{thm}

Theorems \ref{thm:reg} and \ref{thm:sing} generalize, respectively, \cite[Theorems 1.1, 1.2]{LN:top} of Lytchak--Nagano in the CAT($0$) case, and Theorem \ref{thm:four} generalizes \cite[Theorem 1.1]{LNS} of Lytchak--Nagano--Stadler.
In Theorem \ref{thm:reg}, the homotopy type of a punctured ball in a BNPC neighborhood is independent of its radius (and of its openness or closedness), thanks to the unique geodesic contraction centered at $p$.
Although Theorem \ref{thm:reg} may look different from \cite[Theorem 1.1]{LN:top}, they are essentially the same, because in the CAT($0$) case any small punctured ball around $p$ is homotopy equivalent to the space of directions at $p$ (\cite[Theorem A]{Kram}; see also Remark \ref{rem:reg}).
Note that Theorems \ref{thm:reg} and \ref{thm:sing} are local results, whereas Theorem \ref{thm:four} is a global one.
For more generalizations, such as general curvature bounds, see Section \ref{sec:gen}.

\subsection{Comments}\label{sec:com}

Here are some more comments on the main theorems, which largely parallel the CAT($0$) case.

First, we remark that Theorems \ref{thm:reg} and \ref{thm:sing} for Alexandrov spaces with curvature bounded below were proved by Wu \cite{Wu:alex} (cf.\ \cite{Wu:ed}).

In the situation of Theorem \ref{thm:reg}, the punctured ball around $p$ is homotopy equivalent to any sufficiently small metric sphere around $p$; see Theorem \ref{thm:link}.
However, for dimensions $\ge 5$, such a metric sphere need not be homeomorphic to a sphere, as a consequence of the double suspension theorem of Cannon \cite{Ca} and Edwards \cite{Ed} (cf.\ \cite{Be:alex}).
On the other hand, for dimension $\le 4$, every sufficiently small metric sphere around $p$ is a topological sphere, exactly as in the CAT($0$) case (Corollary \ref{cor:4link}).

Theorem \ref{thm:sing} for Alexandrov curvature bounds was a question raised by Quinn \cite[Problem 7.2]{Qu}.
Note that for dimension $\ge4$, a locally BNPC homology manifold is not necessarily a topological manifold.
On the other hand, for dimension $\le3$, any locally BNPC homology manifold is a topological manifold.
The dimension $1$ and $2$ cases are general facts holding without the BNPC condition (Theorem \ref{thm:moo}).
The dimension $3$ case was essentially observed by Thurston \cite[Theorem 3.3]{Th:cat} (see Theorem \ref{thm:three}).
Our argument provides an alternative proof of this fact in Corollary \ref{cor:three}.

Theorem \ref{thm:four} provides a definitive answer to a question of Gromov \cite[Section 2.1]{Gr}, asking whether there exists a BNPC topological manifold different from Euclidean space. The same statement as Theorem \ref{thm:four} is known to hold in dimensions $\le 3$ and to fail in dimensions $\ge 5$.
The dimension $1$ and $2$ cases follow from the contractibility of BNPC spaces.
The dimension $3$ case is a combination of the results of Brown \cite{Br} and Rolfsen \cite{Ro} (Theorem \ref{thm:three}).
A counterexample in dimension $\ge 5$ was constructed by Davis--Januszkiewicz \cite{DJ} (cf.\ \cite{ADG}) in the CAT($0$) setting.
Our theorem thus provides an affirmative solution to the only remaining case, extending the result of Lytchak--Nagano--Stadler \cite{LNS} in the CAT($0$) case.

For additional results related to the main theorems, see Sections \ref{sec:prfloc} and \ref{sec:prfglo}.

\subsection{General results}\label{sec:res}

In our main theorems, the (locally) BNPC spaces under consideration are \textit{locally geodesically complete}, i.e., any geodesic is locally extendable (Proposition \ref{prop:geo}).
The original results of Lytchak--Nagano and Lytchak--Nagano--Stadler were applications of the theory of \textit{GCBA spaces}, developed in \cite{LN:geo}, \cite{LN:top}, \cite{LNS}.
A GCBA space is a separable, locally compact, locally geodesically complete, locally CAT($\kappa$) space.
For recent results on GCBA spaces, see \cite{Na:vol}, \cite{Na:asy}, \cite{Na:wall}, \cite{NSY:1}, \cite{NSY:2}, \cite{CS}, \cite{Fu:nc}, \cite{Fu:ex}.

Similarly, in this paper we develop the theory of \textit{GNPC spaces}.
A GNPC space is a separable, locally compact, locally geodesically complete, locally BNPC space (the separability is not necessary for our purposes).
Every GCBA space of nonpositive curvature is, in fact, a GNPC space.
Throughout this paper, all spaces under consideration are GNPC spaces.
We prove the following two theorems in this general setting; compare with \cite[Theorems 1.2, 1.12]{LN:geo} (cf.\ \cite{Kl}, \cite{Kram}).

\begin{thm}\label{thm:mfd}
Let $X$ be a GNPC space.
Then $X$ contains an open dense subset of manifold points (of possibly different dimensions).
\end{thm}

\begin{thm}\label{thm:link}
Let $X$ be a GNPC space and $p\in X$.
Then any sufficiently small metric sphere around $p$ has the same homotopy type, independent of its radius, and is also homotopy equivalent to any sufficiently small punctured ball around $p$.
\end{thm}

Lytchak--Nagano \cite{LN:geo} further revealed the geometric structure of GCBA spaces, such as the tangent cone, measure-theoretic stratification, and almost everywhere Riemannian structure.
However, we focus here on the topological aspects of GNPC spaces, and only delve into as much geometric structure as needed for the proofs of the main theorems.
We leave more detailed geometric aspects for future work.
See Section \ref{sec:prob} for open problems (see also Remark \ref{rem:prob} for the latest information after the first version of this paper appeared on arXiv).

\subsection{Applications}\label{sec:app}

Busemann \cite{Bu:geo} also introduced the notion of \textit{G-space}, which can be viewed as a qualitative generalization of Finsler manifolds.
The BNPC condition was first introduced in the context of G-spaces (\cite{Bu:npc}, \cite{Bu:geo}).
In simple terms, a G-space is a complete, locally compact geodesic space with local unique extendability of geodesics.
See Section \ref{sec:prfapp} for background.

In this paper we only consider G-spaces with (locally) BNPC metrics.
Note that a locally BNPC G-space is nothing but a complete GNPC space without branching geodesics.
The nonbranching assumption for geodesics is very strong and prevents the existence of singular points of GNPC spaces.
Indeed, Andreev \cite{An:prf} proved that any locally (resp.\ globally) BNPC G-space is a topological manifold (resp.\ homeomorphic to Euclidean space).
Compare with the general situation described in Section \ref{sec:com} (see also \cite{Fuk}).
As a corollary of Theorem \ref{thm:mfd}, we obtain an alternative proof of Andreev's result.

\begin{thm}\label{thm:g}
Any locally BNPC G-space is a topological manifold.
Any (globally) BNPC G-space is homeomorphic to Euclidean space.
\end{thm}

Furthermore, any small metric sphere in a locally BNPC G-space is a topological sphere (Theorem \ref{thm:glink}).
Hence the pathological phenomena arising from the double suspension theorem mentioned in Section \ref{sec:com} do not occur.

Andreev \cite{An:norm} later proved a more rigid result: the tangent cone of a locally BNPC G-space is a strictly convex normed space.
In other words, any GNPC space without branching geodesics admits a Finsler-like structure.
This can be viewed as a model case of the geometric structure of the ``regular part'' of a GNPC space.
Compare with the Riemannian structure of the regular part of a GCBA space: \cite[Theorem 1.3]{LN:geo}, \cite{Be:up}.
See Section \ref{sec:prob} for more details.

As an application of Theorem \ref{thm:reg}, we obtain the following topological stability theorem for locally BNPC G-spaces.
This partially generalizes \cite[Theorem 1.3]{LN:top}.
The \textit{injectivity radius} of a G-space $X$ is the supremum of $r>0$ such that for any $p\in X$ every shortest path emanating from $p$ of length $<r$ extends uniquely to length $r$ (i.e., without branching).

\begin{thm}\label{thm:stab}
Let $X_j$ be a sequence of compact, locally BNPC G-spaces of dimension $n$ (as manifolds) and injectivity radius $\ge r>0$, where $n$ and $r$ are independent of $j$.
Suppose $X_j$ converges to a compact metric space $X$ in the Gromov--Hausdorff topology.
Then $X$ is a topological manifold and $X_j$ is homeomorphic to $X$ for any sufficiently large $j$.
\end{thm}

Note that the limit space \(X\) in Theorem~\ref{thm:stab} may fail to be a G-space if geodesics branch; it may even fail to be locally BNPC if geodesics are not unique.
For instance, $L^p$ spaces ($1<p<\infty$) converge to $L^1$ or $L^\infty$ space under the Gromov--Hausdorff topology.
However, as we explain in the next subsection, the limit space locally admits a convex geodesic bicombing.

\subsection{Generalizations}\label{sec:gen}

We now discuss two possible/potential generalizations.

One is a possible generalization to spaces with convex geodesic bicombings.
The notion of \textit{convex geodesic bicombing}, introduced by Descombes--Lang \cite{DL:conv}, is a generalization of the BNPC condition.
Similar approaches can also be found in \cite{BP} and \cite{Kl}.
Roughly speaking, all these approaches weaken the BNPC condition by dropping the uniqueness of geodesics and instead imposing the convexity condition only on some suitable choice of geodesics.
See Section \ref{sec:prfgen} for a precise definition.
For example, all normed spaces admit convex geodesic bicombings by choosing the affine geodesics.
The advantage of this treatment is that this class is closed under limiting operations, which is necessary to prove Theorem \ref{thm:stab}.
We claim

\begin{thm}\label{thm:conv}
All the above statements remain valid if we replace ``BNPC spaces'' by ``spaces with convex geodesic bicombings'' in appropriate manners (see Section \ref{sec:prfgen} for more details).
\end{thm}

The other is a potential generalization to general curvature bounds.
The original statements of Lytchak--Nagano and Lytchak--Nagano--Stadler were established for CAT($\kappa$) spaces with arbitrary $\kappa\in\mathbb R$ (even a global result; see \cite[Corollary 1.8]{LNS}).
By using triangle comparison, Busemann curvature bounds can also be defined for any $\kappa\in\mathbb R$ (Remark \ref{rem:curv}).
While it should be possible to extend our arguments to that broader setting, we do not pursue this here and leave it as an open problem (Problem~\ref{prob:curv}).

\subsection{Main tool}\label{sec:tool}

We now explain how to prove the main theorems.
Our proofs are parallel to the original proofs of Lytchak--Nagano \cite{LN:geo}, \cite{LN:top} and Lytchak--Nagano--Stadler \cite{LNS}.
The original proofs can be divided into a topological part and a geometrical part.
In our generalizations, the topological part is the same as the original one, and the problem to overcome lies in the geometrical part (contrary to the title of this paper).

The main technical tool in the geometrical part of the original proofs is the notion of a \textit{strainer} (and its variant called an \textit{extended strainer}).
Roughly speaking, a strainer at a point $p$ is a finite collection of points around $p$ that behaves like an almost orthonormal family of vectors in the tangent space at $p$.

Historically, strainers were first introduced by Burago--Gromov--Perelman \cite{BGP} in the study of Alexandrov spaces with curvature bounded below.
Lytchak--Nagano \cite{LN:geo}, \cite{LN:top} later adapted them to the setting of GCBA spaces.
Extended strainers were introduced by Lytchak--Nagano--Stadler \cite{LNS} to handle global problems on GCBA spaces.
In this paper, we define strainers and extended strainers for GNPC spaces, prove their existence, and establish their topological properties.
This is the main contribution of this paper.

Although the various notions of strainers in different settings are defined abstractly using angles, their underlying constructions are essentially the same and remain meaningful (whether or not they are immediately useful) in general metric spaces.
Since GNPC spaces, like Finsler manifolds, have no fixed notion of angle, we now recall the concrete construction based on metric spheres.

Let $X$ be a metric space.
We denote by $\partial B(p,r)$ the metric sphere centered at $p\in X$ of radius $r>0$. 
Fix the first point $p_1\in X$ and choose the second point $p_2\in\partial B(p_1,r_1)$, where $r_1\ll1$ depends on $p_1$ (we discuss the dependence later).
Next choose the third point $p_3\in\partial B(p_1,r_1)\cap\partial B(p_2,r_2)$, where $r_2\ll r_1$ depends on $p_2$.
Repeating this procedure, one can construct a sequence $p_1,\dots,p_k,p_{k+1}$ of points in $X$ such that
\[p_{i+1}\in\bigcap_{j=1}^i\partial B(p_j,r_j),\quad r_{i+1}\ll\ r_i\]
for any $1\le i\le k$, where $r_i$ depends on $p_i$. (Note that the construction may stop if the intersection is empty.)
The sequence $p_1,\dots,p_k$ is called a \textit{$k$-strainer} at $p_{k+1}$.
The distance map
\[f=(d(p_1,\cdot),\dots,d(p_k,\cdot)):X\to\mathbb R^k\]
is called a \textit{$k$-strainer map} at $p_{k+1}$.
See Figure \ref{fig:tool}.

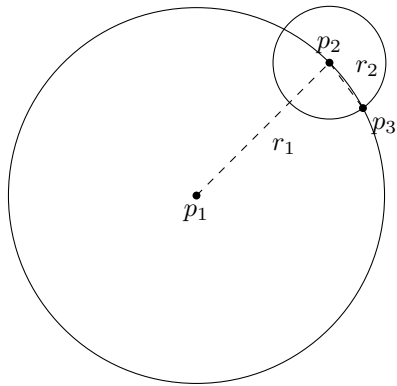
\begin{figure}[ht]
\centering
\begin{tikzpicture}
\coordinate[label=below:$p_1$](p1)at(0,0);
\coordinate[label=above:$p_2$](p2)at(45:2.5);

\draw[name path=r1](p1)circle[radius=2.5];
\draw[name path=r2](p2)circle[radius=0.75];

\path[name intersections={of=r1 and r2}];
\coordinate[label=below right:$p_3$](p3)at(intersection-2);

\fill(p1)circle(1.5pt);
\fill(p2)circle(1.5pt);
\fill(p3)circle(1.5pt);

\draw[dashed](p1)to(p2);
\draw[dashed](p2)to(p3);

\node[below right]at($(p1)!.5!(p2)$){$r_1$};
\node[above right]at($(p2)!.5!(p3)$){$r_2$};
\end{tikzpicture}
\caption{}\label{fig:tool}
\end{figure}

If $X$ is an Alexandrov or GCBA space, the strainer $p_1,\dots,p_k$ behaves like an almost orthonormal family of vectors at $p_{k+1}$.
As a result, the associated strainer map $f$ has some fibration property around $p_{k+1}$.
We recall this in more detail below.

Suppose $X$ is an Alexandrov space.
Then the geodesics in $X$ are generally not extendable (but never branch).
However, if $r_i$ is sufficiently small depending on $p_i$, the geodesics emanating from $p_i$ become \textit{almost extendable} in the $r_i$-neighborhood of $p_i$, especially at $p_{k+1}$.
Moreover, if $r_i\ll r_j$, these almost extendable geodesics from $p_i$ and $p_j$ to $p_{k+1}$ are \textit{almost orthogonal} at $p_{k+1}$ for any $1\le i\neq j\le k$.
These properties, together with a deep analysis using geometric topology, imply that $f$ is locally a trivial fibration around $p_{k+1}$ (\cite{Per:alex}, \cite{Per:mor}).

Suppose $X$ is a GCBA space.
Then the geodesics in $X$ are always extendable, but may branch.
Nevertheless, if $r_i$ is sufficiently small depending on $p_i$, the geodesics emanating from $p_i$ are \textit{almost nonbranching} in the $r_i$-neighborhood of $p_i$, especially at $p_{k+1}$.
Furthermore, if $r_i\ll r_j$, these almost nonbranching geodesics from $p_i$ and $p_j$ to $p_{k+1}$ are \textit{almost orthogonal} at $p_{k+1}$ for any $1\le i\neq j\le k$.
These properties, combined with a retraction argument using geometric topology, show that $f$ is locally a Hurewicz fibration around $p_{k+1}$ (\cite{LN:geo}, \cite{LN:top}; but not a trivial fibration because of the complexity of singularity).

In both cases, the key properties of a strainer are the almost extendability (in the Alexandrov case) or the almost nonbranching (in the GCBA case), and the almost orthogonality (in both cases) of geodesics.
All these properties are established using the unique notion of angle, whose existence follows from the Riemann-type curvature bounds of Alexandrov/GCBA spaces.

Now let us turn to the case of GNPC spaces.
As in the case of GCBA spaces, what we want to prove are:
\begin{enumerate}
\item the almost nonbranching and almost orthogonality of a strainer, and
\item the fibration property of a strainer map based on them.
\end{enumerate}

Regarding (1), the main difficulty is the lack of a fixed angle notion in GNPC spaces, reflecting their Finsler-like character.
To overcome this, we introduce two distinct notions of angle which are tied to the branching and orthogonality of geodesics, respectively.
These angles coincide in GCBA spaces but can differ in general GNPC spaces.
This will be discussed in Section \ref{sec:key}.

Regarding (2), the challenge is the asymmetry of the almost orthogonality of a GNPC strainer, which is addressed by using a modified
consecutive-approximation argument.
An informal analogy is as follows.
Let $f$ and $g$ be smooth maps from $\mathbb R^n$ to $\mathbb R^k$ ($n\ge k$), whose Jacobian matrices are
\begin{equation*}
df=
\left(
\begin{array}{ccc|c}
1\pm\delta & & \pm\delta & \\
&  \ddots &  & ?\\
\pm\delta & & 1\pm\delta &
\end{array} 
\right)
,\quad
dg=
\left(
\begin{array}{ccc|c}
1\pm\delta & & \pm\delta & \\
&  \ddots &  & ?\\
\pm C & & 1\pm\delta &
\end{array} 
\right),
\end{equation*}
where $\delta$ is a small positive number and $C$ is a uniformly bounded positive number.
An Alexandrov/GCBA strainer map looks like $f$, whereas a GNPC strainer map looks like $g$.
The first $k$ diagonal entries of $df$ being almost $1$ represents the almost extendability or almost nonbranching, while the other entries being almost $0$ correspond to the almost orthogonality.
The lower triangular entries of $dg$ are responsible for the infinitesimal deviation of our GNPC metric from the standard Euclidean metric.
In either case, however, if $\delta$ is small enough compared to $C$, the Jacobian matrices are regular and thus, after appropriate coordinate changes, $f$ and $g$ are locally projection maps.
An analogous argument allows us to prove the fibration property of a GNPC strainer map.
This will be discussed in Section \ref{sec:str}.

In addition to the the absence of a canonical angle structure, another difficulty overcome in Sections \ref{sec:key} and \ref{sec:str} is the possible ``non-differentiability'' of sufficiently small metric spheres in GNPC spaces, which is not explicitly stated but is implicitly handled in our arguments.
This stems from the fact that the GNPC setting differs from the GCBA setting by allowing non-smooth normed spaces, reflecting again its Finsler-like nature.
This perspective will play an important role in future studies of the geometric structure of GNPC spaces; see Remarks \ref{rem:dich} and \ref{rem:prob}.

Furthermore, in Section \ref{sec:exstr}, we introduce an extended strainer for a GNPC space.
Our treatment of an extended strainer is slightly different from the original one of Lytchak--Nagano--Stadler \cite{LNS}: we introduce the associated \textit{half strainer map} and prove its fibration property, in parallel with that of the standard strainer maps.
This half-strainer formulation, and
in particular the corresponding fibration theorem, does not appear in \cite{LNS} in this form, and constitutes another contribution of this paper.

Hence, Sections \ref{sec:key}, \ref{sec:str}, and \ref{sec:exstr} constitute the geometric core of this paper.
Once the existence and fibration properties of strainers and extended strainers are established, the proofs of the main topological
theorems follow the GCBA case and rely mainly on topological arguments.

\step{Organization}
In Section \ref{sec:pre}, we recall the basics of BNPC spaces, introduce GNPC spaces, and present some results from geometric topology.
In Section \ref{sec:fac}, we collect several facts known for CAT($0$) spaces that directly extend to BNPC spaces.
In Section \ref{sec:key}, we prove key lemmas on the almost nonbranching and almost orthogonality of geodesics in GNPC spaces.
In Section \ref{sec:str}, we develop strainers for GNPC spaces and prove Theorems \ref{thm:mfd} and \ref{thm:link}.
In Section \ref{sec:exstr}, we develop extended strainers for GNPC spaces in parallel with the previous section.
In Section \ref{sec:prf}, we prove Theorems \ref{thm:reg}, \ref{thm:sing}, \ref{thm:four}, \ref{thm:g}, \ref{thm:stab}, and \ref{thm:conv}.
In Section \ref{sec:prob}, we discuss open problems.

\step{Acknowledgments}
The authors are especially grateful to Stephan Stadler for giving us the opportunity to collaborate and for his insightful suggestions.
The first author would like to thank Shin-ichi Ohta and Koichi Nagano for helpful discussions, and Alexander Lytchak for his valuable comments, especially on Section \ref{sec:prob}.
The second author would like to thank Wenyuan Yang for his inquiry regarding potential applications to spaces admitting convex geodesic bicombings, and Shmuel Weinberger for helpful discussions on homology manifolds. The first author was supported by JSPS KAKENHI Grant Numbers 22KJ2099 and 25K23336.
The second author was supported by NSFC grant 12201102.

\section{Preliminaries}\label{sec:pre}

In this section, we recall the definition and basic properties of BNPC spaces and introduce GNPC spaces together with their tiny balls.
We also present several results from geometric topology concerning homology manifolds and Hurewicz fibrations.

\subsection{Notation}\label{sec:nota}
We denote the distance between two points $x$ and $y$ by $d(x,y)$ or $|xy|$.
A shortest path (see the next subsection) between $x$ and $y$ is denoted by $xy$.
The open and closed metric balls, and metric sphere centered at $p$ of radius $r$ are denoted by $B(p,r)$, $\bar B(p,r)$, and $\partial B(p,r)$, respectively.
Open and closed punctured balls mean $B(p,r)\setminus\{p\}$ and $\bar B(p,r)\setminus\{p\}$, respectively.

In Sections \ref{sec:key}, \ref{sec:str}, and \ref{sec:exstr}, we will use the symbols $\delta$, $\varkappa(\delta)$, $\varkappa_k(\delta)$, $\delta_k$, and $\varepsilon_k$ with special meanings.
See Notations \ref{nota:del} and \ref{nota:k}.

\subsection{Geodesic spaces}\label{sec:geo}
A \textit{shortest path} is an isometric embedding of a closed interval into a metric space.
A \textit{geodesic space} is a metric space such that any two points can be joined by a shortest path.
Any complete, locally compact geodesic space is \textit{proper}, i.e., every closed ball is compact (\cite[Theorem 2.5.28]{BBI}).
In other words, a distance function from a point is proper in the usual sense.

A \textit{geodesic} is a curve whose restriction to each sufficiently small sub-interval is a shortest path.
A geodesic space $X$ is called \textit{locally geodesically complete} if every geodesic is locally extendable, or more precisely, for any geodesic $\gamma:[a,b]\to X$, there exist $\varepsilon>0$ and a geodesic
\[\bar\gamma:[a-\varepsilon,b+\varepsilon]\to X\]
such that $\bar\gamma_{[a,b]}=\gamma$.
Similarly, $X$ is called \textit{(globally) geodesically complete} if each geodesic is infinitely extendable, i.e., one can take $\varepsilon=\infty$.
In a complete metric space, local geodesic completeness implies global geodesic completeness (\cite[Proposition 9.1.28]{BBI}).

By definition, every geodesic is unit-speed parametrized.
In this paper, we often consider a linear reparametrization of a geodesic, also described as a constant-speed geodesic.

\subsection{BNPC spaces}\label{sec:bnpc}

Now we introduce BNPC spaces, i.e., globally nonpositively curved spaces in the sense of Busemann.
While there are many textbooks about CAT($0$) spaces, globally nonpositively curved spaces in the sense of Alexandrov (e.g., \cite{AKP:cat}, \cite{AKP:found}, \cite{BH}, \cite{BBI}), the references for BNPC spaces are quite limited.
We refer the reader to \cite{Pa} and \cite{Jo}.

\begin{dfn}\label{dfn:bnpc}
A complete geodesic space $X$ is said to be a \textit{BNPC space} if for any pair of linearly reparametrized shortest paths $\gamma,\eta:[0,1]\to X$, the restriction of the distance function
\[t\mapsto d(\gamma(t),\eta(t))\]
is convex on $[0,1]$.
Such spaces are also called  \textit{Busemann spaces} or \textit{convex spaces}.
\end{dfn}

The following remark points out simple yet crucial consequences of the above definition.

\begin{rem}\label{rem:bnpc}
Let $X$ be a BNPC space and $p\in X$.
\begin{enumerate}
\item Let $\gamma,\eta:[0,1]\to X$ be linearly reparametrized shortest paths emanating from $p$.
Then for any $0\le t\le 1$,
\begin{equation}\label{eq:mono}
|\gamma(t)\eta(t)|\le t|\gamma(1)\eta(1)|.
\end{equation}
We call it the \textit{Busemann monotonicity} at $p$.
\item Let $\zeta:[0,1]\to X$ be an arbitrary linearly reparametrized shortest path.
Then for any $0\le t\le 1$,
\begin{equation}\label{eq:conv}
|p\zeta(t)|\le(1-t)|p\zeta(0)|+t|p\zeta(1)|.
\end{equation}
We call it the \textit{Busemann convexity} for $p$.
\end{enumerate}
\end{rem}

Indeed, Condition (\ref{eq:mono}) is equivalent to the definition of a BNPC space, thanks to the triangle inequality (\cite[Proposition 8.1.2]{Pa}).
The above two properties are, in a sense, closely related to the two angles we will introduce in Sections \ref{sec:bran} and \ref{sec:orth}, respectively.

\begin{ex}\label{ex:bnpc}
The following are examples of BNPC spaces.
\begin{itemize}
\item Strictly convex Banach spaces (\cite[Proposition 8.1.6]{Pa}).
\item Complete, simply connected Finsler manifolds with reversible Berwald metrics of nonpositive flag curvature (\cite{KVK}, \cite{KK}, \cite{IL}, \cite[Section 8.4]{Oh}).
\item A ``suitable'' gluing or Gromov--Hausdorff limit of BNPC spaces (\cite{An:const}).
\item CAT($0$) cube complexes equipped with piecewise $L^p$-metrics, where $1<p<\infty$ (\cite[Theorem D]{HHP}).
\end{itemize}
\end{ex}

Note that if $(X,d)$ is BNPC, then its rescaling $(X,\lambda d)$ remains BNPC for any $\lambda>0$.
The following are basic properties of BNPC spaces that immediately follow from the above definition (see \cite[Chapter 8]{Pa}).

\begin{itemize}
\item A geodesic between any two points is unique; in particular, every geodesic is actually a shortest path.
\item Every open or closed ball $B$ is \textit{convex}, i.e., any shortest path between two points in $B$ is contained in $B$.
\item Every convex subset $C$ is contractible to an arbitrary point in $C$ along the unique shortest paths: we call this the \textit{geodesic contraction}.
\end{itemize}

In particular, any BNPC space is contractible and therefore simply connected.
Conversely, the Cartan--Hadamard-type globalization theorem holds: any complete, simply connected, locally BNPC space is (globally) BNPC (\cite{AB}, \cite{Ba}).
We say that a metric space is \textit{locally BNPC} if every point has a neighborhood that is BNPC with respect to the ambient metric.
In particular, the universal cover of any complete, locally BNPC space is BNPC.

\subsection{BNPC vs.\ CAT(0)}\label{sec:cat}

We now compare BNPC spaces with CAT($0$) spaces.
Recall that the Busemann monotonicity \eqref{eq:mono} is an equivalent definition of a BNPC space and we interpret it geometrically as follows.

Let $\triangle pxy$ be a geodesic triangle in a BNPC space, and let $\triangle\tilde p\tilde x\tilde y$ be its \textit{comparison triangle}, that is, a geodesic triangle in the Euclidean plane with the same side lengths.
For $0\le t\le 1$, let $tx$ and $ty$ denote the points on the shortest paths $px$ and $py$ such that $|p,tx|=t|px|$ and $|p,ty|=t|py|$, respectively.
Similarly, let $t\tilde x$ and $t\tilde y$ be the corresponding points on the shortest paths $\tilde p\tilde x$ and $\tilde p\tilde y$, respectively.
Then the Busemann monotonicity \eqref{eq:mono} says
\[|tx,ty|\le|t\tilde x,t\tilde y|.\]

By contrast, the CAT($0$) condition says
\[|tx,sy|\le|t\tilde x,s\tilde y|\]
for any $0\le t\le 1$ and $0\le s\le 1$.
Hence every CAT($0$) space is a BNPC space.
This distinction reflects the Riemannian vs. Finsler nature of CAT($0$) vs. BNPC, as illustrated in the following examples.

\begin{ex}\label{ex:cat}
Compare the following with Example \ref{ex:bnpc}.
\begin{itemize}
\item A BNPC Banach space is CAT($0$) if and only if it is Hilbertian.
\item A BNPC Finsler manifold is CAT($0$) if and only if it is Riemannian.
\end{itemize}
See \cite[Sections 1.2, 8.3]{Oh} for the proofs.
\end{ex}

In general, a BNPC space is CAT($0$) if and only if the angle is well-defined for every hinge (\cite[Exercise 9.81]{AKP:found}).
Foertsch--Lytchak--Schroeder \cite{FLS} characterized CAT($0$) spaces among BNPC spaces in terms of the Ptolemy inequality.
Another characterization was given by the second author \cite{Gu:ab} in terms of average angles.

\begin{rem}\label{rem:curv}
From the geometric interpretation of the BNPC condition, one can define a \textit{Busemann space} with curvature $\le\kappa$ for any $\kappa\in\mathbb R$, by replacing the Euclidean comparison triangles with those in the model plane of constant curvature $\kappa$ (cf.\ \cite{Gu:ab}).
However, we will not pursue that generalization here; see Problem \ref{prob:curv}.
\end{rem}

\begin{rem}\label{rem:cd}
The above difference between BNPC spaces and CAT($0$) spaces in a sense resembles the difference between \textit{CD spaces} and \textit{RCD spaces}, which are synthetic generalizations of manifolds with lower Ricci curvature bounds.
Although the CD condition characterizes lower Ricci curvature bounds for Riemannian manifolds, it also makes sense for Finsler manifolds.
To exclude Finsler-type spaces, one needs the so-called \textit{infinitesimally Hilbertian} condition, which led to the notion of RCD space.
See \cite[Chapter 18]{Oh} and the references therein.
\end{rem}

\subsection{GNPC spaces}\label{sec:gnpc}

Following Lytchak--Nagano \cite{LN:geo}, \cite{LN:top}, we introduce

\begin{dfn}\label{dfn:gnpc}
A \textit{GNPC space} is a separable, locally compact, locally geodesically complete, locally BNPC space.
\end{dfn}

Note that BNPC is a global condition, whereas GNPC is a local condition, except for the separability (which we do not use in this paper).
Lytchak--Nagano's GCBA space with nonpositive curvature, defined by replacing ``BNPC" with ``CAT($0$)" in Definition \ref{dfn:gnpc}, is our GNPC space.
Any open subset of a GNPC space remains GNPC, and any locally BNPC homology manifold (without boundary) is GNPC (Proposition \ref{prop:geo}; see the next subsection for homology manifolds).

Following Lytchak--Nagano \cite{LN:geo}, \cite{LN:top}, we also introduce

\begin{dfn}\label{dfn:tiny}
Let $X$ be a GNPC space.
An open metric ball $B$ of radius $r$ in $X$ is called a \textit{tiny ball} if the concentric closed ball of radius $10r$, denoted by $10\bar B$, is a compact BNPC space.
\end{dfn}

For each $p\in X$, any sufficiently small ball around $p$ is a tiny ball.
We allow  arbitrarily large radius to deal with global problems like Theorem \ref{thm:four}, but after rescaling, the actual radius becomes irrelevant.

Within a tiny ball $B$ (and its enlargement $10\bar B$), all properties of a BNPC space from Section \ref{sec:bnpc} hold.
Moreover, any shortest path in $B$ is extendable to length $9r$ in $10\bar B$, where $r$ is the radius of $B$.
This follows from the same argument that local geodesic completeness implies global one in a complete space (see Section \ref{sec:geo}).

\subsection{Homology manifolds}\label{sec:hom}

We say that $X$ is a \textit{homology $n$-manifold with boundary} if it is a locally compact, separable, metrizable space of finite topological dimension such that the local (integral) homology group at each point is isomorphic to that of the closed $n$-cell $D^n$ at some point, i.e.,
\[H_\ast(X,X\setminus\{x\})\cong H_\ast(D^n,D^n\setminus\{y\})\]
for any $x\in X$ and some $y=y(x)\in D^n$.

The \textit{boundary} of a homology $n$-manifold is the set of points where the $n$-th local homology vanishes.
The following theorem is due to Mitchell \cite{Mit}.

\begin{thm}\label{thm:mit}
The boundary of a homology $n$-manifold, if nonempty, is a homology $(n-1)$-manifold without boundary.
\end{thm}

A \textit{homology manifold} simply refers to a homology manifold without boundary (the same convention applies to topological manifolds).
Any homology $n$-manifold has topological dimension $n$.

The following theorem, due to Moore (\cite[Chapter IX]{Wi}), is the starting point of the inductive proofs of the main theorems.

\begin{thm}\label{thm:moo}
Any homology manifold of dimension $\le2$ is a topological manifold.
\end{thm}

A metric space $X$ is called an \textit{absolute neighborhood retract}, abbreviated as \textit{ANR}, if for any closed subset $A$ of a metric space $Z$, every continuous map $A\to X$ has a continuous extension $U\to X$ defined on some neighborhood $U$ of $A$ in $Z$.
For our purposes, the following characterization is more useful: a finite-dimensional, locally compact metric space is an ANR if and only if it is locally contractible.
See \cite[Section 14]{Da} for ANRs.
Every ANR homology manifold with boundary satisfies the Poincar\'e--Lefschetz duality (note that our homology manifolds are of finite dimension).

\subsection{Hurewicz fibrations}\label{sec:fib}

Here we assume all maps are continuous.
A surjective map between topological spaces is a \textit{Hurewicz fibration} if it satisfies the homotopy lifting property with respect to all topological spaces.

We recall the following condition from \cite[Definition 4.3]{LN:top}.

\begin{dfn}\label{dfn:fib}
We say that a map $f:X\to Y$ between topological spaces has \textit{locally uniformly contractible fibers} if the following holds:
for any $x\in X$ and every neighborhood $U$ of $x$, there exists a neighborhood $V\subset U$ of $x$ such that for any fiber $\Pi$ of $f$ intersecting $V$, the intersection $\Pi\cap V$ is contractible in $\Pi\cap U$. 
\end{dfn}

The following theorems of Ungar \cite[Theorems 1, 2]{Un} provide sufficient conditions for a map to be a Hurewicz fibration (see also \cite[Theorems 4.5, 4.6]{LN:top}).

\begin{thm}\label{thm:ung1}
Suppose $X$ and $Y$ are finite-dimensional compact metric spaces and $Y$ is an ANR.
Let $f:X\to Y$ be an open surjective map with locally uniformly contractible fibers.
Then $f$ is a Hurewicz fibration (and $X$ is an ANR).
\end{thm}

\begin{thm}\label{thm:ung2}
Suppose $X$ and $Y$ are finite-dimensional locally compact metric spaces.
Let $f:X\to Y$ be an open surjective map with locally uniformly contractible fibers.
If every fiber of $f$ is contractible, then $f$ is a Hurewicz fibration.
\end{thm}

The following theorem of Raymond \cite[Theorem (1)]{Ra} describes the relationship between homology manifolds and Hurewicz fibrations (see also \cite[Theorem 4.10]{LN:top}).

\begin{thm}\label{thm:ray}
Let $f:X\to Y$ be a Hurewicz fibration, where $X$ is a homology $n$-manifold and $Y$ is a connected, locally contractible space.
Then there exists some $k\le n$ such that any fiber of $f$ is a homology $(n-k)$-manifold and $Y$ is a homology $k$-manifold.
\end{thm}

\subsection{Recognition theorems}\label{sec:rec}

In the proofs of the main theorems, the key is to ``recognize'' topological manifolds among homology manifolds (cf.\ Theorem \ref{thm:moo}) and fiber bundles among Hurewicz fibrations.
Our proofs rely heavily on such \textit{recognition theorems} listed in \cite[Sections 4.3, 6.1]{LN:top}.
However, these recognition theorems are used in the topological part of the proofs, match exactly the CAT($0$) case and thus omitted here.
For this reason, although the use of these theorems is essential, we avoid including their statements and numerous related references here.
See \cite[Sections 4.3, 6.1]{LN:top} and the references therein.

\section{Facts}\label{sec:fac}

In this section, we collect several facts known for CAT($0$) spaces whose proofs also work for BNPC spaces.
For completeness (and as a warm-up), we include the proofs whenever possible.
The reader familiar with Lytchak--Nagano's theory of GCBA spaces \cite{LN:geo}, \cite{LN:top} and Thurston's result \cite{Th:cat} may skip this section.

The first proposition is about the finite-dimensionality of GNPC spaces (\cite[Proposition 5.1, Corollary 5.4]{LN:geo}).
A metric space $X$ is called \textit{doubling}, or more specifically \textit{$L$-doubling}, if there exists a doubling constant $L>0$ such that every $r$-ball in $X$ can be covered by at most $L$ balls of radius $r/2$.
Clearly, if any $r$-ball in $X$ contains at most $L$ points with pairwise distance $\ge r/2$, then $X$ is $L$-doubling (we say that such points with pairwise distance $\ge r/2$ are \textit{$r/2$-separated}).
A metric space is \textit{locally doubling} if each point has a neighborhood that is doubling, where the doubling constant may depend on each point.

\begin{prop}\label{prop:dim}
Any GNPC space $X$ is locally doubling.
More precisely, any tiny ball in $X$ is doubling.
In particular, $X$ has locally finite Hausdorff and topological dimensions, and is an ANR.
\end{prop}

\begin{proof}
The proof of the first statement is exactly the same as \cite[Proposition 5.1]{LN:geo}.
Let $B$ be a tiny ball centered at $p\in X$.
By rescaling, we may assume $B=B(p,1)$ (note that the doubling constant is a scaling invariant).

Since $\bar B(p,10)$ is compact, the maximal number of $1$-separated points in $\bar B(p,10)$ is bounded, say by $L>0$.
We show that $B(p,1)$ is $L$-doubling.
Let $q\in B(p,1)$ and $0<r<2$ (the case $r\ge2$ is trivial).
Consider $r/2$-separated points $\{x_\alpha\}_{\alpha=1}^N$ in $B(q,r)$.
We denote by $(2/r)x_\alpha$ a point on an extension of the shortest path $qx_\alpha$ at distance $(2/r)|qx_\alpha|$ from $q$, which is contained in $\bar B(p,10)$.
By the Busemann monotonicity \eqref{eq:mono}, $\{(2/r)x_\alpha\}_{\alpha=1}^N$ is $1$-separated.
Therefore $N\le L$, as desired.

It is easy to see that the Hausdorff dimension is not greater than the doubling dimension, i.e., $\log_2L$.
It is also known that the topological dimension is not greater than the Hausdorff dimension.
Since $X$ is locally compact, locally finite-dimensional, and locally contractible, it is an ANR.
\end{proof}

\begin{rem}\label{rem:anr}
The fact that $X$ is an ANR holds even without assuming local compactness or geodesic completeness; that is, any locally BNPC space is an ANR.
Indeed, the proof of \cite[Theorem 3.2]{Kram} for locally CAT($\kappa$) spaces applies to locally BNPC spaces.
See also the references listed in \cite[P.\ 353, Footnote 4]{Kram}.
\end{rem}

\begin{rem}\label{rem:glodim}
For $X$ to be globally finite-dimensional, we have to assume that the local doubling constants are uniformly bounded above.
This is the same as in the GCBA case.
A counterexample is a real line with Euclidean $n$-space attached to each $n$ on the line.
\end{rem}

\begin{rem}\label{rem:emb}
It is known that if in addition $X$ is (globally) BNPC and doubling, then it admits a bi-Lipschitz embedding into finite-dimensional Euclidean space: see \cite{LP} and \cite{Zo}.
Compare with the GCBA case \cite[Proposition 5.3]{LN:geo}.
\end{rem}

\begin{rem}\label{rem:dim}
For GCBA spaces, the Hausdorff dimension is equal to the topological dimension, both locally and globally (\cite[Theorem 1.1]{LN:geo}).
For now the authors do not know whether the same holds true for GNPC spaces (Problem \ref{prob:dim}).
Note that this claim is false without geodesic completeness.
A counterexample is a tree branching very fast, described in \cite[P.\ 411, Footnote 3]{Kl}.
See also \cite[Theorem D]{Kl} for a related result under cocompact group actions.
\end{rem}

The second proposition provides a sufficient condition for a BNPC space to be geodesically complete (\cite[Lemma 2.2]{LN:top}, \cite[Proposition 2.1]{Th:g}; see also \cite[Theorem 1.5(1)]{LS} for a slightly stronger statement).
In particular, any (locally) BNPC space in the main theorems is (locally) geodesically complete.

\begin{prop}\label{prop:geo}
Let $X$ be a locally BNPC space.
Suppose for any point $p\in X$, there exists a tiny ball $B$ centered at $p$ such that $B\setminus\{p\}$ is not contractible.
Then $X$ is locally geodesically complete.

In particular, every locally BNPC homology manifold (without boundary) is locally geodesically complete, and hence a GNPC space.
\end{prop}

\begin{proof}
We show that if some shortest path $\gamma$ with endpoint $p$ cannot be extended beyond $p$, then any punctured tiny ball around $p$ is contractible.

Let $B\subset X$ be an arbitrary tiny ball centered at $p$.
To show that $B\setminus\{p\}$ is contractible, we fix a point $q\in\gamma\cap B\setminus\{p\}$.
Recall that $B$ is convex and therefore contractible to $q$ by the geodesic contraction (see Section \ref{sec:bnpc}).
This contraction can be restricted to $B\setminus\{p\}$.
Indeed, the non-extendability of $\gamma$ and the uniqueness of shortest paths imply that any shortest path between $q$ and $x\in B\setminus\{p\}$ does not pass through $p$.
This completes the proof of the first statement.

The second statement immediately follows from the first.
Suppose $X$ is a locally BNPC homology $n$-manifold and let $B$ be a tiny ball around $p\in X$.
By the excision theorem, $H_n(B,B\setminus\{p\})$ is nontrivial.
Since $B$ is contractible, the Mayer--Vietoris exact sequence shows that $H_n(B,B\setminus\{p\})$ is equal to $H_{n-1}(B\setminus\{p\})$. 
Therefore a punctured tiny ball $B\setminus\{p\}$ has nontrivial homology and is noncontractible, as desired.
Note that our definition of a homology manifold includes the other conditions of a GNPC space, i.e., separability and local compactness (see Section \ref{sec:hom}).
\end{proof}

\begin{rem}\label{rem:dunce}
The converse of the above statement is not true; that is, there exists a GNPC (actually GCBA) space for which every small punctured ball at some point is contractible.
Indeed, any finite simplicial complex admits a CAT($1$) metric (\cite{Be:bor}).
Therefore, if one takes a compact, contractible simplicial complex $K$ without boundary, such as the dunce hat or Bing's house, one can equip it with a geodesically complete CAT($1$) metric.
Then the Euclidean cone over $K$ gives a counterexample, as every punctured ball around the vertex is contractible.
\end{rem}

The third proposition addresses the local topology of BNPC homology manifolds (\cite[Lemma 3.2]{LN:top}, \cite[Proposition 2.7]{Th:cat}).

\begin{prop}\label{prop:hom}
Let $X$ be a locally BNPC homology $n$-manifold (and hence GNPC from Proposition \ref{prop:geo}).
Then every closed tiny ball $\bar B$ is an ANR homology $n$-manifold with boundary. Furthermore,
the boundary as a homology manifold is the corresponding metric sphere $\partial B$ and has the same homology as the $(n-1)$-sphere.
\end{prop}

\begin{proof}
To prove that $\bar B$ is a homology manifold with boundary $\partial B$, it suffices to show that the local homology vanishes at every point $q\in \partial B$.
That follows immediately from the Mayer--Vietoris exact sequence and the contractibility of $\bar B$ and $\bar B\setminus\{q\}$.
Here both contractions are given by the geodesic contraction to $p$.
Furthermore, since $\bar B$ is locally contractible, it is an ANR (and hence an AR, i.e., an absolute retract; see \cite[Section 14]{De}).

By Theorem \ref{thm:mit}, the boundary $\partial B$ is a homology $(n-1)$-manifold.
Since $\bar B$ is contractible, the homology of $\partial B$ can be calculated by the Mayer--Vietoris exact sequence and the Poincar\'e--Lefschetz duality.
We leave the details to the reader.
\end{proof}

The fourth theorem is the $3$-dimensional version of Theorems \ref{thm:sing} and \ref{thm:four}.
The first half is due to Thurston \cite[Theorem 3.3]{Th:cat} and its second half follows from Rolfsen \cite[Theorem 1]{Ro} and Brown \cite[Theorem]{Br}.
Later we will provide an alternative proof of the first half (Corollary \ref{cor:three}).

\begin{thm}\label{thm:three}
Every locally BNPC homology $3$-manifold is a topological manifold.
Moreover, if it is (globally) BNPC, then it is homeomorphic to Euclidean space.
\end{thm}

\begin{proof}
The proof of the first half is the same as the reverse implication of \cite[Theorem 3.3]{Th:cat}, where Thurston assumed the CAT($\kappa$) condition instead of the BNPC condition.
However, the only geometric property he used is the geodesic contraction, which is available in our BNPC setting, and the other ingredients are the $2$-manifold recognition (Theorem \ref{thm:moo}) and the decomposition theory.
Therefore the original proof can be applied to BNPC spaces.
We leave the details to the reader.

The second half is an immediate consequence of the results of Rolfsen \cite[Theorem 1]{Ro} and Brown \cite[Theorem]{Br}.
Note that our BNPC metric is SC in the sense of Rolfsen.
(The resolution of the Poincar\'e conjecture, together with the generalized Schoenflies theorem, also implies that any compact contractible $3$-manifold with boundary is a $3$-cell.)
\end{proof}

The last theorem is the main result of Thurston \cite[Theorem 1.6(1)]{Th:cat}.
This plays an essential role in the proof of Theorem \ref{thm:four}, as in the CAT($0$) case.
Later we will show that the tameness assumption is always satisfied (Theorem \ref{thm:4sph}).

\begin{thm}\label{thm:thur}
Let $X$ be a BNPC topological $4$-manifold with a topologically tame point $p$, i.e., every metric sphere around $p$ is a closed topological manifold.
Then $X$ is homeomorphic to Euclidean space.
\end{thm}

\begin{proof}
As in the 3-dimensional case, the only geometric property Thurston used is the geodesic contraction.
Therefore the original proof in \cite{Th:cat} carries over directly to BNPC spaces.
We leave the details to the reader.
\end{proof}

\begin{rem}\label{rem:cont}
In dimension $\ge 3$, many contractible open manifolds are not homeomorphic to Euclidean space.
See for instance \cite{Gui}.
See also \cite[Sections 1.2, 1.3]{LNS} and references therein for geometric structures of such spaces.
\end{rem}

\section{Key lemmas}\label{sec:key}

In this section, we prove key lemmas needed to develop strainers for GNPC spaces in the next section.
They involve two different notions of angle, so we divide this section into two subsections: in Section \ref{sec:bran}, we introduce the angle related to the branching of geodesics, and in Section \ref{sec:orth}, we introduce the angle related to the orthogonality of geodesics.
These subsections are independent; neither part relies on results from the other.

Let $\triangle pqr$ be a geodesic triangle in a tiny ball $B$ of a GNPC space $X$.
The two angles at $p$ we will introduce are denoted by
\[\angle_p(q,r),\quad\angle qpr.\]
These two angles are in a sense related to the Busemann monotonicity \eqref{eq:mono} and Busemann convexity \eqref{eq:conv}, respectively.
In general they do not coincide unless $X$ is a GCBA space.
This asymmetry of angles is the central problem we must deal with.

Before delving into the details, we establish key notations that will be used extensively in this and the next two sections.

\begin{nota}\label{nota:del}
We use the following symbols to express errors (as in \cite{BGP}):
\begin{itemize}
\item $\delta\ll1$ denotes a small positive number;
\item $\varkappa(\delta)$ denotes various positive functions depending only on $\delta$ such that $\varkappa(\delta)\to0$ as $\delta\to 0$ (usually nondecreasing and $\varkappa(\delta)\gg\delta$).
\end{itemize}
To be more precise about $\varkappa$, even if different functions appear in close places, we will use the same symbol.
That is, after claiming $c=\varkappa(\delta)$ in the first line, we may claim $2c=\varkappa(\delta)$ in the next line, where the second $\varkappa(\delta)$, strictly speaking, should be twice the first one.
However, this eliminates redundant calculations and makes the arguments clearer.
\end{nota}

\begin{rem}\label{rem:del}
If necessary, all the $\varkappa$ appearing below can be computed explicitly by following the arguments in detail.
Later, we will specify an upper bound for $\delta$ in the next two sections, where we introduce the notion of a \textit{$(k,\delta)$-strainer} (and an \textit{extended $(k,\delta)$-strainer}). This ensures that all the statements regarding (extended) $(k,\delta)$-strainers hold whenever $0<\delta<\delta_k$, where $\delta_k$ is a constant depending only on $k$.
\end{rem}

\subsection{Almost nonbranching}\label{sec:bran}

In this subsection, we introduce the ``angle of fixed scale,'' which controls the branching of geodesics.
The goal of this subsection is to prove Proposition \ref{prop:bran}, the almost nonbranching of geodesics.

Throughout this subsection:
\begin{itemize}
\item $X$ denotes a GNPC space and $B$ a tiny ball.
\item $p\in B$ denotes an arbitrary fixed point.
\item $x,y\in B\setminus\{p\}$ denote variable points (near $p$).
\item $\tilde\angle_p(x,y)$ denotes the comparison angle at $p$ of the Euclidean triangle with side-lengths $|px|$, $|py|$, and $|xy|$.
\item For $\lambda\ge0$, $\lambda x$ denotes a point on an extension of the shortest path from $p$ to $x$ such that $|p,\lambda x|=\lambda|px|$.
\end{itemize}
In other words, $\lambda x$ is the reparametrization of an extension of the shortest path $px$ by the constant speed $|px|$.
This notation allows us to regard $p$ as if it were the origin of a vector space.
Note that, however, $\lambda x$ is not necessarily unique for $\lambda>1$ because of the branching of geodesics.

Let $p,x,y\in B$ be as above.
The Busemann monotonicity \eqref{eq:mono} literally says that the comparison angle $\tilde\angle_p(tx,ty)$ is nondecreasing in $t\ge 0$ (see Section \ref{sec:cat}).
Therefore we can define the following angle.

\begin{dfn}\label{dfn:bran}
In the situation above, we call
\[\angle_p(x,y):=\lim_{t\to0}\tilde\angle_p(tx,ty)\]
the \textit{angle} at $p$ between $x$ and $y$.
\end{dfn}

\begin{rem}\label{rem:bran}
Unlike the GCBA case, this angle generally depends on (the ratio of) $|px|$ and $|py|$, or in other words, the parametrizations of the shortest paths $px$ and $py$.
\end{rem}

This angle controls the branching of geodesics emanating from $p$.
In the GCBA case, the independence of this angle on parametrizations implies that the tangent cone (i.e., the Gromov--Hausdorff blow-up) at $p$ is a Euclidean cone  (\cite[Corollary 5.7]{LN:geo}).
Since the geodesics from the vertex of the Euclidean cone are nonbranching, one sees that the geodesics from $p$ are almost nonbranching near $p$ (\cite[Proposition 7.3]{LN:geo}).
In the GNPC case, we must restrict this argument to some almost fixed scale.
The following is a uniform version of the convergence in Definition \ref{dfn:bran} with respect to some almost fixed scale.

\begin{lem}\label{lem:bran}
Let $X$ be a GNPC space and $p\in X$.
For any $1\gg\delta>0$, there exists $r_0=r_0(p,\delta)>0$ satisfying the following.
Let $x,y\in B(p,r_0)\setminus\{p\}$ be such that
\[\left|\frac{|py|}{|px|}-1\right|<\delta.\]
Then for any $0<\lambda<2$, we have
\begin{equation}\label{eq:bran}
|\tilde\angle_p(\lambda x,\lambda y)-\angle_p(x,y)|<\varkappa(\delta),
\end{equation}
where $\varkappa(\delta)\to0$ as $\delta\to0$ (see Notation \ref{nota:del}).
\end{lem}

\begin{proof}
This follows from local compactness and the Busemann monotonicity \eqref{eq:mono}.
The proof is almost the same as in the GCBA case, \cite[Lemma 5.6]{LN:geo}, except that we have to restrict our argument to some almost fixed scale as in the assumption.
In what follows, we denote by $\varkappa(\delta)$ various positive functions such that $\varkappa(\delta)\to0$ as $\delta\to0$, as explained in Notation \ref{nota:del}.

Let $B\subset X$ be a tiny ball centered at $p$.
By rescaling, we may assume $B=B(p,1)$ (note that everything we consider here, such as the Euclidean comparison angle, is a scaling invariant).
Choose a maximal $\delta$-separated set $A\subset\partial B(p,1)$.
By compactness, $A$ is finite.
As in Definition \ref{dfn:bran}, we can define the angle
\[\angle_p(z,w)=\lim_{t\to0}\tilde\angle_p(tz,tw)\]
for any $z,w\in A$.
Since $A$ is finite, there exists $0<r_0<1/2$ such that for any $z,w\in A$ and any $0<t<2r_0$, we have
\begin{equation}\label{eq:bran'}
|\tilde\angle_p(tz,tw)-\angle_p(z,w)|<\delta.
\end{equation}
In other words, we have chosen $r_0$ so that the desired inequality \eqref{eq:bran} holds on the finitely many shortest paths to $A$.

By the Busemann monotonicity \eqref{eq:mono}, these finitely many shortest paths are ``$\delta$-dense'' among all the shortest paths from $p$ to $\partial B(p,1)$ with natural parametrization.
This fact, together with the assumption \eqref{eq:bran'}, implies the desired inequality \eqref{eq:bran}.
We will see the details below.

Let $x,y\in B(p,r_0)\setminus\{p\}$ be as in the assumption and set $r:=|px|$.
Let $r^{-1}x\in\partial B(p,1)$ be a point on an (arbitrary) extension of the shortest path $px$.
Since $A$ is a maximal $\delta$-separated set of $\partial B(p,1)$, there exists $x'\in A$ that is $\delta$-close to $r^{-1}x$.
By the Busemann monotonicity \eqref{eq:mono}, we have
\begin{equation}\label{eq:branx}
|\lambda x,\lambda rx'|\le\delta\lambda r
\end{equation}
for any $0<\lambda<2$.
See Figure \ref{fig:bran}.

\begin{figure}[ht]
\centering
\begin{tikzpicture}
\coordinate[label=below:$p$](p)at(0,0);
\coordinate[label=above left:$x'$](x')at(130:2.5);
\coordinate[label=above right:$y'$](y')at(50:2.5);
\coordinate[label=above left:$r^{-1}x$](rx)at(145:2.5);
\coordinate[label=above right:$r^{-1}y$](ry)at(35:2.75);
\coordinate[label=below:$x$](x)at(145:1);
\coordinate[label=below:$y$](y)at(35:1.1);

\draw(p)circle[radius=1.25];
\draw(p)circle[radius=2.5];

\draw(p)to(x');
\draw(p)to(y');
\draw(p)to(rx);
\draw(p)to(ry);

\node[below right]at(-45:1.25){$r_0$};
\node[below right]at(-45:2.5){$1$};

\fill(p)circle(1.5pt);
\fill(x')circle(1.5pt);
\fill(y')circle(1.5pt);
\fill(rx)circle(1.5pt);
\fill(ry)circle(1.5pt);
\fill(x)circle(1.5pt);
\fill(y)circle(1.5pt);
\end{tikzpicture}
\caption{}\label{fig:bran}
\end{figure}

A similar argument applies to $y$ because of the assumption $||py|/|px|-1|<\delta$.
Indeed, $r^{-1}y$ is not necessarily on $\partial B(p,1)$, but this assumption literally says that $||p,r^{-1}y|-1|<\delta$.
Hence there exists $y'\in A$ that is $2\delta$-close to $r^{-1}y$.
As before, by the Busemann monotonicity \eqref{eq:mono}, we have
\begin{equation}\label{eq:brany}
|\lambda y,\lambda ry'|\le2\delta\lambda r
\end{equation}
for any $0<\lambda<2$.
See again Figure \ref{fig:bran}.

The above two inequalities \eqref{eq:branx} and \eqref{eq:brany}, together with an easy calculation in Euclidean plane, imply that
\[|\tilde\angle_p(\lambda x,\lambda y)-\tilde\angle_p(\lambda rx',\lambda ry')|<\varkappa(\delta).\]
The latter angle $\tilde\angle_p(\lambda rx',\lambda ry')$ is almost constant up to the error $\delta$ for $0<\lambda<2$, as already assumed in \eqref{eq:bran'}.
This completes the proof.
\end{proof}

The above lemma immediately implies the following property, which we call the \textit{almost nonbranching} of geodesics.

\begin{prop}\label{prop:bran}
Let $X$ be a GNPC space and $p\in X$.
For $1\gg\delta>0$, let $r_0=r_0(p,\delta)>0$ be the radius from Lemma \ref{lem:bran}.
Suppose $x\in\partial B(p,r)$ and $y\in B(x,\delta r)$, where $0<r<r_0$.
Then we have
\[|2x,2y|<\varkappa(\delta)r.\]
\end{prop}

\begin{proof}
The assumption $y\in B(x,\delta r)$, where $r=|px|$, implies that $||py|/|px|-1|<\delta$, which satisfies the assumption of Lemma \ref{lem:bran}.
Furthermore, since $|xy|<\delta r$, we have $\tilde\angle_p(x,y)<\varkappa(\delta)$.
By Lemma \ref{lem:bran}, we get $\tilde\angle_p(2x,2y)<\varkappa(\delta)$.
Since both $2x$ and $2y$ are at distance almost $2r$ from $p$, or more precisely,
\[|p,2x|=2r,\quad||p,2y|-2r|<2\delta r,\]
we obtain $|2x,2y|<\varkappa(\delta)r$.
\end{proof}

\subsection{Almost orthogonality}\label{sec:orth}

In this subsection, we introduce the ``angle viewed from a fixed point,'' controlling the orthogonality of geodesics.
The goal of this subsection is to prove Proposition \ref{prop:orth}, the almost orthogonality of geodesics.

Throughout this subsection, we will use the following notation, some of which are different from those in the previous subsection:
\begin{itemize}
\item $X$ denotes a GNPC space and $B$ a tiny ball.
\item $p\in B$ and $q\in B\setminus\{p\}$ denote arbitrary fixed points.
\item $x\in B\setminus\{q\}$ denotes a variable point (near $q$).
\item $\tilde\angle pqx$ denotes the comparison angle at $q$ of the Euclidean triangle with side-lengths $|pq|$, $|px|$, and $|qx|$.
\item $x(t)$ denotes the shortest path from $q$ to $x$ and its extension beyond $x$, parametrized by arclength.
\end{itemize}

\begin{dfn}\label{dfn:orth}
In the situation above, we call
\[\angle pqx:=\lim_{t\to0}\tilde\angle pqx(t)\]
the \textit{angle} viewed from $p$ at $q$ toward $x$ (see below on the existence).
\end{dfn}

Unlike the GCBA case, the comparison angle $\tilde\angle pqx(t)$ is not necessarily nondecreasing in $t$.
Therefore the existence of the above limit is a little bit nontrivial.
However, the ``almost monotonicity'' coming from the Busemann convexity \eqref{eq:conv} guarantees the existence of the limit.

\begin{lem}\label{lem:comp}
The limit $\angle pqx=\lim_{t\to0}\tilde\angle pqx(t)$ exists and satisfies the following almost comparison property:
\begin{equation}\label{eq:comp}
\angle pqx<\tilde\angle pqx(t)+\delta(t),
\end{equation}
where $\delta(t)$ is a positive function defined for $0<t<t_0$ such that $\delta(t)\to0$ as $t\to 0$.
Here $\delta(t)$ and $t_0$ depend only on $|pq|$.
\end{lem}

\begin{proof}
By the definition of Euclidean comparison angle, we have
\begin{equation}\label{eq:cos}
\begin{aligned}
\cos\tilde\angle pqx(t)&=\frac{|pq|^2+t^2-|px(t)|^2}{2|pq|t}\\
&=\frac{|pq|-|px(t)|}t+O\left(\frac t{|pq|}\right),
\end{aligned}
\end{equation}
where $O$ denotes the Landau's symbol.

By the Busemann convexity \eqref{eq:conv}, the above main part
\[\frac{|pq|-|px(t)|}t\]
is nonincreasing in $t$.
Therefore the limit exists and satisfies the desired almost comparison property (the error function $\delta(t)$ comes from $O(t/|pq|)$).
\end{proof}

\begin{rem}\label{rem:comp}
It should be emphasized that the almost comparison \eqref{eq:comp} holds for any $q'$ sufficiently close to $q$ (compared to $|pq|$) instead of the fixed $q$, possibly with the same error function $\delta(t)$ defined on the same interval $(0,t_0)$.
More precisely, for any such $q'\in B$ and $x'\in B\setminus\{q'\}$, we have
\[\angle pq'x'<\tilde\angle pq'x'(t)+\delta(t)\]
for $0<t<t_0$, where $x'(t)$ denotes the unit-speed shortest path from $q'$ to $x'$.
This is because the error term of the cosine formula \eqref{eq:cos} depends only on $|pq|$.
This observation plays an important role in the proof of Lemma \ref{lem:orth2} below.
On the other hand, the speed of the convergence $\tilde\angle pqx(t)\to\angle pqx$ depends heavily on $q$ (cf.\ Lemma \ref{lem:orth1}).
\end{rem}

\begin{rem}\label{rem:orth}
As can be seen from the above proof, the angle $\angle pqx$ is essentially the directional derivative of the distance function from $p$, i.e.,
\[\cos\angle pqx=-\left.\frac d{dt}\right|_{t=0}d(p,x(t)).\]
Indeed, it might be better to consider the differential quotient $t^{-1}(|px(t)|-|pq|)$ rather than the comparison angle $\tilde\angle pqx(t)$, since the former is really monotone whereas the latter is almost monotone.
Although we stick to the latter angle notation for inspiring geometric intuition and for comparing with the GCBA case, the above viewpoint plays an important role in the definition of a strainer in the next section.
\end{rem}

The above angle is used to control the orthogonality of geodesics around $q$.
To see this, we first prove a uniform version of the convergence in Definition \ref{dfn:orth}.
The proof is similar to that of Lemma \ref{lem:bran}.

\begin{lem}\label{lem:orth1}
Let $X$ be a GNPC space, $B$ a tiny ball, and $p\neq q\in B$.
For any $1\gg\delta>0$, there exists $r_0=r_0(p,q,\delta)>0$ (much smaller than $|pq|$) such that for every $x\in B(q,r_0)\setminus\{q\}$ we have
\begin{equation}\label{eq:orth}
|\tilde\angle pqx-\angle pqx|<\varkappa(\delta),
\end{equation}
where $\varkappa(\delta)\to0$ as $\delta\to 0$ (see Notation \ref{nota:del}).
\end{lem}

\begin{proof}
This follows from local compactness and the Busemann monotonicity \eqref{eq:mono} (compare with the proof of Lemma \ref{lem:bran}).
By rescaling, we may assume $B=B(q,1)$.
In what follows, we denote by $\varkappa(\delta)$ various positive functions such that $\varkappa(\delta)\to0$ as $\delta\to 0$, as explained in Notation \ref{nota:del}.

Let $A\subset\partial B(q,1)$ be a maximal $\delta$-separated set.
By compactness, $A$ is finite.
As in Definition \ref{dfn:orth}, we can define the angle
\[\angle pqz=\lim_{t\to 0}\tilde\angle pqz(t)\]
for any $z\in A$.
Since $A$ is finite, there exists $r_0>0$ such that for any $z\in A$ and $0<t<r_0$, we have
\begin{equation}\label{eq:orth'}
|\tilde\angle pqz(t)-\angle pqz|<\delta.
\end{equation}
In other words, we have chosen $r_0$ so that the desired inequality \eqref{eq:orth} holds on the finitely many shortest paths to $A$.
For later use, we also assume $r_0<\delta|pq|$.

By the Busemann monotonicity \eqref{eq:mono}, these finitely many shortest paths are $\delta$-dense among all the shortest paths from $q$ to $\partial B(q,1)$ with natural parametrization.
This fact, together with the assumption \eqref{eq:orth'}, implies the desired inequality \eqref{eq:orth}.
We will see the details below.
See Figure \ref{fig:orth1}.

\begin{figure}[ht]
\centering
\begin{tikzpicture}
\coordinate[label=below:$p$](p)at(-0.5,-4);
\coordinate[label=left:$q$](q)at(0,0);
\coordinate[label=below:$x$](x)at(0.5,0);
\coordinate[label=right:$x(1)$](x1)at(2,0);
\coordinate[label=above right:$x'$](x')at(15:2);

\draw(q)to(x)to(x1);
\draw(q)to(x');
\draw(p)to(q);

\draw(q)circle[radius=1];
\draw(q)circle[radius=2];
\node[below right]at(-45:1){$r_0$};
\node[below right]at(-45:2){$1$};

\fill(p)circle(1.5pt);
\fill(q)circle(1.5pt);
\fill(x)circle(1.5pt);
\fill(x1)circle(1.5pt);
\fill(x')circle(1.5pt);
\end{tikzpicture}
\caption{}\label{fig:orth1}
\end{figure}

Let $x\in B(q,r_0)\setminus\{q\}$ be an arbitrary point.
We denote by $x(1)\in\partial B(q,1)$ a point on an extension of the shortest path $qx$ beyond $x$.
Since $A$ is a maximal $\delta$-separated set of $\partial B(q,1)$, there exists $x'\in A$ that is $\delta$-close to $x(1)$.
By the Busemann monotonicity \eqref{eq:mono}, we have $|x(t)x'(t)|<\delta t$ for any $0<t<1$.
Together with $r_0<\delta |pq|$, this gives us
\[|\tilde\angle pqx(t)-\tilde\angle pqx'(t)|<\varkappa(\delta)\]
for any $0<t<r_0$.
The latter angle $\tilde\angle pqx'(t)$ is almost constant up to the error $\delta$ for $0<t<r_0$, as already assumed in \eqref{eq:orth'}.
This completes the proof.
\end{proof}

Next we show a basic property of this angle.
In the GCBA case, the following is a special case of the triangle inequality for angles.
In the GNPC case, this can be viewed as an infinitesimal version of the Busemann convexity \eqref{eq:conv}.
Note that we do not use the fixed point $q$ here.

\begin{lem}\label{lem:sum}
Let $X$ be a GNPC space, $B$ a tiny ball, and $p\in B$.
Let $x,y,z\in B\setminus\{p\}$ be such that $x$ lies on the shortest path between $y$ and $z$.
Then we have
\[\angle pxy+\angle pxz\ge\pi.\]
\end{lem}

\begin{proof}
Recall that for any $\alpha,\beta\in[0,\pi]$, we have
\[\cos\alpha+\cos\beta\le 0\iff\alpha+\beta\ge\pi.\]
Let $\gamma(t)$ be the unit-speed parametrization of the shortest path between $y$ and $z$ such that $\gamma(0)=x$.
Then the Busemann convexity \eqref{eq:conv} implies
\[|p\gamma(t)|+|p\gamma(-t)|\ge2|px|\]
for any small $t>0$.
This gives us
\[\frac{|p\gamma(t)|-|px|}t+\frac{|p\gamma(-t)|-|px|}t\ge0.\]
Taking $t\to0$ and using the definition of angle (cf.\ Remark \ref{rem:orth}), we obtain
\[\cos\angle pxy+\cos\angle pxz\le 0,\]
as desired.
\end{proof}

Combining the above three lemmas, we obtain the following key property.
In the GCBA case, this is a combination of the triangle inequality and the upper semicontinuity of angles (\cite[Lemma 5.5]{LN:geo}).

\begin{lem}\label{lem:orth2}
Let $X$ be a GNPC space, $B$ a tiny ball, and $p\neq q\in B$.
For any $1\gg\delta>0$, there exists $r_1=r_1(p,q,\delta)>0$ (much smaller than $r_0$ of Lemma \ref{lem:orth1}) satisfying the following.
For any $x\in B(q,r_1)\setminus\{q\}$, let $\hat x$ be a point on an extension of the shortest path $qx$ beyond $x$ at distance $\delta^{-1}r_1$ from $q$.
Then we have
\begin{gather*}
|\angle pxq+\angle px\hat x-\pi|<\varkappa(\delta),\\
|\angle pxq-\tilde\angle pxq|<\varkappa(\delta),\quad|\angle px\hat x-\tilde\angle px\hat x|<\varkappa(\delta).
\end{gather*}
See Figure \ref{fig:orth2}.
\end{lem}

\begin{figure}[ht]
\centering
\begin{tikzpicture}
\coordinate[label=below:$p$](p)at(-0.5,-4);
\coordinate[label=above left:$q$](q)at(0,0);
\coordinate[label=above:$x$](x)at(0.5,0);
\coordinate[label=above right:$\hat x$](x')at(2,0);

\draw(p)to(x);
\draw(q)to(x)to(x');
\draw[dashed](p)to(q);
\draw[dashed](p)to(x');

\draw(q)circle[radius=1];
\node[above right]at(45:1){$r_1$};

\fill(p)circle(1.5pt);
\fill(q)circle(1.5pt);
\fill(x)circle(1.5pt);
\fill(x')circle(1.5pt);

\pic[draw, angle radius=2.5mm] {angle = q--x--p};
\pic[draw, angle radius=2.5mm] {angle = p--x--x'};
\end{tikzpicture}
\caption{}\label{fig:orth2}
\end{figure}

\begin{proof}
Suppose $r_1$ is small enough compared to $|pq|$, $\delta$, and the radius $r_0$ from Lemma \ref{lem:orth1} (to be determined later).
By Lemma \ref{lem:sum}, we have $\angle pxq+\angle px\hat x\ge\pi$.
Since $r_1$ is small enough compared to $|pq|$ and $\delta$, the almost comparison \eqref{eq:comp} implies that
\[\angle pxq<\tilde\angle pxq+\delta,\quad\angle px\hat x<\tilde\angle px\hat x+\delta.\]
Here we used the fact that the almost comparison holds at variable $x$, not necessarily at fixed $q$ (see Remark \ref{rem:comp}).
Since $|qx|<\delta|q\hat x|=r_1$, we have $|\tilde\angle px\hat x-\tilde\angle pq\hat x|<\varkappa(\delta)$.
Furthermore, assuming $r_1<\delta r_0$, we have $\hat x\in B(q,r_0)$.
By Lemma \ref{lem:orth1}, we get
\[|\tilde\angle pq\hat x-\tilde\angle pqx|<\varkappa(\delta).\]
Combining all the inequalities above, we obtain
\begin{align*}
\pi&\le\angle pxq+\angle px\hat x\tag{Lemma \ref{lem:sum}}\\
&\le\tilde\angle pxq+\tilde\angle px\hat x+\varkappa(\delta)\tag{\eqref{eq:comp}}\\
&\le\tilde\angle pxq+\tilde\angle pq\hat x+\varkappa(\delta)\tag{$|qx|<\delta|q\hat x|$}\\
&\le\tilde\angle pxq+\tilde\angle pqx+\varkappa(\delta)\tag{Lemma \ref{lem:orth1}}\\
&\le\pi+\varkappa(\delta).
\end{align*}
Therefore, all the equalities hold up to the error $\varkappa(\delta)$.
This completes the proof.
\end{proof}

Finally, we reach the following property, which we call the \textit{almost orthogonality} of geodesics.

\begin{prop}\label{prop:orth}
Let $X$ be a GNPC space, $B$ a tiny ball, and $p\neq q\in B$ with $R:=|pq|$.
For any $1\gg\delta>0$, let $r_1=r_1(p,q,\delta)>0$ be the radius from Lemma \ref{lem:orth2} (in particular assume $r_1<\delta R$).
Suppose $x\in\partial B(p,R)\cap\partial B(q,r)$ and $y\in B(x,\delta r)$, where $0<r<r_1$.
Then we have
\[|\angle pyq-\pi/2|<\varkappa(\delta),\quad|\angle py\hat y-\pi/2|<\varkappa(\delta),\]
where $\hat y$ is an arbitrary point on an extension of the shortest path $qy$ beyond $y$.
\end{prop}

\begin{proof}
Since $|px|=|pq|=R$ and $|qx|=r<\delta R$, we have $|\tilde\angle pxq-\pi/2|<\varkappa(\delta)$.
Since $|xy|<\delta r$, this implies $|\tilde\angle pyq-\pi/2|<\varkappa(\delta)$.
Hence the desired inequalities follow from Lemma \ref{lem:orth2}.
\end{proof}

\begin{rem}\label{rem:orthglo}
In Proposition \ref{prop:orth} (and preceding lemmas), the distance $R$ between $p$ and $q$ may be arbitrarily large, as long as they are contained in a tiny ball $B$.
This allows us to develop extended strainers in Section \ref{sec:exstr}, which is necessary to tackle a global problem like Theorem \ref{thm:four}.
\end{rem}

\section{Strainers}\label{sec:str}

In this section, we develop strainers for GNPC spaces, which play central roles in the proofs of the main theorems.
The definition and existence are based on the key lemmas proved in the previous section, specifically their consequences, Propositions \ref{prop:bran} and \ref{prop:orth}.
Once we establish the existence of strainers, we will not need to refer back to those lemmas when proving further properties.

The organization of this section is as follows.
In Section \ref{sec:strdfn}, we define strainers and show their existence.
In Section \ref{sec:strfib}, we prove the fibration property of strainers.
In Section \ref{sec:strcon}, as consequences, we prove Theorems \ref{thm:mfd} and \ref{thm:link}.

\subsection{Definition and existence}\label{sec:strdfn}

In this subsection, we define strainers for GNPC spaces and show their existence (cf.\ \cite[Section 7]{LN:geo}).

Throughout this subsection, $X$ denotes a GNPC space, $B$ a tiny ball, and $\delta$ a small positive number (see Notation \ref{nota:del} and Remark \ref{rem:del}).
The definition of a strainer is given in terms of the angle
\[\angle pqx\]
viewed from $p$ at $q$ toward $x$, introduced in Section \ref{sec:orth}.
We again emphasize that this angle is essentially the directional derivative of the distance function from $p$, i.e.,
\[\cos\angle pqx=-\left.\frac d{dt}\right|_{t=0}d(p,x(t)),\]
where $x(t)$ is the unit-speed shortest path from $q$ to $x$ (see Remark \ref{rem:orth}).
Therefore, the reader should consider the following definition of a strainer as controlling the derivative of the distance function rather than the angles, unlike in the GCBA case.
Note that the other angle introduced in Section \ref{sec:bran} is not used for the definition of a strainer, but will be used later to show the existence of a strainer.

We first define a $1$-strainer, which controls the branching of geodesics.

\begin{dfn}\label{dfn:1str}
A point $p\in B$ is called a \textit{$(1,\delta)$-strainer} at $x\in B$ if there exists $\bar p\in B$ and a neighborhood $U\subset B$ of $x$ such that for any $y\in U$, we have
\[\angle py\bar p>\pi-\delta,\]
where we are assuming $p,\bar p\in B\setminus U$.
We call $\bar p$ the \textit{opposite strainer} for $p$ at $x$.
See Figure \ref{fig:1str}.
\end{dfn}

\begin{figure}[ht]
\centering
\begin{tikzpicture}
\coordinate[label=below left:$p$](p)at(-0.5,-2);
\coordinate[label=above left:$\bar p$](p')at(0.5,2);
\coordinate(x)at(0,0);
\coordinate[label=left:$y$](y)at(-0.1,0);

\draw(p)to(y)to(p');
\draw[dashed](x)circle[radius=0.75];
\node[below right]at(-45:0.75){$U$};

\fill(p)circle(1.5pt);
\fill(p')circle(1.5pt);
\fill(y)circle(1.5pt);

\pic[draw, angle radius=2.5mm] {angle = p--y--p'};
\end{tikzpicture}
\caption{}\label{fig:1str}
\end{figure}

The point of the above definition is that the inequality holds for the same $\bar p$ independent of $y\in U$ (otherwise, one may choose $\bar p=\bar p(y)$ as a point on an extension of the shortest path $py$ beyond $y$).
Roughly speaking, this means that there is one direction at $x$ that uniformly increases the distance from $p$ around $x$.
If the shortest path $px$ branches significantly at $x$, there is no such direction.
This is the same as in the GCBA case.

Based on the above definition, we next define a $k$-strainer ($k\ge 2$) that takes into account the almost orthogonality.
Our definition is by induction.

\begin{dfn}\label{dfn:kstr}
A sequence of points $p_1\dots,p_k\in B$ is called a \textit{$(k,\delta)$-strainer} at $x\in B$ if it satisfies the following inductive condition:
\begin{enumerate}
\item $p_1,\dots,p_{k-1}$ is a $(k-1,\delta)$-strainer at $x$;
\item $p_k$ is a $(1,\delta)$-strainer at $x$;
\item there exists a neighborhood $U\subset B$ of $x$ such that for any $y\in U$, we have
\[|\angle p_iyp_k-\pi/2|<\delta,\quad|\angle p_iy\bar p_k-\pi/2|<\delta\]
for any $1\le i\le k-1$, where $\bar p_k$ is an opposite strainer for $p_k$ at $x$ (we are assuming $p_1,\dots,p_k,\bar p_k\in B\setminus U$).
\end{enumerate}
See Figure \ref{fig:kstr}.
\end{dfn}

\begin{figure}[ht]
\centering
\begin{tikzpicture}
\coordinate[label=below:{$p_1,\dots,p_{k-1}$}](p1)at(-0.5,-4);
\coordinate[label=above left:$p_k$](pk)at(-1.5,0);
\coordinate[label=above right:$\bar p_k$](pk')at(1.5,0);
\coordinate(x)at(0,0);
\coordinate[label=above:$y$](y)at(0,-0.1);

\draw(pk)to(y)to(pk');
\draw(p1)to(y);
\draw[dashed](x)circle[radius=0.75];
\node[below right]at(-45:0.75){$U$};

\fill(p1)circle(1.5pt);
\fill(pk)circle(1.5pt);
\fill(pk')circle(1.5pt);
\fill(y)circle(1.5pt);

\pic[draw, angle radius=2.5mm] {angle = pk--y--p1};
\pic[draw, angle radius=2.5mm] {angle = p1--y--pk'};
\end{tikzpicture}
\caption{}\label{fig:kstr}
\end{figure}

The point of the above definition is that only $\angle p_iyp_k$ and $\angle p_iy\bar p_k$ are almost $\pi/2$, and we make no assumptions on $\angle p_kyp_i$ and $\angle p_ky\bar p_i$.
Thus, when we move from $y$ to $p_k$ or $\bar p_k$, the distance from $p_i$ remains almost unchanged, but there is no control on the distance from $p_k$ when we move to $p_i$ or $\bar p_i$.
This asymmetry is the crucial difference from the GCBA case, and is the Finsler nature of GNPC spaces we should deal with carefully.
However, it turns out that the above asymmetric condition is enough to derive the same topological properties of a strainer as in the GCBA case.

Note that the definition of a strainer is an open condition.
For later use, we also introduce (a lower bound for) the straining radius, which gives a quantitative estimate on the openness of a strainer.

\begin{dfn}\label{dfn:rad}
Let $p_1,\dots,p_k\in B$ be a $(k,\delta)$-strainer at $x\in B$.
We say that $p_1,\dots,p_k$ has \textit{straining radius} $>r$ at $x$ if it is a $(k,\delta)$-strainer at any $y\in B(x,r)$ with the same opposite strainer $\bar p_1,\dots,\bar p_k$.
\end{dfn}

Clearly, the straining radius is positive.
It also has the following uniformity.
This simple fact plays an important role in the proof of the local uniform contractibility of the fibers of a strainer map (Theorem \ref{thm:str}).

\begin{rem}\label{rem:rad}
Let $p_1,\dots,p_k\in B$ be a $(k,\delta)$-strainer at $x\in B$ with straining radius $>r$.
Then it has straining radius $>r/2$ at any $y\in B(x,r/2)$.
\end{rem}

Now we prove the existence of a strainer.
First we show the existence of a $1$-strainer.
This follows from the key lemmas in Section \ref{sec:bran}, specifically Proposition \ref{prop:bran}.
Compare with \cite[Proposition 7.3]{LN:geo}.

\begin{prop}\label{prop:1str}
For any $p\in X$ and $1\gg\delta>0$, there exists $r_0=r_0(p,\delta)>0$ such that $p$ is a $(1,\delta)$-strainer at any point of $B(p,r_0)\setminus\{p\}$.
\end{prop}

\begin{proof}
Let $r_0$ be the radius from Proposition \ref{prop:bran}.
In what follows, we denote by $\varkappa(\delta)$ various positive functions such that $\varkappa(\delta)\to0$ as $\delta\to0$, as explained in Notation \ref{nota:del}.

Fix $x\in\partial B(p,r)$, where $0<r<r_0$, and let $y\in B(x,\delta r)$.
By Proposition \ref{prop:bran},
\begin{equation}\label{eq:2x2y}
|2x,2y|<\varkappa(\delta)r,
\end{equation}
where $2x$ denotes a point on an extension of the shortest paths $px$ at distance $2|px|$ from $p$ ($2y$ is the same notation for $y$).
We show that $\bar p:=2x$ is an opposite strainer for $p$ at $x$, that is,
\begin{equation}\label{eq:pyp}
\angle py\bar p>\pi-\varkappa(\delta).
\end{equation}
for any $y\in B(x,\delta r)$ (see Definition \ref{dfn:1str}; the error being $\varkappa(\delta)$ instead of $\delta$ does not matter as we will see below).
See Figure \ref{fig:1strprf}.

\begin{figure}[ht]
\centering
\begin{tikzpicture}
\coordinate[label=below left:$p$](p)at(-0.5,-2);
\coordinate[label=left:$x$](x)at(0,0);
\coordinate[label=right:$y$](y)at(0.3,-0.1);
\coordinate[label=above left:$2x$](2x)at(0.5,2);
\coordinate[label=above right:$2y$](2y)at(1.1,1.8);

\draw(p)to(x)to(2x);
\draw(p)to(y)to(2y);
\draw(y)to(2x);
\draw[dashed](0,0)circle[radius=0.75];
\node[below right]at(-45:0.75){$\delta r$};

\fill(p)circle(1.5pt);
\fill(x)circle(1.5pt);
\fill(y)circle(1.5pt);
\fill(2x)circle(1.5pt);
\fill(2y)circle(1.5pt);
\end{tikzpicture}
\caption{}\label{fig:1strprf}
\end{figure}

Let $\gamma$ and $\eta$ denote the shortest paths from $y$ to $2x$ and $2y$, respectively.
By linear reparametrization, we may assume that $\gamma$ and $\eta$ are defined on $[0,r]$.
Note that these reparametrizations are almost isometries (i.e., bi-Lipschitz homeomorphisms with Lipschitz constants $1\pm\varkappa(\delta)$) since the lengths of $\gamma$ and $\eta$ are almost $r$.
By the inequality \eqref{eq:2x2y} and the Busemann monotonicity \eqref{eq:mono}, we have
\begin{equation}\label{eq:gameta}
|\gamma(t)\eta(t)|<\varkappa(\delta)t.
\end{equation}
Since $\eta$ is part of the shortest path from $p$ to $2y$, the distance from $p$ increases with velocity almost $1$ along $\eta$ (not exactly $1$ because of the reparametrization).
In other words,
\begin{equation}\label{eq:peta}
|p\eta(t)|>|py|+(1-\varkappa(\delta))t.
\end{equation}
Together with the triangle inequality, these imply
\begin{align*}
|p\gamma(t)|&\ge|p\eta(t)|-|\gamma(t)\eta(t)|\\
&\ge|py|+(1-\varkappa(\delta))t-\varkappa(\delta)t\tag{\eqref{eq:gameta}, \eqref{eq:peta}}\\
&=|py|+(1-\varkappa(\delta))t.
\end{align*}
Therefore the derivative of $|p\gamma(t)|$ at $t=0$ is greater than $1-\varkappa(\delta)$, which implies the desired inequality \eqref{eq:pyp} (see Remark \ref{rem:orth}; note that $\gamma(t)$ has almost unit-speed).
Since $\varkappa(\delta)$ depends only on $\delta$, we can replace $\varkappa(\delta)$ with $\delta$ by choosing $r_0$ small enough.
This completes the proof.
\end{proof}

\begin{rem}\label{rem:1strrad}
The above proof shows that the $(1,\delta)$-strainer $p$ has straining radius $>\delta'r$ at $x\in\partial B(p,r)$, where $\delta'\ll\delta$ is a constant depending only on $\delta$.
\end{rem}

Next we show the existence of a $k$-strainer, or more precisely, how to extend a strainer.
This follows from the previous proposition and the key lemmas in Section \ref{sec:orth}, specifically Proposition \ref{prop:orth}.
Compare with \cite[Proposition 9.4]{LN:geo}.

\begin{prop}\label{prop:kstr}
Let $p_1,\dots,p_{k-1}\in B$ be a $(k-1,\delta)$-strainer at $p_k\in B$, where $k\ge 2$.
Suppose $x\in B\setminus\{p_k\}$ is sufficiently close to $p_k$ and has the same distances from $p_i$'s as $p_k$, i.e.,
\[|p_ip_k|=|p_ix|\]
for any $1\le i\le k-1$.
Then $p_1,\dots,p_k$ is a $(k,\delta)$-strainer at $x$.
\end{prop}

\begin{proof}
We check the three conditions of Definition \ref{dfn:kstr}.
The first is trivial from the openness of a strainer (cf.\ Remark \ref{rem:rad}), provided $x$ is sufficiently close to $p_k$.
Similarly, the second follows from Proposition \ref{prop:1str}.
We show that the third follows from Proposition \ref{prop:orth} (and the proof of Proposition \ref{prop:1str}).

Let $y\in B(x,\delta r)$, where $r:=|p_kx|$.
Since $|p_ip_k|=|p_ix|$, Proposition \ref{prop:orth} implies
\[|\angle p_iyp_k-\pi/2|<\varkappa(\delta),\]
provided $r$ is small enough (depending on $p_i,p_k,\delta$).
Therefore it remains to show
\begin{equation}\label{eq:piypk}
|\angle p_iy\bar p_k-\pi/2|<\varkappa(\delta),
\end{equation}
where $\bar p_k$ is an opposite strainer for $p_k$ at $x$ (as before, we can eventually replace $\varkappa(\delta)$ with $\delta$ since $\varkappa(\delta)$ depends only on $\delta$).
See Figure \ref{fig:kstrprf}.

\begin{figure}[ht]
\centering
\begin{tikzpicture}
\coordinate[label=below:{$p_1,\dots,p_{k-1}$}](p1)at(-0.5,-4);
\coordinate[label=above left:$p_k$](pk)at(-1.5,0);
\coordinate[label=above right:$\bar p_k$](pk')at(1.5,0);
\coordinate[label=above:$x$](x)at(0,0);
\coordinate[label=below left:$y$](y)at(0,-0.25);
\coordinate[label=below right:$2y$](2y)at(1.5,-0.5);

\draw(p1)to(y);
\draw(pk)to(y)to(pk');
\draw(pk)to(y)to(2y);
\draw[dashed](x)circle[radius=0.75];
\node[above right]at(45:0.75){$\delta r$};

\fill(p1)circle(1.5pt);
\fill(pk)circle(1.5pt);
\fill(x)circle(1.5pt);
\fill(y)circle(1.5pt);
\fill(pk')circle(1.5pt);
\fill(2y)circle(1.5pt);
\end{tikzpicture}
\caption{}\label{fig:kstrprf}
\end{figure}

As in the proof of Proposition \ref{prop:1str}, we take $\bar p_k:=2x$ as an opposite strainer for $p_k$ at $x$, where $2x$ denotes a point on an extension of the shortest path $p_kx$ at distance $2|p_kx|$ from $p_k$.
Using the same notation, we define $\hat y:=2y$.
Then Proposition \ref{prop:orth} implies that
\begin{equation}\label{eq:piyy}
|\angle p_iy\hat y-\pi/2|<\varkappa(\delta).
\end{equation}

Let $\gamma$ and $\eta$ denote the linear reparametrizations of the shortest paths from $y$ to $\bar p_k=2x$ and $\hat y=2y$ defined on $[0,r]$, respectively.
From the proof of Proposition \ref{prop:1str}, we obtain the same inequality as \eqref{eq:gameta}:
\[|\gamma(t)\eta(t)|<\varkappa(\delta)t.\]
This gives us
\[|\tilde\angle p_iy\gamma(t)-\tilde\angle p_iy\eta(t)|<\varkappa(\delta).\]
Taking $t\to 0$ (and noting that $\gamma$ and $\eta$ have almost unit-speed), we have
\[|\angle p_iy\bar p_k-\angle p_iy\hat y|<\varkappa(\delta)\]
Combining this with \eqref{eq:piyy}, we obtain the desired inequality \eqref{eq:piypk}.
\end{proof}

\subsection{Fibration property}\label{sec:strfib}

In this subsection, we establish the fibration property of strainers, which is the key ingredient in the proofs of the main theorems.
This proof relies only on the definition of a strainer, so we no longer need to refer back to the key lemmas in Section \ref{sec:key}.

Henceforth, in addition to $\delta$ and $\varkappa(\delta)$ from Notation \ref{nota:del}, we often use the following notation.

\begin{nota}\label{nota:k}
For a natural number $k$, we use the following symbols:
\begin{itemize}
\item $\delta_k$ and $\varepsilon_k$ denote small positive constants depending only on $k$.
\item $\varkappa_k(\delta)$ denotes various positive functions with the same property as $\varkappa(\delta)$ but depending additionally on $k$.
\end{itemize}
The constant $\delta_k$ appears in the assumption of a statement for a $(k,\delta)$-strainer, ensuring that  whenever $\delta<\delta_k$, the conclusion holds.
The constant $\varepsilon_k$ usually appears in the conclusion of a statement. To be more precise, the statement claims
the existence of such constants $\delta_k$ and $\varepsilon_k$.
These constants will be determined explicitly in the proof of the statement by induction on $k$.
\end{nota}

The main object of this subsection is the following strainer map.

\begin{dfn}\label{dfn:strm}
Let $X$ be a GNPC space, $B$ a tiny ball, and $p_1,\dots,p_k\in B$ a $(k,\delta)$-strainer at $x\in B$.
We call a distance map
\[f=(d(p_1,\cdot),\dots,d(p_k,\cdot))\]
a \textit{$(k,\delta)$-strainer map} at $x$ associated with the strainer $p_1,\dots,p_k$.

Furthermore, for a (not necessarily open) subset $A\subset B$, we say that a distance map $f$ is a \textit{$(k,\delta)$-strainer map} on $A$ if it is a $(k,\delta)$-strainer map at any point of $A$ (possibly with different opposite strainers).
\end{dfn}

The following two theorems are the main results of this subsection; compare with \cite[Theorem 1.11]{LN:geo} and \cite[Theorem 5.1]{LN:top}, respectively.
Let $X$ be a GNPC space, and let $\delta_k$ be a positive constant depending only on $k$, as in Notation \ref{nota:k}.

\begin{thm}\label{thm:str}
Let $f$ be a $(k,\delta)$-strainer map at $x\in X$, where $\delta<\delta_k$.
Then $f$ is open and has locally uniformly contractible fibers in an open neighborhood of $x$.
\end{thm}

\begin{thm}\label{thm:fib}
Let $f$ be a $(k,\delta)$-strainer map at $x\in X$, where $\delta<\delta_k$.
Then there exists an arbitrarily small contractible open neighborhood $V$ of $x$ such that $f|_V:V\to f(V)$ is a Hurewicz fibration with contractible fibers.
\end{thm}

See Definition \ref{dfn:fib} for local uniform contractibility of fibers.
As we will see later, Theorem~\ref{thm:fib} is a combination of Theorem~\ref{thm:str}, Proposition \ref{prop:kstr}, and Theorems \ref{thm:ung1} and \ref{thm:ung2}.
Note that both of the above statements are local.

\begin{rem}\label{rem:fib}
In general, $f$ need not be a locally trivial fibration even in the GCBA case.
This is a crucial difference from the Alexandrov case (\cite{Per:alex}, \cite{Per:mor}).

A simple counterexample is constructed as follows.
Let $S$ be the circle of length $2\pi+\delta$ with $0<\delta\ll1$.
Let $p,q\in S$ be antipodal points, and let $I\subset S$ be the interval of length $\delta$ centered at $q$.
Glue two copies of $S$ along $p$ and $I$, and call the resulting space $\hat S$.
The Euclidean cone $X$ over $\hat S$ is a geodesically complete CAT($0$) space.
In that scenario, $p\in\hat S\subset X$ is a $(1,\varkappa(\delta))$-strainer at the vertex $o$ of $X$, but $d(p,\cdot)$ is not a trivial fibration around $o$.
See \cite{NSY:1} for much more complicated examples.
\end{rem}

\begin{rem}\label{rem:bilip}
Theorem \ref{thm:fib} still holds even when $x$ is an isolated point of the fiber of $f$, as noted in Remark \ref{rem:isol}.
Furthermore, in the next subsection we prove the following: if $f$ cannot be extended to a $(k+1,\delta)$-strainer map $(f,g)$ at any point near $x$, then $f$ is a bi-Lipschitz homeomorphism between open neighborhoods of $x$ and $f(x)$.
See Theorem \ref{thm:bilip}.
\end{rem}

The key technique used in the proof of Theorem \ref{thm:str} is the so-called \textit{consecutive approximation}, originating in the fundamental theory of Alexandrov spaces from Burago--Gromov--Perelman \cite{BGP} and Perelman \cite{Per:alex}.
In some sense, what we develop here can be seen as the asymmetric version of it.
We modify the standard Euclidean metric of the target space of a strainer map so that the asymmetry of the almost orthogonality can be ignored.

We remark that this type of asymmetry appeared in the technical proof of the stability theorem for Alexandrov spaces by Perelman \cite{Per:alex}, and more relevantly, in Pogorelov's approach \cite{Po} to Busemann G-spaces with some regularity.

To prove the openness, we introduce a metric version of an open map.

\begin{dfn}\label{dfn:open}
Let $f:X\to Y$ be a map between metric spaces and let $\varepsilon>0$.
We say that $f$ is \textit{$\varepsilon$-open} if, for each $x\in X$, there exists $r>0$ such that for any $v\in B(f(x),r)$, there is $y\in X$ with
\[f(y)=v,\quad\varepsilon|xy|\le|f(x)v|.\]
\end{dfn}

In particular, for any $0<s<r$, we have $f(B(x,\varepsilon^{-1}s))\supset B(f(x),s)$.
Therefore every $\varepsilon$-open map is a (topological) open map.

The following is a criterion for the $\varepsilon$-openness (cf.\ \cite[2.2.1]{Per:alex}, \cite[Lemma 8.1]{LN:geo}; see \cite[Theorem 1.2]{Ly:open} for a more general statement).
We say that a metric space $Y$ is \textit{locally geodesic} if each point has a neighborhood $U$ that is geodesic in $Y$, i.e., any two points in $U$ are connected by a shortest path in $Y$.

\begin{lem}\label{lem:open}
Let $f:X\to Y$ be a continuous map between metric spaces.
Suppose $X$ is locally compact, $Y$ is locally geodesic, and there exists $\varepsilon>0$ satisfying the following: for every $x\in X$ and each $v\in Y\setminus\{f(x)\}$ sufficiently close to $f(x)$, there exists $y\in X\setminus\{x\}$ such that
\begin{equation}\label{eq:open}
|f(y)v|-|f(x)v|\le-\varepsilon|xy|.
\end{equation}
Then $f$ is an $\varepsilon$-open map.
\end{lem}

For completeness, we give the straightforward proof.

\begin{proof}
Let $x\in X$ and $v\in B(f(x),r)\setminus\{f(x)\}$, where $r>0$ is small enough (to be determined later).
We set
\[A:=\left\{y\in X\mid|f(y)v|-|f(x)v|\le-\varepsilon|xy|\right\},\]
which is nonempty as it contains $x$.
Since $|f(x)v|<r$, we have $A\subset B(x,\varepsilon^{-1}r)$.
Hence $A$ is compact, provided $r$ is small enough (recall $X$ is locally compact).

It suffices to show that $f(A)$ contains $v$.
Suppose $f(A)$ does not contain $v$.
Since $A$ is compact and $f$ is continuous, there exists $y\in A$ such that
\[|f(y)v|=|f(A)v|>0.\]
In particular, $|f(y)v|\le|f(x)v|<r$, and thus $f(y)\in B(f(x),2r)$.
Now consider a shortest path $f(y)v$, provided $r$ is small enough (recall $Y$ is locally geodesic).
Choose a point $w$ sufficiently close to $f(y)$ on the shortest path $f(y)v$.
Applying the assumption \eqref{eq:open} at $y$, we find $z\in X\setminus\{y\}$ such that
\[|f(z)w|-|f(y)w|\le-\varepsilon|yz|.\]
Since $w$ lies on the shortest path $f(y)v$, this implies
\begin{align*}
|f(z)v|-|f(y)v|&=|f(z)v|-|f(y)w|-|wv|\\
&\le|f(z)w|-|f(y)w|\\
&\le-\varepsilon|zy|.
\end{align*}
In particular, $f(z)$ is closer to $v$ than $f(y)$.
Furthermore, we have
\begin{align*}
|f(z)v|-|f(x)v|&=(|f(z)v|-|f(y)v|)+(|f(y)v|-|f(x)v|)\\
&\le-\varepsilon|zy|-\varepsilon|yx|\\
&\le-\varepsilon|zx|.
\end{align*}
This means $z\in A$, which contradicts the choice of $y$.
\end{proof}

\begin{rem}\label{rem:openrad}
The above proof shows that it suffices to choose $r$ so that $\bar B(x,\varepsilon^{-1}r)$ is compact and $B(f(x),2r)$ is geodesic in $Y$.
\end{rem}

We now prove the openness of a strainer map, which is the first half of Theorem \ref{thm:str}.
Compare the statement below with \cite[Lemma 8.2]{LN:geo}.

\begin{prop}\label{prop:open}
Let $X$ be a GNPC space, $B$ a tiny ball, and $p_1,\dots,p_k\in B$ a $(k,\delta)$-strainer at $x\in B$.
Let $f$ be the associated strainer map at $x$.
If $\delta<\delta_k$, then $f$ is $\varepsilon_k$-open in an open neighborhood of $x$, with respect to the Euclidean norm of $\mathbb R^k$.
Here $\delta_k$ and $\varepsilon_k$ are positive constants depending only on $k$.
\end{prop}

\begin{proof}
The proof is similar to that of \cite[Lemma 8.2]{LN:geo}, except that we must deal with the asymmetry of a strainer.
Given $v\in\mathbb R^k$ near $f(x)$, we move $x$ toward the strainer $p_1,\dots,p_k$ or the opposite strainer $\bar p_1,\dots,\bar p_k$ so that the image of $f$ approaches to the desired value $v$.
Here, we modify the norm of $\mathbb R^k$ so that the image of $f$ approaches to $v$ at a constant speed with respect to this new norm.
Since our definition of a strainer is inductive, we use induction on $k$.

\step{Step 1}
The base case $k=1$ is clear from Definition \ref{dfn:1str} (and Remark \ref{rem:orth}).
Given $v\in\mathbb R$ near $f(x)$, we find $y\in X$ on the shortest path $p_1x$ or $\bar p_1x$ such that
\[f(y)=v,\quad(1-\varkappa(\delta))|xy|\le|f(x)v|.\]
The same is true for any point around $x$.
Therefore $f$ is $(1-\varkappa(\delta))$-open around $x$.
It suffices to define $\delta_1>0$ so that $\varepsilon_1:=1-\varkappa(\delta_1)>0$.

\step{Step 2}
Suppose the claim holds for any $(k-1,\delta)$-strainer map (in particular, we assume the existence of the constants $\delta_{k-1}$ and $\varepsilon_{k-1}$).
Let us denote
\[f=(f_{[k-1]},f_k),\]
where $f_{[k-1]}$ is the first $(k-1)$-coordinates of $f$ and $f_k$ is the last $k$-th coordinate of $f$.
By Definition \ref{dfn:kstr}(1), $f_{[k-1]}$ is a $(k-1,\delta)$-strainer map at $x$, and hence by the inductive assumption, it is an $\varepsilon_{k-1}$-open map around $x$, provided $\delta<\delta_{k-1}$.
Similarly, by Definition \ref{dfn:kstr}(2), $f_k$ is a ($1,\delta$)-strainer map; thus, by the base case, Step 1, it is $(1-\varkappa(\delta))$-open around $x$.
We prove that the almost orthogonality of a strainer, Definition \ref{dfn:kstr}(3), ensures that their product $f$ is $\varepsilon_k$-open around $x$ for some $\varepsilon_k\ll\varepsilon_{k-1}$, provided $\delta<\delta_k\ll\delta_{k-1}$.

To this end, we introduce a new norm on $\mathbb R^k$.
For $v=(v_{[k-1]},v_k)\in\mathbb R^{k-1}\times\mathbb R$, we define
\[||v||:=|v_{[k-1]}|+(\varepsilon_{k-1}/2)|v_k|,\]
where $|\cdot|$ denotes the standard Euclidean norm.

We will show that the assumption \eqref{eq:open} of Lemma \ref{lem:open} is satisfied for this norm (hence the proof is by contradiction and not constructive).
Since the definition of a strainer is open (cf.\ Remark \ref{rem:rad}), it suffices to verify the condition \eqref{eq:open} at $x$.
Suppose $v\in\mathbb R^k$ is sufficiently close to $f(x)$.
We split the proof into two cases:
\begin{enumerate}
\item the case $f_{[k-1]}(x)\neq v_{[k-1]}$,
\item the case $f_k(x)\neq v_k$.
\end{enumerate}

Suppose (1) holds.
By the inductive assumption, there is $y\in X$ such that
\begin{equation}\label{eq:openind}
f_{[k-1]}(y)=v_{[k-1]},\quad\varepsilon_{k-1}|xy|\le|f_{[k-1]}(x)v_{[k-1]}|.
\end{equation}
Therefore
\begin{align*}
&||f(y)v||-||f(x)v||\\
&=(\varepsilon_{k-1}/2)|f_k(y)v_k|-(|f_{[k-1]}(x)v_{[k-1]}|+(\varepsilon_{k-1}/2)|f_k(x)v_k|)\\
&\le(\varepsilon_{k-1}/2)|f_k(x)f_k(y)|-\varepsilon_{k-1}|xy|\\
&\le-(\varepsilon_{k-1}/2)|xy|,
\end{align*}
where the second inequality follows from the triangle inequality and the inductive assumption \eqref{eq:openind}, and the third uses the fact that $f_k$ is $1$-Lipschitz.
This finishes the case (1).

Suppose (2) holds.
By the base case, Step 1, we find $y\in X$ such that
\begin{equation}\label{eq:openbase}
f_k(y)=v_k,\quad(1-\varkappa(\delta))|xy|\le|f_k(x)v_k|.
\end{equation}
Hence
\begin{align*}
&||f(y)v||-||f(x)v||\\
&=|f_{[k-1]}(y)v_{[k-1]}|-(|f_{[k-1]}(x)v_{[k-1]}|+(\varepsilon_{k-1}/2)|f_k(x)v_k|)\\
&\le|f_{[k-1]}(y)f_{[k-1]}(x)|-(\varepsilon_{k-1}/3)|xy|,
\end{align*}
where the second inequality follows from the triangle inequality and the base case \eqref{eq:openbase} (we may assume $\varkappa(\delta)\approx0$).
Furthermore, by the almost orthogonality, Definition \ref{dfn:kstr}(3) (and Remark \ref{rem:orth}), we have
\[|f_{[k-1]}(y)f_{[k-1]}(x)|<\varkappa_k(\delta)|xy|,\]
where $\varkappa_k(\delta)$ depends additionally on $k$.
This is because the shortest path $xy$ is part of $p_kx$ or $\bar p_k x$ (see Step 1) and by \eqref{eq:openbase} it is contained in a metric ball of straining radius (see Definition \ref{dfn:rad}), provided $v$ is sufficiently close to $f(x)$.
Therefore if $\delta$ is sufficiently small, say $\delta<\delta_k\ll\delta_{k-1}$, we obtain
\[||f(y)v||-||f(x)v||\le-(\varepsilon_{k-1}/4)|xy|.\]
This finishes the case (2).

By Lemma \ref{lem:open}, the map $f$ is $(\varepsilon_{k-1}/4)$-open with respect to the new norm $||\cdot||$.
Finally, replacing this norm with the standard one, we conclude that $f$ is $\varepsilon_k$-open in the Euclidean norm for some $\varepsilon_k\ll\varepsilon_{k-1}$.
This completes the proof.
\end{proof}

\begin{rem}\label{rem:open}
In the above proof, we only require the almost orthogonality (Definition \ref{dfn:kstr}(3) or Proposition \ref{prop:orth}) and do not need the almost nonbranching (Definition \ref{dfn:kstr}(2) or Proposition \ref{prop:bran}).
Indeed, geodesic completeness alone suffices to obtain the $1$-openness of a $1$-strainer map in Step 1.
\end{rem}

We next prove the local uniform contractibility of the fibers of a strainer map, the second half of Theorem \ref{thm:str}.
Since a GNPC space is locally contractible, it suffices to construct a local uniform retraction to the fiber of a strainer map.
The construction is a refinement of the previous proof of the openness.
Here, we do require the almost nonbranching property, in addition to the almost orthogonality (cf.\ Remark \ref{rem:open}).
We also need the notion of straining radius, Definition \ref{dfn:rad}, to guarantee the uniformity of the retraction.
Compare the following with \cite[Theorem 9.1]{LN:geo}, and see also Remarks \ref{rem:retlen} and \ref{rem:retball}.

\begin{prop}\label{prop:ret}
Let $X$ be a GNPC space, $B$ a tiny ball, and $p_1,\dots,p_k\in B$ a $(k,\delta)$-strainer at $x\in B$ with straining radius $>r_0>0$.
Let $f$ be the associated strainer map at $x$ and $\Pi=f^{-1}(f(x))$ the fiber of $f$ through $x$.
If $\delta<\delta_k$ and $0<r<\varepsilon_kr_0$, then there exists a continuous map
\[\Phi:B(x,r)\to B(x,\varepsilon_k^{-1}r)\cap\Pi\]
whose restriction to $B(x,r)\cap\Pi$ is the identity and
\begin{equation}\label{eq:ret}
|\Phi(y)y|\le\varepsilon_k^{-1}|f(x)f(y)|
\end{equation}
for any $y\in B(x,r)$, where $|\cdot|$ on the right-hand side denotes the Euclidean norm.
Here $\delta_k$ and $\varepsilon_k$ are positive constants depending only on $k$.
\end{prop}

\begin{rem}\label{rem:retrad}
Condition \eqref{eq:ret} is kind of the $\varepsilon_k$-openness of $f$ (cf.\ Proposition \ref{prop:open}).
Note that \eqref{eq:ret} implies $\Phi(y)\in B(x,{\varepsilon_k'}^{-1}|xy|)$ for some $\varepsilon_k'\ll\varepsilon_k$.
Therefore, to ensure that the image of $\Phi$ is contained in $B(x,\varepsilon_k^{-1}r)$, it suffices to show \eqref{eq:ret} and then replace $\varepsilon_k$ with a smaller one.
\end{rem}

Before the proof, we introduce the following notation.
We denote the \textit{(unit-speed) geodesic contraction} centered at $p\in B$ by
\[\Phi_p(x,t)\]
for $x\in B$ and $t\ge 0$, that is, $\Phi_p(x,t)$ is the unique unit-speed shortest path from $x$ to $p$ for $0\le t\le |px|$ and then stops at $p$ for $t\ge |px|$.

\begin{proof}
The proof is similar to that of \cite[Theorem 9.1]{LN:geo}, which is constructive and requires a more delicate argument than the proof of Proposition \ref{prop:open} based on Lemma \ref{lem:open}.
The desired retraction $\Phi$ is obtained by composing $\Phi_{p_1},\dots,\Phi_{p_k}$ and $\Phi_{\bar p_1},\dots,\Phi_{\bar p_k}$ infinitely many times, where $\bar p_1,\dots,\bar p_k$ is an opposite strainer.
As before, we have to take into account the asymmetry of the almost orthogonality.
We use induction on $k$.

\step{Step 1}
We first prove the base case $k=1$.
Since $p_1$ is a $(1,\delta)$-strainer at $x$ with straining radius $>r_0$, there exists an opposite strainer $\bar p_1$ such that
\[\angle p_1y\bar p_1>\pi-\delta\]
for any $y\in B(x,r_0)$ (see Definitions \ref{dfn:1str} and \ref{dfn:rad}).

Let $\Phi_{p_1}$ and $\Phi_{\bar p_1}$ be the geodesic contractions.
Clearly, $\Phi_{p_1}$ monotonically decreases the value of $f$ with velocity $1$.
Furthermore, the above strainer condition (and Remark \ref{rem:orth}) ensures that $\Phi_{\bar p_1}$ monotonically increases the value of $f$ with velocity almost $1$, i.e.,
\begin{equation}\label{eq:p1y}
\frac d{dt}f(\Phi_{\bar p_1}(y,t))>1-\varkappa(\delta)
\end{equation}
as long as $\Phi_{\bar p_1}(y,t)$ remains in $B(x,r_0)$.

Let $0<r<r_0/10$.
For $y\in B(x,r)$, we define $\Phi(y)$ as follows, where $\im$ denotes the image of a map:
\begin{itemize}
\item if $f(y)>f(x)$, then $\Phi(y)$ is the unique point of $\im\Phi_{p_1}(y,\cdot)\cap\Pi$.
\item if $f(y)<f(x)$, then $\Phi(y)$ is the unique point of $\im\Phi_{\bar p_1}(y,\cdot)|_{[0,5r]}\cap\Pi$.
\item if $f(y)=f(x)$, then $\Phi(y)=y$.
\end{itemize}

In the second case, since $f$ is $1$-Lipschitz, we have $f(y)>f(x)-r$.
Because $r<r_0/10$, the shortest path $\Phi_{\bar p_1}(y,t)$ stays within $B(x,r_0)$ for $0\le t\le 5r$, where the monotone estimate \eqref{eq:p1y} is valid.
Thus one can find the unique desired point, provided $1-\varkappa(\delta)\approx 1$.

The uniqueness guarantees that $\Phi$ is continuous.
It is now easy to see that $\Phi$ satisfies the other desired properties, by choosing $\delta_1$ so small that $\varkappa(\delta_1)\approx 0$ and setting $\varepsilon_1:=1/10$.

\step{Step 2}
We next prove the induction step.
Suppose we have constructed the desired retraction map $\Phi_{[k-1]}$ for a $(k-1,\delta)$-strainer $p_1,\dots,p_{k-1}$ at $x$ (in particular, we assume the existence of $\delta_{k-1}$ and $\varepsilon_{k-1}$).
Since $p_k$ is a $(1,\delta)$-strainer at $x$, we also have a retraction $\Phi_k$ constructed in the base case, Step 1.

To construct the desired retraction, we consider the modified norm
\[||v||:=|v_{[k-1]}|+(\varepsilon_{k-1}/2)|v_k|\]
as before, where $v=(v_{[k-1]},v_k)\in\mathbb R^{k-1}\times\mathbb R$.

We first set up the notation and summarize the assumptions.
Let $0<r<\varepsilon_kr_0$, where $\varepsilon_k\ll\varepsilon_{k-1}$ will be determined later so that the following argument is carried out in $B(x,r_0)$.
Let $y\in B(x,r)$ and set $z:=\Phi_{[k-1]}(y)$, $w:=\Phi_k(z)$, and $u:=f(x)$.
By the inductive assumption and the base case, we have
\begin{align}
f_{[k-1]}(z)&=u_{[k-1]},&\varepsilon_{k-1}|yz|&\le|f_{[k-1]}(y)u_{[k-1]}|,\label{eq:retind}\\
f_k(w)&=u_k,&(1-\varkappa(\delta))|zw|&\le|f_k(z)u_k|.\label{eq:retbase}
\end{align}
By the almost orthogonality, Definition \ref{dfn:kstr}(3) (and Remark \ref{rem:orth}), we have
\begin{equation}\label{eq:retorth}
|f_{[k-1]}(w)f_{[k-1]}(z)|<\varkappa_k(\delta)|zw|,
\end{equation}
assuming that the shortest path $zw$ is contained in $B(x,r_0)$ (recall $zw$ is part of the shortest path $p_kz$ or $\bar p_kz$; see Step 1).
Here $\varkappa_k(\delta)$ depends additionally on $k$.
The following argument depends only on the above properties.
See Figure \ref{fig:ret}.

\begin{figure}[ht]
\centering
\begin{tikzpicture}
\coordinate[label=below:{$p_1,\dots,p_{k-1}$}](p1)at(-0.5,-4);
\coordinate[label=above left:$p_k$](pk)at(-3,0);
\coordinate[label=above right:$\bar p_k$](pk')at(3,0);
\coordinate[label=right:$y$](y)at(1,1);

\draw[name path=f1] (-2,0) .. controls (-1,0.3) and (1,0.3) .. (2,0) node[below]{$f_{[k-1]}^{-1}(u_{[k-1]})$} coordinate[pos=.7] (z);
\draw[name path=fk] (0,2) .. controls (0.3,1) and (0.3,-1) .. (0,-2) node[below]{$f_k^{-1}(u_k)$} coordinate[pos=.5] (w);
\path[name intersections={of=f1 and fk,by=x}];

\fill (x) circle (1.5pt) node[above left]{$x$};
\fill (z) circle (1.5pt) node[above right] {$z$};
\fill (w) circle (1.5pt) node[left] {$w$};

\draw[->,dashed](y)to(z);
\draw[->,dashed](z)to(w);

\fill(p1)circle(1.5pt);
\fill(pk)circle(1.5pt);
\fill(pk')circle(1.5pt);
\fill(y)circle(1.5pt);
\end{tikzpicture}
\caption{}\label{fig:ret}
\end{figure}

\begin{clm}\label{clm:ret}
In the situation above, we have
\begin{align}
&||f(w)u||-||f(y)u||\le-\varepsilon_k|yw|,\label{eq:clm1}\\
&||f(w)u||\le(1/2)||f(y)u||,\label{eq:clm2}
\end{align}
where $\varepsilon_k\ll\varepsilon_{k-1}$ and $\delta<\delta_k\ll\delta_{k-1}$.
\end{clm}

\begin{proof}
We first prove \eqref{eq:clm1}.
This is essentially the $\varepsilon_k$-openness of $f$.
Therefore the following calculations are the same as in Step 2 of the proof of Proposition \ref{prop:open}.
First, we have
\begin{align*}
&||f(z)u||-||f(y)u||\\
&\le-|f_{[k-1]}(y)u_{[k-1]}|+(\varepsilon_{k-1}/2)|f_k(y)f_k(z)|\tag{triangle inequality, \eqref{eq:retind}}\\
&\le-\varepsilon_{k-1}|yz|+(\varepsilon_{k-1}/2)|yz|\tag{\eqref{eq:retind}, $1$-Lipschitzness of $f_k$}\\
&\le-(\varepsilon_{k-1}/2)|yz|.
\end{align*}
Second, we have
\begin{align*}
&||f(w)u||-||f(z)u||\\
&\le|f_{[k-1]}(z)f_{[k-1]}(w)|-(\varepsilon_{k-1}/2)|f_k(z)u_k|\tag{triangle inequality, \eqref{eq:retbase}}\\
&\le\varkappa_k(\delta)|zw|-(\varepsilon_{k-1}/3)|zw|\tag{\eqref{eq:retorth}, \eqref{eq:retbase}}\\
&\le-(\varepsilon_{k-1}/4)|zw|,
\end{align*}
provided $\delta$ is small enough, say $\delta<\delta_k\ll\delta_{k-1}$.
Combining the above two inequalities and using the triangle inequality, we obtain the desired inequality \eqref{eq:clm1} with $\varepsilon_k<\varepsilon_{k-1}/4$.

Next we prove \eqref{eq:clm2}.
By \eqref{eq:retind}, \eqref{eq:retbase}, and \eqref{eq:retorth}, we have
\[||f(w)u||=|f_{[k-1]}(w)f_{[k-1]}(z)|\le\varkappa_k(\delta)|zw|.\]
Furthermore, we have
\begin{align*}
|zw|&\le(1+\varkappa(\delta))|f_k(z)u_k|\tag{\eqref{eq:retbase}}\\
&\le2(|f_k(z)f_k(y)|+|f_k(y)u_k|)\tag{triangle inequality}\\
&\le2(|zy|+|f_k(y)u_k|)\tag{$1$-Lipschitzness of $f_k$}\\
&\le2(\varepsilon_{k-1}^{-1}|f_{[k-1]}(y)u_{[k-1]}|+|f_k(y)u_k|)\tag{\eqref{eq:retind}}\\
&\le\varepsilon_k^{-1}||f(y)u||,
\end{align*}
where the last inequality holds by choosing $\varepsilon_k\ll\varepsilon_{k-1}$.
Combining the above two inequalities and choosing $\delta_k$ small enough again, we obtain the desired inequality \eqref{eq:clm2}.
\end{proof}

Set $\Psi:=\Phi_k\circ\Phi_{[k-1]}$ on $B(x,r)$ (thus $w=\Psi(y)$).
We are going to define
\[\Phi(y):=\lim_{n\to\infty}\Psi^n(y)\]
for $y\in B(x,r)$, where $\Psi^n$ denotes the $n$-th iterate of $\Psi$.
By \eqref{eq:clm2}, we see that $\lim_{n\to\infty}f(\Psi^n(y))$ exists and is equal to $f(x)=u$.

Furthermore, for any $n\ge m$, we have
\begin{align*}
|\Psi^m(y)\Psi^n(y)|&\le|\Psi^m(y)\Psi^{m+1}(y)|+\dots+|\Psi^{n-1}(y)\Psi^n(y)|\\
&\le\varepsilon_k^{-1}(||f(\Psi^m(y))u||-||f(\Psi^n(y))u||)\tag{\eqref{eq:clm1}}\\
&\le\varepsilon_k^{-1}||f(\Psi^m(y))u||\\
&\le\varepsilon_k^{-1}2^{-m}||f(y)u||.\tag{\eqref{eq:clm2}}
\end{align*}
Therefore $\Psi^n(y)$ is a Cauchy sequence and the limit point $\Phi(y)$ certainly exists.
Moreover, since $||f(y)u||$ is uniformly bounded, this is a uniform convergence and thus $\Phi$ is continuous.
Taking $m=0$ in the above inequality, we also have
\[|\Psi^n(y)y|\le\varepsilon_k^{-1}||f(y)u||.\]
Taking $n\to\infty$, we obtain the desired inequality \eqref{eq:ret} (to be more precise, we have to change the modified norm back to the Euclidean norm and replace $\varepsilon_k$ with a smaller one).

Finally, we remark that the above argument is carried out in the metric ball $B(x,r_0)$ of straining radius so that all the inequalities are valid, assuming $r<\varepsilon_k r_0$.
Indeed, the last inequality shows that $\Psi^n(y)$ remains in $B(x,{\varepsilon_k'}^{-1}r)$ for some $\varepsilon_k'\ll\varepsilon_k$, as noted in Remark \ref{rem:retrad}.
Again, if necessary, replace $\varepsilon_k$ with a smaller one.
This completes the proof.
\end{proof}

\begin{rem}\label{rem:retlen}
As can be seen from the above proof, the retraction $\Phi$ is actually given by a homotopy as in the original statement \cite[Theorem 9.1]{LN:geo}: there exists a continuous map
\[\hat\Phi:B(x,r)\times[0,1]\to B(x,\varepsilon_k^{-1}r)\]
such that $\hat\Phi(\cdot,0)$ is the identity, $\hat\Phi(\cdot,1)=\Phi$, and $\hat\Phi$ fixes the points of $B(x,r)\cap\Pi$.
The homotopy $\hat\Phi(y,t)$ is defined by a suitable reparametrization of the infinite concatenation of the curves connecting $y$, $\Psi(y)$, \dots, $\Psi^n(y)$, \dots, $\Phi(y)$.
Here, $y$ and $\Psi(y)$ are connected by the concatenation of the curve from $y$ to $z=\Phi_{[k-1]}(y)$ given by the inductive assumption and the shortest path from $z$ to $w=\Phi_k(z)=\Psi(y)$ given by the base case.

Moreover, the following refinement of the inequality \eqref{eq:ret} holds (compare with \cite[Theorem 9.1(2)]{LN:geo}):
\[\length\hat\Phi(y,t)\le\varepsilon_k^{-1}|f(y)u|.\]
Indeed, to prove this, it suffices to just replace the distance between a point and its retraction image in the above proof by the length of the corresponding curve.
More precisely, by the inductive assumption, one can replace $|yz|$ in \eqref{eq:retind} by the length of the $\hat\Phi_{[k-1]}$-curve from $y$ to $z$.
Then one can prove \eqref{eq:clm1} with $|yw|$ replaced by the length of the concatenation of the $\hat\Phi_{[k-1]}$-curve from $y$ to $z$ and the $\hat\Phi_k$-curve from $z$ to $w$, i.e., the shortest path $zw$.
In this way, one can replace $|\Psi^n(y)y|$ with the length of the concatenation of the curves connecting $y$, $\Psi(y)$, \dots, $\Psi^n(y)$ and get the above estimate.

Though not used here, we will need this refinement in the next section to prove a counterpart to Proposition \ref{prop:ret} for an extended strainer (Proposition \ref{prop:exret}). 
\end{rem}

\begin{rem}\label{rem:retball}
On the other hand, unlike the GCBA case, it seems impossible to expect that the image of the retraction $\Phi$ remains in the original ball $B(x,r)$ (compare with \cite[Theorem 9.1(3)]{LN:geo}).
We deal with this problem by complementing a strainer in the proof of Theorem \ref{thm:fib} below.
See also Corollary \ref{cor:fib}.
\end{rem}

Recall that every metric ball in a tiny ball of a GNPC space is contractible via the geodesic contraction.
By composing this contraction with the retraction to the fiber of a strainer map, as constructed above, we obtain the following corollary.

\begin{cor}\label{cor:ret}
Under the setup of Proposition \ref{prop:ret}, $B(x,r)\cap\Pi$ is contractible in $B(x,\varepsilon_k^{-1}r)\cap\Pi$.
\end{cor}

We are now ready to finish the proof of Theorem \ref{thm:str}.

\begin{proof}[Proof of Theorem \ref{thm:str}]
Recall that the (topological) openness of a strainer map was proved in Proposition \ref{prop:open}. The local uniform contractibility of fibers follows from Corollary \ref{cor:ret}.
Indeed, by Remark \ref{rem:rad}, the straining radius for the strainer map $f$ is locally uniformly bounded below around $x$.
As stated in Proposition \ref{prop:ret}, the upper bound for radius $r$, needed to apply Corollary \ref{cor:ret}, depends only on the straining radius (and  $k$).
Furthermore, Corollary \ref{cor:ret} asserts that the part of the fiber of $f$ contained within a ball of some radius $r$ is contractible inside a larger, concentric ball whose radius is a multiple of $r$.
Therefore the condition of Definition \ref{dfn:fib} is satisfied.
\end{proof}

Finally, we derive Theorem \ref{thm:fib} from Theorem \ref{thm:str}. The following is a global version of Theorem~\ref{thm:fib}, which follows directly from Theorems~\ref{thm:str} and \ref{thm:ung1}.

\begin{cor}\label{cor:glofib}
Let $X$ be a GNPC space and $B$ a tiny ball.
Suppose $f:B\to\mathbb R^k$ is a $(k,\delta)$-strainer map on an open subset $U\subset B$.
Let $K\subset f(U)$ be a compact ANR (such as a compact convex set) such that $f^{-1}(K)\cap U$ is compact.
Then the restriction of $f$ to $f^{-1}(K)\cap U$ is a Hurewicz fibration, provided $\delta<\delta_k$.
\end{cor}

We now prove Theorem \ref{thm:fib}, using Theorem \ref{thm:ung2}.
The idea of the proof is to complement the given strainer map by adding one more distance function, coming from Proposition \ref{prop:kstr}, to control the local contractibility of fibers (cf.\ \cite[Proof of Theorem 1.7]{Fu:nc}).

\begin{proof}[Proof of Theorem \ref{thm:fib}]
Let $f$ be a $(k,\delta)$-strainer map at $x\in X$.
We construct an arbitrarily small contractible open neighborhood $V$ of $x$ such that $f|_V:V\to f(V)$ is a Hurewicz fibration with contractible fibers.
Note that if $f(V)$ is contractible, then the contractibility of $V$ follows from the other properties.

We want to apply Theorem \ref{thm:ung2} to $f|_V$.
By Theorem \ref{thm:str}, if $V$ is small enough, then $f|_V$ is open and has locally uniformly contractible fibers.
Hence it suffices to choose small $V$ such that the fibers of $f|_V$ (and $f(V)$) are contractible.

Suppose $r>0$ is sufficiently small.
By Proposition \ref{prop:kstr},
\[F:=(f,d(x,\cdot))\]
is a $(k+1,\delta)$-strainer map at any point of $f^{-1}(f(x))\cap\bar B(x,r)\setminus\{x\}$.
By Proposition \ref{prop:open}, $F$ is an open map in a small neighborhood of $f^{-1}(f(x))\cap\bar B(x,r)\setminus B(x,\varepsilon_kr/2)$, where $\varepsilon_k$ is a constant from Proposition \ref{prop:ret}.
Therefore, by Corollary \ref{cor:glofib}, there exists a compact neighborhood $K$ of $f(x)$ such that the restriction of $F$ to $f^{-1}(K)\cap\bar B(x,r)\setminus B(x,\varepsilon_kr/2)$ is a Hurewicz fibration.

We define
\[V:=f^{-1}(\mathring K)\cap B(x,r),\]
where $\mathring K$ denotes the interior of $K$.
We may assume that $f(V)=\mathring K$ is contractible.
We show that the fibers of $f|_V$ are contractible.
See Figure \ref{fig:fib}.

\begin{figure}[ht]
\centering
\begin{tikzpicture}
\coordinate[label=below:$x$](x)at(0,0);

\draw(x)circle[radius=0.75];
\draw(x)circle[radius=2];
\draw(-3,0.1) -- (3,0.1) node[right]{$f^{-1}(v)$};
\draw(-3,0.4)--(3,0.4);
\draw(-3,-0.4)--(3,-0.4);
\draw[<->] (-3.2,0.4)--(-3.2,-0.4) node[midway, left] {$f^{-1}(K)$};

\node[below right]at(-45:0.75){$\varepsilon_kr/2$};
\node[below right]at(-45:2){$r$};
\node at(1.25,0.75){$V$};

\fill (x) circle (1.5pt);
\fill[pattern={Lines[angle=45,distance=6pt,line width=0.4pt]}] (-2,-0.4) rectangle (2,0.4);
\end{tikzpicture}
\caption{}\label{fig:fib}
\end{figure}

Let $v\in\mathring K$.
Since $F$ is a Hurewicz fibration on $f^{-1}(K)\cap\bar B(x,r)\setminus B(x,\varepsilon_kr/2)$, its component $d(x,\cdot)$ is a Hurewicz fibration on $f^{-1}(v)\cap\bar B(x,r)\setminus B(x,\varepsilon_kr/2)$.
By lifting a homotopy crushing $[\varepsilon_kr/2,r)$ to $\{\varepsilon_kr/2\}$, we get a homotopy pushing $f^{-1}(v)\cap B(x,r)$ into $f^{-1}(v)\cap\bar B(x,\varepsilon_kr/2)$ (and fixing $f^{-1}(v)\cap\bar B(x,\varepsilon_kr/2)$).
By Corollary \ref{cor:ret} (and Remark \ref{rem:rad}), $f^{-1}(v)\cap\bar B(x,\varepsilon_kr/2)$ is contractible in $f^{-1}(v)\cap B(x,r)$.
Combining these two homotopies, we obtain a contraction of $f^{-1}(v)\cap B(x,r)$ in itself.
This completes the proof.
\end{proof}

\begin{rem}\label{rem:isol}
The above argument works even when $x$ is an isolated point of the fiber of $f$, i.e.,
\[f^{-1}(f(x))\cap B(x,2r)\setminus\{x\}=\emptyset.\]
In that case, by choosing $K$ small enough, we see  $f^{-1}(K)\cap\bar B(x,r)\setminus B(x,\varepsilon_kr/2)=\emptyset$ (otherwise, by the $\varepsilon_k$-openness of $f$ (cf.\ Remark \ref{rem:openrad}), we get a contradiction).
Therefore, it is clear from Corollary \ref{cor:ret} (and Remark \ref{rem:rad}) that $f^{-1}(v)\cap B(x,r)$ is contractible in itself for any $v\in\mathring K$.
Note that, however, we do not know whether the other fiber $f^{-1}(v)$ is discrete inside $B(x,\varepsilon_kr/2)$, where $v\neq f(x)$.
Compare with Theorem \ref{thm:bilip} below.
\end{rem}

The next result follows directly from the proof above.

\begin{cor}\label{cor:fib}
Let $X$ be a GNPC space and $f$ be a $(k,\delta)$-strainer map at $x\in X$.
Then for any sufficiently small $r>0$ (depending on $x$), the restriction of $d(x,\cdot)$ to $f^{-1}(f(x))\cap B(x,r)\setminus\{x\}$ is a Hurewicz fibration, provided $\delta<\delta_k$.
In particular, $f^{-1}(f(x))\cap B(x,r)$ is contractible.
\end{cor}

\begin{proof}
The last statement is already proved in the argument for Theorem \ref{thm:fib}.
Furthermore, that argument shows $d(x,\cdot)$ is a Hurewicz fibration on $f^{-1}(f(x))\cap B(x,r)\setminus B(x,s)$ for any $0<s<r$.
Since the homotopy lifting property is indeed a local condition on the base space (see \cite[Chapter XX, Corollary 3.6]{Du}), the proof follows.
\end{proof}

\subsection{Consequences}\label{sec:strcon}

In this subsection,  we apply the properties of strainer maps to prove Theorems \ref{thm:mfd} and \ref{thm:link}.

Theorem \ref{thm:link} follows from the special case $k=0$ of Corollary \ref{cor:fib}.
Here, by convention, any constant map to $\mathbb R^0$ is a \textit{$(0,\delta)$-strainer map} for any $\delta>0$.
Hence the proof of Corollary \ref{cor:fib} is valid for $k=0$ (specifically, by Proposition \ref{prop:1str} and Corollary \ref{cor:glofib}).

\begin{proof}[Proof of Theorem \ref{thm:link}]
Let $X$ be a GNPC space and $p\in X$.
By the $k=0$ case of Corollary \ref{cor:fib}, $d(p,\cdot)$ is a Hurewicz fibration on $B(p,r)\setminus\{p\}$, provided $r$ is sufficiently small.
Since the fibers of a Hurewicz fibration are homotopy equivalent, any two small metric spheres around $p$ are homotopy equivalent.
Furthermore, by lifting the contraction of $(0,r)$, we see that any small punctured ball around $p$ is homotopy equivalent to a metric sphere.
\end{proof}

Next, we prove Theorem \ref{thm:mfd}.
The following notion plays a key role.

\begin{dfn}\label{dfn:comp}
Let $X$ be a GNPC space and $f$ be a $(k,\delta)$-strainer map at $x\in X$.
We say that $f$ is \textit{complementable} at $x$ if there exists $p\in X$ such that $(f,d(p,\cdot))$ is a $(k+1,\delta)$-strainer map at $x$.
\end{dfn}

The next lemma follows immediately from the finite dimensionality of a tiny ball (Proposition \ref{prop:dim}) and the openness of a strainer map (Proposition \ref{prop:open}).
This tells us that any strainer map will eventually become non-complementable.

\begin{lem}\label{lem:num}
Let $X$ be a GNPC space and $B$ a tiny ball.
Then there exists a natural number $k=k(B)$ such that there is no $(k+1,\delta)$-strainer at any $x\in B$, where $\delta<\delta_k$.
\end{lem}

We now prove another main result for a strainer map.

\begin{thm}\label{thm:bilip}
Let $X$ be a GNPC space and $f$ be a $(k,\delta)$-strainer map at $x\in X$, where $\delta<\delta_k$.
Suppose $f$ is not complementable at any point near $x$.
Then $f$ is an open bi-Lipschitz embedding on an open neighborhood of $x$.
\end{thm}

\begin{proof}
Since $f$ is $\sqrt k$-Lipschitz and $\varepsilon_k$-open, it suffices to show that $f$ is injective in a neighborhood $V$ of $x$.
Indeed, suppose $y,z\in B(x,r)$, where $r$ is small enough.
Since $f$ is $\sqrt k$-Lipschitz, $|f(y)f(z)|\le2\sqrt kr$.
By the $\varepsilon_k$-openness of $f$, Proposition \ref{prop:open} (cf.\ Remark \ref{rem:openrad}), one can find $z'\in X$ such that 
\[f(z')=f(z),\quad\varepsilon_k|yz'|\le|f(y)f(z)|.\]
The latter implies $|yz'|\le2\varepsilon_k^{-1}\sqrt kr$.
Therefore, if $r$ is small enough compared to $V$, we have $z=z'$, as desired.

To see the injectivity of $f$, let $V$ be a neighborhood of Theorem \ref{thm:fib} on which $f$ has contractible fibers (cf.\ Remark \ref{rem:isol}).
In particular, each fiber of $f|_V$ is path-connected.
We may assume that $f$ is not complementable on $V$.
Given any $y\in V$, we claim that $f^{-1}(f(y))\cap V$ is a singleton.
Otherwise, by the path-connectedness, there is a point of $f^{-1}(f(y))\cap V$ arbitrarily close to $y$.
By Proposition \ref{prop:kstr}, this would imply that $f$ is complementable at a point near $y$, which is a contradiction.
\end{proof}

\begin{rem}\label{rem:isom}
In the GCBA or Alexandrov case, $f$ is a $\varkappa(\delta)$-almost isometry to a Euclidean domain, i.e., a bi-Lipschitz homeomorphism with Lipschitz constants $1\pm\varkappa(\delta)$ (\cite[Corollary 11.2]{LN:geo}, \cite[Theorem 9.4]{BGP}).
Since our GNPC spaces are modeled on (strictly convex) normed spaces, we cannot expect such an improvement of the Lipschitz constants.
However, see also Problems \ref{prob:tan}, \ref{prob:ber}, and \ref{prob:fin}.
\end{rem}

\begin{rem}\label{rem:dich}
The above theorem is still a little weaker than the dichotomy stated in \cite[Section 9.2, Corollary 11.2]{LN:geo}.
The technical reason behind this difference is that the closeness of $x$ to $q$ in Proposition \ref{prop:orth} depends on $q$.
In the GCBA (or Alexandrov) case, if $p$ is a $1$-strainer at $q$, then it depends only on the straining radius, independent of $q$ (cf.\ Remark \ref{rem:rad}).
This also seems related to the above-mentioned problems in Section \ref{sec:prob}.
\end{rem}

We now prove Theorem \ref{thm:mfd}.
Let $X$ be a GNPC space and $p\in X$.
For $\delta>0$, the \textit{local $\delta$-strainer number} at $p$ is the supremum of $k$ for which each neighborhood of $p$ contains a $(k,\delta)$-strained point (i.e., a point with a $(k,\delta)$-strainer).
By Lemma \ref{lem:num}, this local number is finite for $\delta<\delta_p$, where $\delta_p$ depends on $p$.

\begin{proof}[Proof of Theorem \ref{thm:mfd}]
Let $X$ be a GNPC space and $p\in X$ an arbitrary point.
Let $k=k(p,\delta)$ be the local $\delta$-strainer number at $p$, where $\delta<\delta_p$.
By definition, there exists a $(k,\delta)$-strained point $x$ arbitrarily close to $p$.
By Theorem \ref{thm:bilip}, $x$ is a $k$-manifold point.
Therefore the set $M$ of manifold points in $X$ (of different dimensions) is dense in $X$.
Clearly $M$ is open.
This completes the proof.
\end{proof}

\section{Extended strainers}\label{sec:exstr}

In this section, we modify the definition of a strainer to introduce an extended strainer (cf.\ \cite[Section 5]{LNS}).
Roughly speaking, we add a base point $p_0$ to a strainer $p_1,\dots,p_k$ such that $p_1,\dots,p_k$ are contained in a metric sphere centered at $p_0$, which may be arbitrarily large.
This modification is necessary for dealing with a global problem like Theorem \ref{thm:four}.

The contents of this section will be used only in Section \ref{sec:prfglo}, so one may skip this section until then.
However, since this section is organized in parallel with the previous section, some readers may prefer to read it immediately after reading Section \ref{sec:str}.

The organization of this section is as follows.
In Section \ref{sec:exdfn}, we define extended strainers and show their existence.
In Section \ref{sec:exfib}, we prove the fibration property of extended strainers.
In Section \ref{sec:excon}, as consequences, we prove some propositions which will be used in the proof of Theorem \ref{thm:four}.

\subsection{Definition and existence}\label{sec:exdfn}

In this subsection, we define extended strainers and show their existence.
Throughout this subsection, $X$ denotes a GNPC space, $B$ a tiny ball, and $\delta$ a small positive number (see Notation \ref{nota:del}).
Let $k\ge 1$.

\begin{dfn}\label{dfn:exstr}
A sequence of points $p_0,p_1,\dots,p_k\in B$ is called an \textit{extended $(k,\delta)$-strainer} at $x\in B$ if it satisfies the following conditions:
\begin{enumerate}
\item $p_1,\dots,p_k$ is a $(k,\delta)$-strainer at $x$ in the sense of Definition \ref{dfn:kstr}.
\item there exists a neighborhood $U\subset B$ of $x$ such that for any $y\in U$, we have
\[|\angle p_0yp_i-\pi/2|<\delta,\quad|\angle p_0y\bar p_i-\pi/2|<\delta\]
for any $1\le i\le k$, where $\bar p_i$ is an opposite strainer at $x$ for $p_i$.
\end{enumerate}
We call $p_0$ the \textit{base point} of this extended strainer.
\end{dfn}

The difference of the above definition from the standard $(k+1)$-strainer is that $p_0$ is not assumed to be a $1$-strainer at $x$: there is no opposite strainer $\bar p_0$ for $p_0$.
In particular, there is no restriction on the distance between $p_0$ and $x$, as we will see below (recall that for the existence of a $1$-strainer, we have to consider a small neighborhood of a given point; see Proposition \ref{prop:1str}).
The problem of the asymmetry of almost orthogonality appears here again for $p_0$ and $p_i$.

Similar to before, we introduce

\begin{dfn}\label{dfn:exrad}
Let $p_0,p_1,\dots,p_k\in B$ be an extended $(k,\delta)$-strainer at $x\in B$.
We say that $p_0,p_1,\dots,p_k$ has \textit{straining radius} $>r$ at $x$ if it is an extended $(k,\delta)$-strainer for any $y\in B(x,r)$ with the same opposite strainer $\bar p_1,\dots,\bar p_k$.
\end{dfn}

Clearly, we have

\begin{rem}\label{rem:exrad}
Let $p_0,p_1,\dots,p_k\in B$ be an extended $(k,\delta)$-strainer at $x\in B$ with straining radius $>r$ at $x$.
Then it has straining radius $>r/2$ at any $y\in B(x,r/2)$.
\end{rem}

We show the existence of an extended strainer.
For convenience, we say that $p\in B$ is an \textit{extended $(0,\delta)$-strainer} for any $x\in B\setminus\{p\}$ and $\delta>0$ (there is no condition, and thus it exists).
Compare the following with \cite[Lemma 5.3]{LNS}.

\begin{prop}\label{prop:exstr}
Let $p_0,p_1,\dots,p_{k-1}\in B$ be an extended $(k-1,\delta)$-strainer at $p_k\in B$, where $k\ge 1$.
Suppose $x\in B\setminus\{p_k\}$ is sufficiently close to $p_k$ and has the same distances from $p_i$'s as $p_k$, i.e.,
\[|p_ip_k|=|p_ix|\]
for any $0\le i\le k-1$.
Then $p_0,p_1,\dots,p_k$ is an extended $(k,\delta)$-strainer at $x$.
\end{prop}

\begin{proof}
We check the two conditions of Definition \ref{dfn:exstr}.
The first condition follows from Propositions \ref{prop:1str} and \ref{prop:kstr} (in the cases $k=1$ and $k\ge 2$, respectively).
The second condition for $1\le i\le k-1$ is trivial from the openness of an extended strainer (cf.\ Remark \ref{rem:exrad}).
The $i=k$ case follows from exactly the same argument as in the proof of Proposition \ref{prop:kstr}.
Here we need the fact that Proposition \ref{prop:orth} holds globally, as mentioned in Remark \ref{rem:orthglo}.
This completes the proof.
\end{proof}

\subsection{Fibration property}\label{sec:exfib}

In this subsection, we establish the fibration property of extended strainers, which will be used to prove Theorem \ref{thm:four}.
The main object of this subsection is the following half strainer map.
See also Remark \ref{rem:exstrm} below.

\begin{dfn}\label{dfn:exstrm}
Let $X$ be a GNPC space, $B$ a tiny ball, and $p_0,p_1,\dots,p_k\in B$ an extended $(k,\delta)$-strainer at $x\in B$.
Let
\[f=(d(p_1,\cdot),\dots,d(p_k,\cdot))\]
be a $(k,\delta)$-strainer map at $x$ associated with $p_1,\dots,p_k$.
We call the restriction
\[f|_{\bar B(p_0,R)},\quad R:=|p_0x|\]
a \textit{half $(k,\delta)$-strainer map} at $x$ associated with $p_0,p_1,\dots,p_k$.
We also call $p_0$ the \textit{base point} of this half strainer map.
See Figure \ref{fig:exstrm}.
\end{dfn}

\begin{figure}[ht]
\centering
\begin{tikzpicture}
\coordinate[label=above left:$x$](x)at(0,0);
\coordinate[label=below:$p_0$](p0)at(0,-4);
\coordinate[label=below left:{$p_1,\dots,p_k$}](p1)at(-0.75,-1.5);
\coordinate[label=above right:{$\bar p_1,\dots,\bar p_k$}](p1')at(0.5,1);

\draw(-3,0)--(3,0) node[right]{$\partial B(p_0,R)$};
\draw[->](-4.5,-1.5)--(-3.25,1) node[above right]{$f$};
\draw[dashed](x)--(p0);
\draw[dashed](p1)--(x)--(p1');

\fill (x) circle (1.5pt);
\fill (p0) circle (1.5pt);
\fill (p1) circle (1.5pt);
\fill (p1') circle (1.5pt);
\fill(-3.75,0) circle (1.5pt) node[left]{$f(x)$};
\end{tikzpicture}
\caption{}\label{fig:exstrm}
\end{figure}

\begin{rem}\label{rem:exstrm}
This terminology differs from Lytchak--Nagano--Stadler \cite{LNS}, who called the map $(f,d(p_0,\cdot))$ an \textit{extended $(k,\delta)$-strainer map} and did not introduce the half strainer map as above.
The reason for this modification is to emphasize the parallels with the previous section in the following theorems.
See also Remark \ref{rem:exfib} below.
\end{rem}

The following two theorems are the main results of this subsection.
These are counterparts to Theorems \ref{thm:str} and \ref{thm:fib} for standard strainer maps, respectively, basically saying that the same results hold for half strainer maps.
As before, let $X$ be a GNPC space and $\delta_k$ a positive constant depending only on $k$ (see Notation \ref{nota:k}).

\begin{thm}\label{thm:exstr}
Let $f|_{\bar B(p_0,R)}$ be a half $(k,\delta)$-strainer map at $x\in\partial B(p_0,R)$ with base point $p_0$, where $\delta<\delta_k$.
Then $f|_{\bar B(p_0,R)}$ is open and has locally uniformly contractible fibers in an open neighborhood of $x$ in $\bar B(p_0,R)$.
\end{thm}

\begin{thm}\label{thm:exfib}
Let $f|_{\bar B(p_0,R)}$ be a half $(k,\delta)$-strainer map at $x\in\partial B(p_0,R)$ with base point $p_0$, where $\delta<\delta_k$.
Then there exists an arbitrarily small contractible open neighborhood $V$ of $x$ in $\bar B(p_0,R)$ such that $f|_V$ is a Hurewicz fibration with contractible fibers.
\end{thm}

The underlying idea in the proofs is straightforward: we use the retraction onto $\bar B(p_0,R)$ along the shortest paths to $p_0$ so that the desired properties of $f$ still hold when restricted to $\bar B(p_0,R)$.
However, unlike the GCBA case \cite{LNS}, this retraction is not ``almost orthogonal'' to the strainer $p_1,\dots,p_k$, due to the asymmetry issue.
More precisely, if we move $x$ toward $p_0$, the distances from $p_1,\dots,p_k$ may change significantly.
Instead, by Definition \ref{dfn:exstr}(2), $p_1,\dots,p_k$ are almost orthogonal to this retraction, in the sense that moving $x$ toward $p_1,\dots,p_k$ or its opposite strainer $\bar p_1,\dots,\bar p_k$ leaves the distance from $p_0$ almost unchanged.
This is enough to prove the desired properties of the half strainer map $f|_{\bar B(p_0,R)}$.

\begin{rem}\label{rem:exfib}
Unlike Theorems \ref{thm:str} and \ref{thm:fib}, we cannot find exact counterparts to the above theorems in the original paper of Lytchak--Nagano--Stadler \cite{LNS} (partial counterparts are \cite[Lemma 5.1, Proposition 7.1]{LNS}).
These results represent an additional contribution of this paper.
Later we derive from the above theorems the same propositions as in the original paper: see Propositions \ref{prop:hemi}, \ref{prop:neg}, and \ref{prop:half}.
\end{rem}

We first prove Theorem \ref{thm:exstr} along the same lines as Theorem \ref{thm:str}.
The following is a counterpart to Proposition \ref{prop:open}.
Compare with \cite[Lemma 5.1]{LNS} (see also Remark \ref{rem:exopen} below).

\begin{prop}\label{prop:exopen}
Let $f|_{\bar B(p_0,R)}$ be a half $(k,\delta)$-strainer map at $x\in\partial B(p_0,R)$ with base point $p_0$.
If $\delta<\delta_k$, then $f|_{\bar B(p_0,R)}$ is $\varepsilon_k$-open in an open neighborhood of $x$ in $\bar B(p_0,R)$ with respect to the Euclidean norm of $\mathbb R^k$.
Here $\delta_k$ and $\varepsilon_k$ are positive constants depending only on $k$.
\end{prop}

\begin{proof}
The proof is a minor modification of that of Proposition \ref{prop:open}.
We check the condition \eqref{eq:open} of Lemma \ref{lem:open} for $f|_{\bar B(p_0,R)}$ at any point in $\bar B(p_0,R)$ near $x$.
Here we use the modified norm $||\cdot||$ introduced in the proof of Proposition \ref{prop:open}.
By Definition \ref{dfn:exstr}(1) and Proposition \ref{prop:open}, the condition \eqref{eq:open} is satisfied at any interior point in $B(p_0,R)$ near $x$.
Hence it remains to check the condition \eqref{eq:open} at any boundary point in $\partial B(p_0,R)$ near $x$.

By the openness of an extended strainer (cf.\ Remark \ref{rem:exrad}), it suffices to check the condition \eqref{eq:open} for $f|_{\bar B(p_0,R)}$ at $x$.
Suppose $v\in\mathbb R^k\setminus\{f(x)\}$ is sufficiently close to $f(x)$.
We show that there exists $y\in\bar B(p_0,R)\setminus\{x\}$ such that
\[||f(y)v||-||f(x)v||\le-\varepsilon_k|xy|\]
for some constant $\varepsilon_k>0$, provided $\delta<\delta_k$.

By the proof of Proposition \ref{prop:open}, there exists $y'\in X$ near $x$ such that
\begin{equation}\label{eq:fy'}
||f(y')v||-||f(x)v||\le-\varepsilon_k'|xy'|,
\end{equation}
provided $\delta<\delta_k'$, where $\delta_k'$ and $\varepsilon_k'$ are positive constants.
As can be seen from the proof, this point $y'$ is found by moving  $x$ a bit along the shortest path to $p_1,\dots,p_k$ or $\bar p_1,\dots,\bar p_k$.
By the almost orthogonality, Definition \ref{dfn:exstr}(2), we have
\[||p_0y'|-|p_0x||\le\varkappa(\delta)|xy'|.\]

If $y'\in\bar B(p_0,R)$, we are done.
Otherwise, we choose $y$ as the point on the shortest path $p_0y'$ such that $|p_0y|=R$.
By the previous inequality, we have
\begin{equation}\label{eq:yy'}
|yy'|\le\varkappa(\delta)|xy'|.
\end{equation}
Therefore
\begin{align*}
||f(y)v||-||f(x)v||&\le||f(y)f(y')||+||f(y')v||-||f(x)v||\tag{triangle inequality}\\ 
&\le2\sqrt k|yy'|-\varepsilon_k'|xy'|\tag{Lipschitzness of $f$, \eqref{eq:fy'}}\\
&\le-(\varepsilon_k'/2)|xy|,\tag{\eqref{eq:yy'}}
\end{align*}
provided $\delta$ is small enough, say $\delta<\delta_k$.
Setting $\varepsilon_k:=\varepsilon_k'/2$, we obtain the desired inequality \eqref{eq:open}.
This completes the proof.
\end{proof}

\begin{rem}\label{rem:exopen}
More strongly, as in \cite[Lemma 5.1]{LNS}, one can show the $\varepsilon_{k+1}$-openness of the ``extended strainer map'' $\hat f:=(d(p_0,\cdot),f)$ (cf.\ Remark \ref{rem:exstrm}; here we changed the order due to the asymmetry issue).
Indeed, $p_0,\dots,p_k$ still satisfies the almost orthogonality condition as in Definition \ref{dfn:kstr}, even though $p_0$ may fail the almost nonbranching condition.
Therefore, the proof of Proposition \ref{prop:open} applies to $\hat f$ (see also Remark \ref{rem:open}).
We leave the details to the reader.
\end{rem}

The following is a counterpart to Proposition \ref{prop:ret}.
Compare with \cite[Proposition 7.1]{LNS}.

\begin{prop}\label{prop:exret}
Let $f|_{\bar B(p_0,R)}$ be a half $(k,\delta)$-strainer map at $x\in\partial B(p_0,R)$ with base point $p_0$.
Suppose the corresponding extended strainer at $x$ has straining radius $>r_0>0$.
Let $\Pi_+$ be the fiber of $f|_{\bar B(p_0,R)}$ through $x$, i.e.,
\[\Pi_+=\Pi\cap\bar B(p_0,R),\]
where $\Pi$ is the fiber of $f$ through $x$ (in $X$).
If $\delta<\delta_k$ and $0<r<\varepsilon_kr_0$, then there exists a continuous map
\[\Phi_+:B(x,r)\cap\bar B(p_0,R)\to B(x,\varepsilon_k^{-1}r)\cap\Pi_+\]
such that the restriction of $\Phi_+$ to $\Pi_+\cap B(x,r)$ is the identity and
\begin{equation}\label{eq:exret}
|\Phi_+(y)y|\le\varepsilon_k^{-1}|f(x)f(y)|
\end{equation}
for any $y\in B(x,r)\cap\bar B(p_0,R)$, where $|\cdot|$ on the right-hand side denotes the Euclidean norm.
Here $\delta_k$ and $\varepsilon_k$ are positive constants depending only on $k$.
\end{prop}

\begin{rem}\label{rem:exret}
In the original notation of Lytchak--Nagano--Stadler \cite{LNS}, $\Pi_+$ should be written as $\hat\Pi^+$.
The reason for this minor change is to regard $\Pi_+$ as half of the fiber $\Pi$ of a standard strainer map, rather than as a subset bounded by the fiber $\hat\Pi$ of an extended strainer map as in \cite{LNS} (see Remark \ref{rem:exstrm}).
\end{rem}

\begin{rem}
In order to get a retraction map defined on the whole $B(x,r)$ as in the original statement \cite[Proposition 7.1]{LNS}, it suffices to compose $\Phi_+$ with the geodesic retraction $\phi$ to $\bar B(p_0,R)$ centered at $p_0$ (and replace $\varepsilon_k$ with a smaller one); see the proof below for the precise definition of $\phi$.
In this case, an additional term $\max\{|p_0y|-R,0\}$ should appear on the right-hand side of \eqref{eq:exret}.
\end{rem}

\begin{proof}
The proof is a minor modification of that of Proposition \ref{prop:ret}.
The modification is similar to the previous proof of Proposition \ref{prop:exopen}.

Suppose $r$ is sufficiently small, say $<\varepsilon_k r_0$, where $\varepsilon_k$ will be determined later.
By Definition \ref{dfn:exstr}(1) and Proposition \ref{prop:ret}, there exists a continuous map
\[\Phi:B(x,r)\to B(x,\varepsilon_k^{-1}r)\cap\Pi\]
such that the restriction of $\Phi$ to $B(x,r)\cap\Pi$ is the identity.
Moreover, we have
\begin{equation}\label{eq:tilde}
\widetilde{|\Phi(y)y|}\le\varepsilon_k^{-1}|f(y)u|
\end{equation}
for any $y\in B(x,r)$, where $u=f(x)$ and $\widetilde{|\Phi(y)y|}$ denotes the length of the curve from $y$ to $\Phi(y)$ given by the homotopy $\hat\Phi$ defining $\Phi$ (see Remark \ref{rem:retlen}).

As can be seen from the proof of Proposition \ref{prop:ret}, this $\Phi$ is constructed by moving $y$ along the shortest paths to $p_1,\dots,p_k$ and $\bar p_1,\dots,\bar p_k$ infinitely many times.
Therefore, by the almost orthogonality, Definition \ref{dfn:exstr}(2), we have

\begin{clm}\label{clm:exorth}
\[||p_0\Phi(y)|-|p_0y||\le\varkappa_k(\delta)\widetilde{|\Phi(y)y|},\]
where $\varkappa_k(\delta)$ depends additionally on $k$ (see Notation \ref{nota:k}).
\end{clm}

This seems rather trivial, but to prove it we need to go back to the original proof of Proposition \ref{prop:ret}.

\begin{proof}
We prove it by induction on $k$.
More precisely, we insert this claim into the inductive construction of $\Phi$ (and $\hat\Phi$ in Remark \ref{rem:retlen}) in the proof of Proposition \ref{prop:ret}.
In what follows, we denote by $\widetilde{|ab|}$ the length of the curve connecting $a$ and $b$ given by the homotopy under consideration.
The precise meaning will be clear below.

The base case $k=1$ immediately follows from the definition of $\Phi$ and the almost orthogonality, Definition \ref{dfn:exstr}(2).
Indeed, $\Phi$ is defined simply by the shortest path from $y$ to $p_1$ or $\bar p_1$.

Suppose the claim holds for half $(k-1,\delta)$-strainer maps.
Recall $\Psi=\Phi_k\circ\Phi_{[k-1]}$, where we are using the notation from the proof of Proposition \ref{prop:ret}.
Then
\begin{align*}
||p_0\Psi(y)|-|p_0y||&\le||p_0\Psi(y)|-|p_0\Phi_{[k-1]}(y)||+||p_0\Phi_{[k-1]}(y)|-|p_0y||\\
&\le\varkappa(\delta)|\Psi(y)\Phi_{[k-1]}(y)|+\varkappa_{k-1}(\delta)\widetilde{|\Phi_{[k-1]}(y)y|}\\
&\le\varkappa_k(\delta)\widetilde{|\Psi(y)y|}.
\end{align*}
Here the second inequality follows from the base case and the inductive assumption applied to $\Phi_k$ and $\Phi_{[k-1]}$, respectively.
The third inequality may seem to contradict the triangle inequality, but it is true because we are considering the length of the curve instead of the distance between its endpoints (and this is why we need \eqref{eq:tilde} instead of \eqref{eq:ret} later).

Since $\Phi=\lim_{n\to\infty}\Psi^n$, we obtain
\begin{align*}
||p_0\Phi(y)|-|p_0y||&\le\sum_{n=1}^\infty||p_0\Psi^n(y)|-|p_0\Psi^{n-1}(y)||\\
&\le\sum_{n=1}^\infty\varkappa_k(\delta)\widetilde{|\Psi^n(y)\Psi^{n-1}(y)|}\\
&=\varkappa_k(\delta)\widetilde{|\Phi(y)y|},
\end{align*}
where the last equality holds by the same reasoning above.
\end{proof}

For simplicity, assume $\bar B(p_0,R)$ is entirely contained in a tiny ball $B$ (the general case is similar since our argument is local around $x$).
Let us denote by
\[\phi:B\to\bar B(p_0,R)\]
the geodesic retraction of $B$ to $\bar B(p_0,R)$, i.e., $\phi(y)$ is the closest point to $y$ on the intersection of the shortest path $p_0y$ and $\bar B(p_0,R)$.
We consider the composition
\[\Psi_+:=\phi\circ\Phi.\]

Now suppose $y\in B(x,r)\cap\bar B(p_0,R)$.
Clearly $\Psi_+(y)\in\bar B(p_0,R)$.
Furthermore, Claim \ref{clm:exorth} together with the assumption $y\in\bar B(p_0,R)$ implies
\begin{equation}\label{eq:psi+}
|\Psi_+(y)\Phi(y)|\le\varkappa_k(\delta)\widetilde{|\Phi(y)y|}.
\end{equation}
See Figure \ref{fig:exret}.

\begin{figure}[ht]
\centering
\begin{tikzpicture}
\coordinate[label=below:{$p_0$}](p0)at(-0.5,-4);
\coordinate[label=above left:{$p_1,\dots,p_k$}](p1)at(-4,0);
\coordinate[label=above right:{$\bar p_1,\dots,\bar p_k$}](p1')at(4,0);
\coordinate[label=below:$y$](y)at(1.5,0);

\draw[name path=f1] (-3,0) -- (3,0) node[below]{$\partial B(p_0,R)$} coordinate[pos=.5] (y'');
\draw[name path=fk] (0,2) .. controls (0.3,1) and (0.3,-1) .. (0,-2) node[below]{$\Pi$} coordinate[pos=.4] (y');
\path[name intersections={of=f1 and fk,by=x}];

\fill (x) circle (1.5pt) node[below right]{$x$};
\fill (y') circle (1.5pt) node[left] {$\Phi(y)$};
\fill (y'') circle (1.5pt) node[below left] {$\Psi_+(y)$};

\draw[->,dashed](y)to(y');
\draw[->,dashed](y')to(y'');

\fill(p0)circle(1.5pt);
\fill(p1)circle(1.5pt);
\fill(p1')circle(1.5pt);
\fill(y)circle(1.5pt);
\end{tikzpicture}
\caption{}\label{fig:exret}
\end{figure}

Now we can prove the following counterpart to Claim \ref{clm:ret} (cf.\ Remark \ref{rem:retlen}).

\begin{clm}\label{clm:exret}
In the situation above, we have
\begin{align}
&|f(\Psi_+(y))u|-|f(y)u|\le-(\varepsilon_k/3)\widetilde{|\Psi_+(y)y|},\label{eq:exclm1}\\
&|f(\Psi_+(y))u|\le(1/2)|f(y)u|.\label{eq:exclm2}
\end{align}
\end{clm}

\begin{proof}
We first prove the second inequality \eqref{eq:exclm2}.
\begin{align*}
|f(\Psi_+(y))u|&=|f(\Psi_+(y))f(\Phi(y))|\\
&\le\varkappa_k(\delta)\widetilde{|\Phi(y)y|}\tag{Lipschitzness of $f$, \eqref{eq:psi+}}\\
&\le(1/2)|f(y)u|,\tag{\eqref{eq:tilde}}
\end{align*}
provided $\delta$ is small enough (say $<\delta_k$).

Next we prove the first inequality \eqref{eq:exclm1}, using the second inequality.
\begin{align*}
|f(\Psi_+(y))u|-|f(y)u|&\le -(1/2)|f(y)u|\tag{\eqref{eq:exclm2}}\\
&\le-(\varepsilon_k/2)\widetilde{|\Phi(y)y|}\tag{\eqref{eq:tilde}}\\
&\le-(\varepsilon_k/3)\widetilde{|\Psi_+(y)y|},\tag{\eqref{eq:psi+}}
\end{align*}
provided $\delta$ is small enough.
Here the last inequality is obtained as follows.
Since $\Psi_+=\phi\circ\Phi$ and $\phi$ is defined by the shortest path to $p_0$, we have
\[\widetilde{|\Psi_+(y)y|}=|\Psi_+(y)\Phi(y)|+\widetilde{|\Phi(y)y|}\le (1+\varkappa_k(\delta))\widetilde{|\Phi(y)y|},\]
where the second inequality follows from \eqref{eq:psi+}.
\end{proof}

The rest of the proof is exactly the same as that of Proposition \ref{prop:ret}.
We want to define $\Phi_+(y):=\lim_{n\to\infty}\Psi_+^n(y)$.
By \eqref{eq:exclm2}, we see that $\lim_{n\to\infty}f(\Psi_+^n(y))=u$.
Furthermore, we can show that $\Psi_+^n(y)$ is a uniformly converging Cauchy sequence.
More precisely, for any $n\ge m$, we have
\begin{align*}
\widetilde{|\Psi_+^m(y)\Psi_+^n(y)|}&=\widetilde{|\Psi_+^m(y)\Psi_+^{m+1}(y)|}+\dots+\widetilde{|\Psi_+^{n-1}(y)\Psi_+^n(y)|}\\
&\le3\varepsilon_k^{-1}|f(\Psi_+^m(y))u|\tag{\eqref{eq:exclm1}}\\
&\le3\varepsilon_k^{-1}2^{-m}|f(y)u|.\tag{\eqref{eq:exclm2}}
\end{align*}
Therefore the map $\Phi_+$ is defined and satisfies the desired inequality.
In particular, the image is contained in $B(x,\varepsilon_k^{-1}r)\cap\Pi_+$, where we replaced $\varepsilon_k$ with a smaller one (see Remark \ref{rem:retrad}).
Finally, we remark that the above argument is carried out in the metric ball $B(x,r_0)$ of straining radius, assuming $r<\varepsilon_kr_0$ (replace $\varepsilon_k$ again if necessary).
This completes the proof.
\end{proof}

By the Busemann convexity \eqref{eq:conv}, we see that $B(x,r)\cap\bar B(p_0,R)$ is convex, and hence contractible.
Composing this contraction with the retraction to the fiber of a half strainer map constructed above yields a counterpart to Corollary \ref{cor:ret}.

\begin{cor}\label{cor:exret}
In the situation of Proposition \ref{prop:exret}, $B(x,r)\cap\Pi_+$ is contractible in $B(x,\varepsilon_{k}^{-1}r)\cap\Pi_+$.
\end{cor}

Now the proof of Theorem \ref{thm:exstr} is exactly the same as that of Theorem \ref{thm:str}.

\begin{proof}[Proof of Theorem \ref{thm:exstr}]
The openness of a half strainer map was proved in Proposition \ref{prop:exopen}.
Its local uniform contractibility of fibers then follows from Corollaries \ref{cor:ret} and \ref{cor:exret}, where the former applies to the interior points in $B(p_0,R)$ near $x$, and the latter applies to the boundary points in $\partial B(p_0,R)$ near $x$.
Moreover, the straining radius is locally uniformly bounded below around $x$; see Remark \ref{rem:exrad}.
\end{proof}

As before, the following counterpart to Corollary \ref{cor:glofib} is a direct consequence of Theorems \ref{thm:ung1}, \ref{thm:str}, and \ref{thm:exstr}.

\begin{cor}\label{cor:exglofib}
Let $X$ be a GNPC space, $B$ a tiny ball, $p_0\in B$, and $R>0$ such that $\bar B(p_0,R)\subset B$.
Suppose $f:\bar B(p_0,R)\to\mathbb R^k$ satisfies the following: there exists an open subset $U$ of $\bar B(p_0,R)$ such that
\begin{enumerate}
\item $f$ is a $(k,\delta)$-strainer map on $U\cap B(p_0,R)$,
\item $f$ is a half $(k,\delta)$-strainer map with base point $p_0$ on $U\cap\partial B(p_0,R)$.
\end{enumerate}
Let $K\subset f(U)$ be a compact ANR (such as a compact convex set) such that $f^{-1}(K)\cap U$ is compact.
Then the restriction of $f$ to $f^{-1}(K)\cap U$ is a Hurewicz fibration, provided $\delta<\delta_k$.
\end{cor}

Now we can prove Theorem \ref{thm:exfib} by complementing an extended strainer in the same way as Theorem \ref{thm:fib}.

\begin{proof}[Proof of Theorem \ref{thm:exfib}]
The proof is along the same lines as Theorem \ref{thm:fib}.
We just restrict the previous argument to $\bar B(p_0,R)$.

Let $f|_{\bar B(p_0,R)}$ be a half $(k,\delta)$-strainer map at $x\in\partial B(p_0,R)$ with base point $p_0$.
We construct an arbitrarily small contractible open neighborhood $V$ of $x$ in $\bar B(p_0,R)$ such that $f|_V:V\to f(V)$ is a Hurewicz fibration with contractible fibers.
By Theorems \ref{thm:ung2} and \ref{thm:exstr}, it suffices to construct sufficiently small $V$ such that $f|_V$ has contractible fibers and a contractible image.

Let $F=(f,d(x,\cdot))$ and $r>0$ sufficiently small.
By Propositions \ref{prop:kstr} and \ref{prop:exstr},
\begin{enumerate}
\item $F|_{\bar B(p_0,R)}$ is a $(k+1,\delta)$-strainer map on $f^{-1}(f(x))\cap(B(x,r)\setminus\{x\})\cap B(p_0,R)$;
\item $F|_{\bar B(p_0,R)}$ is a half $(k+1,\delta)$-strainer map with base point $p_0$ on $f^{-1}(f(x))\cap(B(x,r)\setminus\{x\})\cap\partial B(p_0,R)$.
\end{enumerate}

By Propositions \ref{prop:open} and \ref{prop:exopen}, $F|_{\bar B(p_0,R)}$ is an open map on a neighborhood of
\[f^{-1}(f(x))\cap(\bar B(x,r)\setminus B(x,\varepsilon_kr/2))\cap\bar B(p_0,R),\]
where $\varepsilon_k$ is a constant from Proposition \ref{prop:exret}.
Therefore, by Corollary \ref{cor:exglofib}, there exists a compact neighborhood $K$ of $f(x)$ such that the restriction of $F|_{\bar B(p_0,R)}$ to
\[f^{-1}(K)\cap(\bar B(x,r)\setminus B(x,\varepsilon_kr/2))\cap\bar B(p_0,R)\]
is a Hurewicz fibration.
See Figure \ref{fig:exfib}.

\begin{figure}[ht]
\centering
\begin{tikzpicture}
\draw(-3,0)--(3,0) node[right]{$\partial B(p_0,R)$};
\draw(-0.75,0) arc[start angle=180, end angle=360, radius=0.75];
\draw(-2,0) arc[start angle=180, end angle=360, radius=2];
\draw[->](-4.5,-1.5)--(-3.5,0.5) node[above right]{$f$};

\node[below right]at(-45:0.75){$\varepsilon_kr/2$};
\node[below right]at(-45:2){$r$};
\node[above]at(2,0.5){$f^{-1}(f(x))$};

\fill(0,0) circle (1.5pt) node[below]{$x$};
\fill(-3.75,0) circle (1.5pt) node[left]{$f(x)$};
\end{tikzpicture}
\caption{}\label{fig:exfib}
\end{figure}

We define
\[V:=f^{-1}(\mathring K)\cap B(x,r)\cap\bar B(p_0,R).\]
We may assume that $f(V)=\mathring K$ is contractible.
It remains to show that the fibers of $f|_V$ are contractible.
Let $v\in\mathring K$.
Using the homotopy lifting property, we first contract the fiber  $f^{-1}(v)\cap B(x,r)\cap\bar B(p_0,R)$ to a small ball $f^{-1}(v)\cap\bar B(x,\varepsilon_kr/2)\cap \bar B(p_0,R)$.
By Corollary \ref{cor:exret}, $f^{-1}(v)\cap\bar B(x,\varepsilon_kr/2)\cap\bar B(p_0,R)$ is contractible in $f^{-1}(v)\cap B(x,r)\cap\bar B(p_0,R)$.
This gives a contraction of $f^{-1}(v)\cap B(x,r)\cap\bar B(p_0,R)$ in itself, as desired.
\end{proof}

The following is a direct corollary of the above proof.
This is exactly the same as Corollary \ref{cor:fib}, so we omit the proof.

\begin{cor}\label{cor:exfib}
Let $f|_{\bar B(p_0,R)}$ be a half $(k,\delta)$-strainer map at $x\in\partial B(p_0,R)$ with base point $p_0$.
Then for any sufficiently small $r>0$, the restriction of $d(x,\cdot)$ to $f^{-1}(f(x))\cap(B(x,r)\setminus\{x\})\cap\bar B(p_0,R)$ is a Hurewicz fibration, provided $\delta<\delta_k$.
In particular, $f^{-1}(f(x))\cap B(x,r)\cap\bar B(p_0,R)$ is contractible.
\end{cor}

\subsection{Consequences}\label{sec:excon}

In this subsection, using the properties of half strainer maps, we prove two propositions which will be used in the proof of Theorem \ref{thm:four}.
These correspond to \cite[Proposition 5.5, Corollary 7.2]{LNS}.

The first is the contractibility of a ``hemisphere'' in the fiber of a half strainer map, which corresponds to \cite[Proposition 5.5]{LNS}.
The proof is somewhat different from the original which used a limiting argument with the property of the tangent cone of a GCBA space.
Here we give in some sense a more direct proof.

\begin{prop}\label{prop:hemi}
Let $f|_{\bar B(p_0,R)}$ be a half $(k,\delta)$-strainer map at $x\in\partial B(p_0,R)$ with base point $p_0$, where $\delta<\delta_k$.
Let $\Pi_+=\Pi\cap\bar B(p_0,R)$ be the fiber of $f|_{\bar B(p_0,R)}$ through $x$, where $\Pi$ is the fiber of $f$ through $x$ (in $X$).
Then for any sufficiently small $r>0$, the hemisphere around $x$ in $\Pi_+$,
\[\Pi_+\cap\partial B(x,r)\]
is contractible.
Furthermore, its small perturbation $\Pi_+(v)\cap\partial B(x,r)$ is contractible, where $\Pi_+(v):=f^{-1}(v)\cap\bar B(p_0,R)$ and $v\in\mathbb R^k$ is sufficiently close to $f(x)$ compared to $r$.
\end{prop}

\begin{rem}\label{rem:hemi}
Note that $\Pi_+\cap\partial B(x,r)$, as well as $\Pi_+(v)\cap\partial B(x,r)$, is locally contractible, as already seen in the proof of Corollary \ref{cor:exfib} (compare with the original statement of \cite[Proposition 5.5]{LNS}).
Proposition \ref{prop:hemi} is still a little weaker than \cite[Proposition 5.5]{LNS} because the original statement perturbs not only the value of $f$ but also the distance $R$ from $p_0$.
However, this simpler statement is enough for the proof of Theorem \ref{thm:4sph}, as it takes place on a fixed metric sphere $\partial B(p_0,R)$.
\end{rem}

\begin{proof}
The desired contraction is given by composing one geodesic contraction with two retractions constructed in the previous subsection.

\step{Step 1}
We first consider the standard case, $\Pi_+\cap\partial B(x,r)$.
Suppose $r>0$ is sufficiently small (to be determined later).
Let $x_0$ be the point on the shortest path $p_0x$ such that $|p_0x_0|=R-r$; in other words, the unique intersection point of  $p_0x$ and $\partial B(x,r)$ (note that $x_0$ is not necessarily contained in $\Pi$).
Let
\[\Phi_0:(\bar B(x,r)\cap\bar B(p_0,R))\times[0,1]\to\bar B(x,r)\cap\bar B(p_0,R)\]
be the linearly reparametrized geodesic contraction to $x_0$ (not $x$), i.e., $\Phi_0(y,t)$ is the linear reparametrization of the shortest path from $y$ to $x_0$, where $y\in \bar B(x,r)$ and $t\in[0,1]$.
Note that the image of $\Phi_0$ is contained in $\bar B(x,r)\cap\bar B(p_0,R)$, as the latter is convex by the Busemann convexity \eqref{eq:conv}.
Let us denote by $\Phi_0|$ the restriction of $\Phi_0$ to $(\partial B(x,r)\cap\Pi_+)\times[0,1]$.
We compose $\Phi_0|$  with the following two retractions.
See Figure \ref{fig:hemi} (where, for simplicity, $x_0$ is contained in $\Pi$).

\begin{figure}[ht]
\centering
\begin{tikzpicture}
\draw(-3,0)--(3,0) node[right]{$\partial B(p_0,R)$};
\draw(-1.5,0) arc[start angle=180, end angle=360, radius=1.5];
\draw[->](-4.5,-1.5)--(-3.5,0.5) node[above right]{$f$};
\draw(0,0)--(0,-2.5);
\draw[->,dashed](1.4,-0.1)--(0.1,-1.4);
\draw[->,dashed](-1.4,-0.1)--(-0.1,-1.4);

\node[below right]at(-45:1.5){$r$};
\node[below right]at(2,-1){$\Pi_+$};
\node[above right]at(2,0.5){$\Pi$};

\fill(0,0) circle (1.5pt) node[below left]{$x$};
\fill(0,-1.5) circle (1.5pt) node[below left]{$x_0$};
\fill(-3.75,0) circle (1.5pt) node[left]{$f(x)$};
\end{tikzpicture}
\caption{}\label{fig:hemi}
\end{figure}

By Proposition \ref{prop:exret}, there exists a retraction
\[\Phi_+:\bar B(x,r)\cap\bar B(p_0,R)\to B(x,\varepsilon_k^{-1}r)\cap\Pi_+,\]
provided $r$ is sufficiently small, where $\varepsilon_k$ is a constant depending only on $k$ (slightly smaller than the one of Proposition \ref{prop:exret}).
Furthermore, by Corollary \ref{cor:exfib}, there exists a (deformation) retraction
\[\theta:(B(x,\varepsilon_k^{-1}r)\setminus\{x\})\cap\Pi_+\to\partial B(x,r)\cap\Pi_+,\]
provided $r$ is sufficiently small (lift a deformation retraction of $(0,\varepsilon_k^{-1}r)$ to $\{r\}$).

We consider the composition
\[\Theta:=\theta\circ\Phi_+\circ\Phi_0|\]
defined on $(\partial B(x,r)\cap\Pi_+)\times[0,1]$.
We have to check that this map is certainly defined, that is, the image of each map is contained in the domain of the next map.
It is clear that if this map is defined, then it gives the desired contraction of the hemisphere.

Clearly the image of $\Phi_0|$ is contained in the domain of $\Phi_+$.
Therefore it suffices to prove that the image of $\Phi_+\circ\Phi_0|$ is contained in the domain of $\theta$, i.e., $(B(x,\varepsilon_k^{-1}r)\setminus\{x\})\cap\Pi_+$.
We show that the image of $\Phi_+\circ\Phi_0|$ misses $x$.

Let $y\in\partial B(x,r)\cap\Pi_+$ and set $z_t:=\Phi_0(y,t)$, i.e., the linearly reparametrized shortest path from $y$ to $x_0$.
By the Busemann convexity \eqref{eq:conv}, we have
\begin{equation}\label{eq:p0zt}
\begin{aligned}
|p_0z_t|&\le(1-t)|p_0y|+t|p_0x_0|\\
&\le(1-t)R+t(R-r)=R-tr.
\end{aligned}
\end{equation}

Recalling the proof of Proposition \ref{prop:exret}, we see that $\Phi_+$ is the limit of $(\Psi_+)^n=(\phi\circ\Phi)^n$. Here $\phi$ is the geodesic retraction to $\bar B(p_0,R)$ centered at $p_0$, and $\Phi$ is the retraction to $\Pi$ constructed in Proposition \ref{prop:ret}.
Furthermore, $\Phi$ keeps the distance from $p_0$ almost unchanged, by the almost orthogonality (Claim \ref{clm:exorth}).
Therefore we have
\begin{align*}
|p_0\Phi(z_t)|&\le|p_0z_t|+\varkappa_k(\delta)\widetilde{|\Phi(z_t)z_t|}\tag{Claim \ref{clm:exorth}}\\
&\le R-tr+\varkappa_k(\delta)\varepsilon_k^{-1}|f(z_t)f(x)|\tag{\eqref{eq:p0zt}, \eqref{eq:tilde}}\\
&\le R-tr+\varkappa_k(\delta)\varepsilon_k^{-1}\sqrt k|z_ty|\tag{$f(x)=f(y)$}\\
&=R-tr+\varkappa_k(\delta)\varepsilon_k^{-1}\sqrt kt|x_0y|\\
&\le R-tr+\varkappa_k(\delta)\varepsilon_k^{-1}\sqrt kt2r<R,
\end{align*}
provided that $\delta$ is small enough and $t$ is positive.
This implies that $\Phi_+(z_t)=\Phi(z_t)$ and $|p_0\Phi_+(z_t)|<R$ for any $0<t\le 1$.
In particular, $\Phi_+(z_t)\neq x$, and thus the composition $\Theta$ is well-defined (note that the $t=0$ case is trivial since $\Phi_+(z_0)=y$).
This completes the proof of the first half.

\step{Step 2}
We next consider the small perturbation, $\Pi_+(v)\cap\partial B(x,r)$.
This immediately follows from the proof of Theorem \ref{thm:exfib} and Step 1.
Indeed, as shown in the proof of Theorem \ref{thm:exfib}, the map $(f,d(x,\cdot))$ is a Hurewicz fibration on a small neighborhood of $\Pi_+\cap\partial B(x,r)$ in $\bar B(p_0,R)$.
In particular, its fibers are homotopy equivalent, and thus $\Pi_+(v)\cap\partial B(x,r)$ is also contractible, provided $v$ is sufficiently close to $f(x)$.
\end{proof}

As the above proof shows, the ``punctured half ball'' $\Pi_+\cap B(x,r)\setminus\{x\}$ in the fiber of a half strainer map is homotopy equivalent to the boundary hemisphere $\Pi_+\cap\partial B(x,r)$ (by using $\theta$ from Corollary \ref{cor:exfib}, which is actually a deformation retraction).
Therefore we obtain the following corollary.

\begin{cor}\label{cor:hemi}
In the situation of Proposition \ref{prop:hemi}, the punctured half ball $\Pi_+\cap B(x,r)\setminus\{x\}$ around $x$ in $\Pi_+$ is contractible.
\end{cor}

The second is the homotopy negligibleness of the ``boundary'' of the fiber of a half strainer map, which corresponds to \cite[Corollary 7.2]{LNS}.
In fact, we do not need this because its consequence, Proposition \ref{prop:half}, is now deduced more directly from Corollaries \ref{cor:exfib} and \ref{cor:hemi} (cf.\ Remark \ref{rem:half}).
Nevertheless, for comparison we include the statement and proof.

Recall that a closed subset $Z$ of a topological space $Y$ is \textit{homotopy negligible} (or a \textit{Z-set}) if, for any open subset $U$ of $Y$, the inclusion
\[U\setminus Z\hookrightarrow U\]
is a homotopy equivalence.
If $Y$ is an ANR, this is equivalent to the following local condition: for any $z\in Z$, there exists a neighborhood basis $\mathcal U_z$ of contractible neighborhoods $U_z$ with contractible complements $U_z\setminus Z$ (\cite{EK}).

\begin{prop}\label{prop:neg}
Let $f|_{\bar B(p_0,R)}$ be a half $(k,\delta)$-strainer map at $x\in\partial B(p_0,R)$ with base point $p_0$, where $\delta<\delta_k$.
Set
\[\Pi_+:=\Pi\cap\bar B(p_0,R),\quad\hat\Pi:=\Pi\cap\partial B(p_0,R),\]
where $\Pi$ is the fiber of $f$ (in $X$) through $x$.
Then there exists an open neighborhood $U$ of $x$ (in $X$) such that $\Pi_+\cap U$ is an ANR and $\hat\Pi\cap U$ is homotopy negligible in $\Pi_+\cap U$.
\end{prop}

Note that our terminology and notation are slightly different from the original ones of \cite{LNS}: see Remarks \ref{rem:exstrm} and \ref{rem:exret}.
In the original terminology, $\hat\Pi$ is a fiber of an extended strainer map $(f,d(p_0,\cdot))$.

\begin{proof}
The proof is along basically the same lines as \cite[Corollary 7.2]{LNS}, but requires an additional argument to handle the problem mentioned in Remark \ref{rem:retball}.
Let $U$ be an open ball around $x$ of the straining radius of the extended strainer defining $f|_{\bar B(p_0,R)}$.
By Corollaries \ref{cor:ret} and \ref{cor:exret}, $\Pi_+\cap U$ is locally contractible, and hence an ANR.

We claim that $\hat\Pi\cap U$ is homotopy negligible in $\Pi_+\cap U$.
Since $\Pi_+\cap U$ is an ANR, this reduces to a local problem, as explained before the proposition.
It suffices to show that for any $z\in\hat\Pi\cap U$, there exists an arbitrarily small contractible neighborhood $U_z$ in $\Pi_+$ such that $U_z\setminus\hat\Pi$ is also contractible.
Since the claim is local, we may assume $z=x$.

By Corollary \ref{cor:exfib}, for any sufficiently small $r>0$, $\Pi_+\cap B(x,r)$ is contractible.
We show that the complement $\Pi_+\cap B(x,r)\setminus\hat\Pi$ is contractible.
Since it is an ANR, it suffices to show that all of its homotopy groups vanish.
To this end, we prove that for any $R'$ less than but close to $R$, a slightly smaller subset
\[\Pi\cap\bar B(p_0,R')\cap B(x,r)\]
of $\Pi_+\cap B(x,r)$ is contractible.
Indeed, by compactness, any continuous image of a sphere into $\Pi_+\cap B(x,r)\setminus\hat\Pi$ is contained in such a subset.
See Figure \ref{fig:neg}.

\begin{figure}[ht]
\centering
\begin{tikzpicture}
\draw(-3,0)--(3,0) node[right]{$\partial B(p_0,R)$};
\draw(-3,-0.25)--(3,-0.25) node[below right]{$\partial B(p_0,R')$};
\draw(-0.75,0) arc[start angle=180, end angle=360, radius=0.75];
\draw(-2,0) arc[start angle=180, end angle=360, radius=2];
\draw[->](-4.5,-1.5)--(-3.5,0.5) node[above right]{$f$};

\node[below right]at(-45:0.75){$\varepsilon_kr/3$};
\node[below right]at(-45:2){$r$};
\node[above right]at(2,0.5){$\Pi$};

\fill(0,0) circle (1.5pt) node[above]{$x$};
\fill(0,-0.25) circle (1.5pt) node[below]{$x'$};
\fill(-3.75,0) circle (1.5pt) node[left]{$f(x)$};
\end{tikzpicture}
\caption{}\label{fig:neg}
\end{figure}

The proof is a minor modification of that of Corollary \ref{cor:exfib}.
By Proposition \ref{prop:kstr}, the map $F=(f,d(x,\cdot))$ is a $(k+1,\delta)$-strainer map on $\Pi\cap B(x,r)\setminus\{x\}$.
Furthermore, by Proposition \ref{prop:exstr} and the openness of an extended strainer (cf.\ Remark \ref{rem:exrad}), $F|_{\bar B(p_0,R')}$ is a half $(k+1,\delta)$-strainer map on
\[\Pi\cap\partial B(p_0,R')\cap\bar B(x,r)\setminus B(x,\varepsilon_kr/3),\]
provided $R'$ is sufficiently close to $R$, where $\varepsilon_k$ is a constant of Proposition \ref{prop:exret}.
By Corollary \ref{cor:exglofib}, we see that $d(x,\cdot)$ is a Hurewicz fibration on $\Pi\cap\bar B(p_0,R')\cap B(x,r)\setminus B(x,\varepsilon_kr/3)$.
Hence $\Pi\cap\bar B(p_0,R')\cap B(x,r)$ is contractible into $\Pi\cap\bar B(p_0,R')\cap\bar B(x,\varepsilon_kr/3)$.

Using the $\varepsilon_{k+1}$-openness of the extended strainer map $\hat f=(d(p_0,\cdot),f)$ (see Remark \ref{rem:exopen}), we find $x'\in\Pi\cap\partial B(p_0,R')$ such that $|xx'|\le\varepsilon_{k+1}^{-1}(R-R')$.
Since $R'$ is sufficiently close to $R$, this implies that $\bar B(x,\varepsilon_kr/3)$ is contained in $B(x',\varepsilon_kr/2)$.
By Corollary \ref{cor:exret} (and Remark \ref{rem:exrad}), $\Pi\cap\bar B(p_0,R')\cap B(x',\varepsilon_kr/2)$ is contractible in $\Pi\cap\bar B(p_0,R')\cap B(x',r/2)$, which is contained in $\Pi\cap\bar B(p_0,R')\cap B(x,r)$.
Therefore $\Pi\cap\bar B(p_0,R')\cap B(x,r)$ is contractible, as desired.
\end{proof}

\section{Proofs}\label{sec:prf}

In this section, we give the proofs of the main results, i.e., Theorems \ref{thm:reg}, \ref{thm:sing}, \ref{thm:four}, \ref{thm:g}, \ref{thm:stab}, and \ref{thm:conv}.

The organization of this section is as follows.
In Section \ref{sec:prfloc}, we prove local results on BNPC manifolds, including Theorems \ref{thm:reg} and \ref{thm:sing}.
In Section \ref{sec:prfglo}, we prove global results on BNPC manifolds, including Theorem \ref{thm:four}.
The contents of these two subsections are almost the same as those in \cite{LN:top} and \cite{LNS}, respectively.
When the argument is exactly the same, we will omit it and refer the reader to the original papers.
In Section \ref{sec:prfapp}, we give applications to G-spaces, i.e., Theorems \ref{thm:g} and \ref{thm:stab}.
In Section \ref{sec:prfgen}, we discuss the generalization to spaces with convex geodesic bicombings, i.e., Theorem \ref{thm:conv}.

\subsection{Local results}\label{sec:prfloc}

In this subsection, we prove local results on BNPC manifolds, including Theorems \ref{thm:reg} and \ref{thm:sing}.
Recall that any locally BNPC homology manifold is a GNPC space (Proposition \ref{prop:geo}).

The proofs of Theorems \ref{thm:reg} and \ref{thm:sing} are based on the following technical result, which is a BNPC counterpart to \cite[Theorem 6.3]{LN:top} of Lytchak--Nagano.
Recall that a $(k,\delta)$-strainer map $f$ at $x\in X$ is \textit{complementable} at $x$ if there exists $p\in X$ such that $(f,d(p,\cdot))$ is a $(k+1,\delta)$-strainer map at $x$ (Definition \ref{dfn:comp}).
For convenience, we treat the constant map to $\mathbb R^0$ as a \textit{$(0,\delta)$-strainer map} for any $\delta>0$.

\begin{thm}\label{thm:rev}
Let $X$ be a locally BNPC homology $n$-manifold and $B$ a tiny ball.
Let $0\le k\le n$.
Suppose $f:U\to\mathbb R^{n-k}$ is an $(n-k,\delta)$-strainer map, where $U\subset B$ is open.
Let $\Pi$ be a fiber of $f$ in $U$.
Define
\[E:=\{x\in\Pi\mid\textit{$f$ is not complementable at $x$}\}.\]
Then the set $E$ is finite and the complement $\Pi\setminus E$ is a topological $k$-manifold, provided $\delta<\delta_{n-k}$.
Moreover, if $k\le 3$, $\Pi$ is a topological $k$-manifold.
Here $\delta_{n-k}$ is a positive constant depending only on $n-k$.
\end{thm}

Theorem \ref{thm:sing} immediately follows from the case $k=n$.

\begin{proof}
The proof is exactly the same as that of \cite[Theorem 6.3]{LN:top}.
Here we only explain why the original proof works in our setting.

In the original CAT($0$) case, the proof of this theorem depends only on the properties (0)--(8) for strainer maps listed in \cite[Section 5]{LN:top}.
In fact, the definition of a strainer was not even stated in \cite{LN:top}, which means that the arguments there do not rely on the specific definition of a strainer.
Therefore it is sufficient to check that all these properties hold for our GNPC strainer map.
In what follows we refer to the original paper \cite{LN:top}.

Properties (0)--(4) are trivial.
Properties (5) and (6) were proved in Propositions \ref{prop:1str} and \ref{prop:kstr}, respectively.
Property (7) follows from the extension property (6), Proposition \ref{prop:kstr} (the ``stronger form'' \cite[Theorem 10.5]{LN:geo} mentioned there is not essential). 
Regarding Property (8), we proved the openness in Theorem \ref{thm:str} (Proposition \ref{prop:open}).
A slight difference is that, as noted in Remark \ref{rem:retball}, it is unclear in the GNPC case whether a metric ball of straining radius in the fiber of a strainer map is contractible in itself (compare with Corollary \ref{cor:fib} where the radius depends on the center).
Still, the required fibration property stated in \cite[Theorem 5.1]{LN:top} holds, as we proved in Theorem \ref{thm:fib} and Corollary \ref{cor:glofib}.

Therefore the original argument from \cite[Theorem~6.3]{LN:top} carries over verbatim.
The set $E$ is finite (Property (7)).
The fact that $\Pi\setminus E$ is a topological $k$-manifold is proved by induction on $k$.
The $k=0$ case follows from Theorem \ref{thm:bilip}, but actually we do not need this for induction.
The $k=1,2$ cases follow from Theorems \ref{thm:moo} and \ref{thm:ray} and the fibration property of strainer maps.
For $k\ge 3$, we need the inductive assumption (hence the actual base case is $k=2$).
The $k=3,4$ cases follow from the fibration property and the fiber bundle recognition (\cite[Section 4.3]{LN:top}).
For $k\ge 5$, the proof is much more complicated, which utilizes the disjoint disk property (\cite[Section 6.1]{LN:top}) via the fibration property.
We omit these technical details.
\end{proof}

The $k=n=3$ case gives an alternative proof of the first half of Theorem \ref{thm:three}, originally by Thurston (cf.\ the first half of \cite[Theorem 6.4]{LN:top}).

\begin{cor}\label{cor:three}
Every locally BNPC homology manifold of dimension $3$ (and thus $\le 3$) is a topological manifold.
\end{cor}

Next we prove Theorem \ref{thm:reg}.

\begin{proof}[Proof of Theorem \ref{thm:reg}]
Let $X$ be a locally compact, locally BNPC space.
We prove that $X$ is a topological $n$-manifold if and only if for any $p\in X$, any sufficiently small punctured ball around $p$ is homotopy equivalent to $S^{n-1}$.
As noted in Section \ref{sec:main}, the homotopy type of a small punctured ball is independent of its radius.

\step{Step 1}
We first show the ``only if'' part.
Suppose $X$ is a topological $n$-manifold.
Let $p\in X$ and $B(p,r)$ a tiny ball.
By assumption, one finds a closed neighborhood $U\subset B(p,r)$ of $p$ that is homeomorphic to the closed $n$-cell $D^n$.
Choose $s>0$ so small that $\bar B(p,s)\subset U$ and find a smaller closed neighborhood $V\subset\bar B(p,s)$ of $p$ that is homeomorphic to $D^n$.
Therefore,
\[V\subset\bar B(p,s)\subset U\subset B(p,r).\]
By the geodesic retraction, $\bar B(p,s)$ is a deformation retract of $B(p,r)$.
We may also assume that $V$ is a deformation retract of $U$.
Therefore $V\setminus\{p\}$ is a deformation retract of $B(p,r)\setminus\{p\}$.
The former is homotopy equivalent to $S^{n-1}$.

\step{Step 2}
We next show the ``if'' part.
Suppose for any $p\in X$, any sufficiently small punctured ball around $p$ is homotopy equivalent to $S^{n-1}$.
Since a small metric ball is contractible, this implies that $X$ is an ANR homology $n$-manifold (by the excision theorem and the Mayer--Vietoris exact sequence).
By Theorem \ref{thm:sing} proved above, the singular points of $X$ are discrete.
Furthermore, by Corollary \ref{cor:three}, we may assume $n\ge 4$.
Now the claim follows from the manifold recognition \cite[Theorem 6.2]{LN:top}, as $S^{n-1}$ is simply connected.
\end{proof}

\begin{rem}\label{rem:reg}
As in \cite[Theorem 1.1]{LN:top}, Theorem \ref{thm:reg} can be slightly generalized as follows:
a locally compact, locally BNPC space $X$ is a topological manifold if and only if every sufficiently small punctured ball is homotopy equivalent to the same space $\Sigma$, which is noncontractible.
In fact, if the latter condition holds, then by Proposition \ref{prop:geo} and Theorem \ref{thm:mfd}, $X$ has manifold points.
Thus $\Sigma$ is homotopy equivalent to a sphere, as assumed in Theorem \ref{thm:reg}.
\end{rem}

The following is a BNPC counterpart to the second half of \cite[Theorem 6.4]{LN:top}.

\begin{cor}\label{cor:4link}
Let $X$ be a locally BNPC topological manifold of dimension $\le 4$ and $p\in X$.
Then any sufficiently small metric sphere around $p$ is homeomorphic to a sphere.
\end{cor}

\begin{proof}
Suppose $r>0$ is small enough (depending on $p$).
By Theorem \ref{thm:link}, $\partial B(p,r)$ is homotopy equivalent to $B(p,r)\setminus\{p\}$.
Since $X$ is a topological manifold, Theorem \ref{thm:reg} shows that the latter is homotopy equivalent to a sphere. 
Furthermore, the $k=n-1\le 3$ case of Theorem \ref{thm:rev} shows that $\partial B(p,r)$ is a topological manifold.
By the resolution of the Poincar\'e conjecture, $\partial B(p,r)$ is homeomorphic to a sphere.
\end{proof}

As mentioned in Section \ref{sec:com}, the above result does not hold for dimension $\ge 5$, due to the double suspension theorem (\cite{Ca}, \cite{Ed}, cf.\ \cite{Be:alex}).
A counterexample is given by the Euclidean cone over the spherical suspension of a CAT($1$) homology sphere.
In general we have the following local structure around a singular point.
The proof is exactly the same as \cite[Theorem 6.5]{LN:top}.
The $k=5$ case is attributed to Steve Ferry.

\begin{thm}\label{thm:cone}
In the situation of Theorem \ref{thm:rev}, any point of $\Pi$ has a neighborhood homeomorphic to an open cone over a homology sphere.
\end{thm}

\subsection{Global results}\label{sec:prfglo}

In this subsection, we prove global results on BNPC manifolds, including Theorem \ref{thm:four}.
The main result is the following BNPC counterpart to \cite[Theorem 1.4]{LNS} of Lytchak--Nagano--Stadler.

\begin{thm}\label{thm:4sph}
Let $X$ be a BNPC homology $4$-manifold.
Then any (arbitrarily large) metric sphere in $X$ is a topological $3$-manifold.
\end{thm}

Together with Theorem \ref{thm:thur} of Thurston, this immediately yields Theorem \ref{thm:four}.

The proof of Theorem \ref{thm:4sph} is along the same lines with the original proof of \cite[Theorem 1.4]{LNS}, which is based only on \cite[Propositions 5.4, 5.5]{LNS}.
Recall that we have already proved a BNPC counterpart to \cite[Proposition 5.5]{LNS}: see Proposition \ref{prop:hemi} (and Remark \ref{rem:hemi}).
Hence it remains to prove a BNPC counterpart to \cite[Proposition 5.4]{LNS}.

The following is a BNPC counterpart to \cite[Proposition 5.4]{LNS}.
The setting is the same as Proposition \ref{prop:neg}, except for the assumption that $X$ is a homology manifold.
Again, note that our terminology and notation are slightly different from the original ones: see Remarks \ref{rem:exstrm} and \ref{rem:exret}.

\begin{prop}\label{prop:half}
Let $X$ be a locally BNPC homology manifold.
Let $f|_{\bar B(p_0,R)}$ be a half $(k,\delta)$-strainer map at $x\in\partial B(p_0,R)$ with base point $p_0$, where $\delta<\delta_k$.
Set
\[\Pi_+:=\Pi\cap\bar B(p_0,R),\quad\hat\Pi:=\Pi\cap\partial B(p_0,R),\]
where $\Pi$ is the fiber of $f$ (in $X$) through $x$.
Then there exists an open neighborhood $U$ of $x$ (in $X$) such that $\Pi_+\cap U$ is an ANR homology manifold with boundary $\hat\Pi\cap U$.
\end{prop}

Using the properties of half strainer maps, we provide a little more direct proof than the original one of \cite[Proposition 5.4]{LNS} (see Remark \ref{rem:half} below).

\begin{proof}
Let $U$ be an open ball around $x$ of the straining radius of the extended strainer defining $f|_{\bar B(p_0,R)}$.
By the first half of Proposition \ref{prop:neg}, $\Pi_+\cap U$ is an ANR.
By Theorems \ref{thm:ray} and \ref{thm:fib}, its interior $\Pi_+\cap U\cap B(p_0,R)$ is a homology manifold.
Hence it remains to show that the local homology vanishes at every point of $\hat\Pi\cap U$.
This immediately follows from Corollaries \ref{cor:exfib} and \ref{cor:hemi} and the Mayer--Vietoris exact sequence.
This completes the proof.
\end{proof}

\begin{rem}\label{rem:half}
Since $\hat\Pi\cap U$ is homotopy negligible in $\Pi_+\cap U$ by Proposition \ref{prop:neg}, the vanishing of the local homology also follows from a result of Toru\'nczyk \cite[Corollary 2.6]{To} (cf.\ \cite{To:cor}), as in the original proof of \cite[Proposition 5.4]{LNS}.
\end{rem}

Now we can prove Theorem \ref{thm:4sph} in exactly the same way as the original proof in \cite[Section 6]{LNS}.

\begin{proof}[Proof of Theorem \ref{thm:4sph}]
The proof is exactly the same as the original proof of \cite[Theorem 1.4]{LNS}.
The original proof relies only on the existence of extended strainers (\cite[Lemma 5.3]{LNS}) and the properties of the ``half spaces'' and ``hemispheres'' associated with extended strainers (\cite[Propositions 5.4, 5.5]{LNS}).
The existence of extended strainers was established in Proposition \ref{prop:exstr}, and the properties of the half spaces and hemispheres were proved in Propositions \ref{prop:half} and \ref{prop:hemi} (cf.\ Remark \ref{rem:hemi}).
Therefore the original proof, using the manifold recognition in dimension $\le 2$ (Theorem \ref{thm:moo}) and the Jordan--Schoenflies theorem, still works in our setting.
We leave it to the reader to check the details.
\end{proof}

The following corollaries are proved in exactly the same way as \cite[Corollaries 1.5, 1.6]{LNS}, only using the geodesic contraction with topological arguments.

\begin{cor}\label{cor:4fib}
Let $X$ be a BNPC topological $4$-manifold and $p\in X$.
Then the distance function $d(p,\cdot):X\setminus\{p\}\to(0,\infty)$ is a fiber bundle with fiber homeomorphic to the $3$-sphere.
\end{cor}

\begin{cor}\label{cor:bdry}
Let $X$ be a BNPC topological $4$-manifold.
Then the ideal boundary $\partial_\infty X$ is homeomorphic to the $3$-sphere, and the canonical compactification $X\cup\partial_\infty X$ is homeomorphic to the closed $4$-cell.
\end{cor}

Here the \textit{ideal boundary} of a BNPC space is the set of geodesic rays emanating from one base point, equipped with the compact-open topology.
This definition does not depend on the choice of a base point (see \cite{Ho}).

\begin{rem}\label{rem:bdry}
There is another notion of boundary for a BNPC space, called the \textit{horofunction boundary} (or \textit{metric boundary}).
Our ideal boundary is also called the \textit{geodesic boundary}.
See \cite{An:bdry} and \cite{An:cone} for details.
These two boundaries coincide for CAT($0$) spaces, but generally differ for BNPC spaces.
The authors do not know whether Corollary \ref{cor:bdry} holds for the horofunction boundary.
\end{rem}

\subsection{Applications}\label{sec:prfapp}

In this subsection, we give two applications on G-spaces, i.e., Theorems \ref{thm:g} and \ref{thm:stab}.

Let us first define a G-space, introduced by Busemann \cite{Bu:geo}.
The following definition may not be standard, but is compatible with our terminology and is equivalent to the original one.

\begin{dfn}\label{dfn:g}
A locally compact, complete geodesic space $X$ is called a \textit{G-space} if each shortest path is locally uniquely extendable: more precisely, for any $p\in X$, there exists $r>0$ such that if $\gamma:[a,b]\to B(p,r)$ is a unit-speed shortest path, then there exists a unique unit-speed shortest path
\[\bar\gamma:[a-r,b+r]\to X\]
such that $\bar\gamma|_{[a,b]}=\gamma$.
\end{dfn}

For the original axiomatic definition, properties, and related problems, we refer to \cite{An:prob}, \cite{BHR}, and \cite{HR}.
Note that a locally BNPC G-space is nothing but a complete GNPC space without branching geodesics.

We briefly review problems on G-spaces.
There are two long-standing problems asked by Busemann himself \cite[Appendix]{Bu:geo}:
\begin{enumerate}
\item Is any G-space finite-dimensional?
\item Is any finite-dimensional G-space a topological manifold? (see also \cite[P. 2 and 49]{Bu:geo}.)
\end{enumerate}

Regarding the finite-dimensionality conjecture (1), Berestovski\v i \cite{Be:dim} proved it under the assumption of a lower Alexandrov curvature bound or the assumption that any small metric ball is convex.
The latter condition is weaker than an upper Alexandrov curvature bound and a local Busemann curvature bound (see \cite{BHR} for further generalization).
Compare with Proposition \ref{prop:dim}.

Regarding the manifold conjecture (2), the dimension $\le 2$ case was proved by Busemann \cite{Bu:geo}, the dimension $3$ and $4$ cases were proved by Krakus \cite{Krak} and Thurston \cite{Th:g}, respectively.
The dimension $\ge 5$ case still remains open.
In general, it is only known that any finite-dimensional G-space is a homogeneous ANR homology manifold (\cite{Th:g}, \cite{BHR}).
Here \textit{(topological) homogeneity} means that any point can be moved to any other point by a self-homeomorphism of the space.

The manifold conjecture was also proved under several curvature bounds.
See Berestovski\v i's results \cite{Be:one}, \cite{Be:up} for Alexandrov curvature bounds and Andreev's results \cite{An:prf}, \cite{An:norm} for a Busemann curvature bound.
In these cases, more rigid geometric structures (Riemannian or Finsler-like structures) were obtained, which can be viewed as the special cases of more general results (see Section \ref{sec:prob} and references therein).
Pogorelov \cite{Po} also proved both conjectures by assuming some regularity of distance function and obtained a Finsler structure.

Here we provide an alternative proof of Andreev's first result \cite{An:prf} on the manifold conjecture (Theorem \ref{thm:g}).
In fact, due to the homogeneity of a G-space mentioned above, it is sufficient to find one manifold point for proving the manifold conjecture.
With additional assumptions of curvature bounds, this is relatively easy (compared to the general case), as a result of strainer arguments.

\begin{proof}[Proof of Theorem \ref{thm:g}]
To show the first half, let $X$ be a locally BNPC G-space.
In particular, $X$ is a GNPC space.
By Theorem \ref{thm:mfd}, $X$ contains a manifold point.
Since any G-space is homogeneous (\cite[Theorem 2.5]{Th:g}, \cite[Corollary 3.12]{BHR}), $X$ is a topological manifold.

Next we show the second half.
Suppose $X$ is a (globally) BNPC G-space.
From the first half, $X$ is a topological manifold.
Furthermore, for every $p\in X$, any two open metric balls around $p$ are homeomorphic to each other along the unique nonbranching geodesics emanating from $p$.
Choose one Euclidean neighborhood $V_1$ of $p$, and inductively enlarge to an exhaustive sequence of open subsets $V_1\subset V_2\subset\cdots$ with each $V_i$ homeomorphic to Euclidean space.
Then Brown's result \cite{Br} concludes that $X$ is homeomorphic to Euclidean space.
\end{proof}

In fact, we can prove the following stronger result.
Compare with Corollary \ref{cor:4link} and the subsequent paragraph.
See also Andreev's second result \cite{An:norm}.

\begin{thm}\label{thm:glink}
Let $X$ be a locally BNPC G-space and $p\in X$.
Then any sufficiently small metric sphere around $p$ is homeomorphic to a sphere.
\end{thm}

\begin{proof}
By Theorem \ref{thm:g}, $X$ is a topological $n$-manifold for some $n$.
By Theorems \ref{thm:reg} and \ref{thm:link}, for any sufficiently small $r>0$, $\partial B(p,r)$ is homotopy equivalent to $S^{n-1}$ (this is indeed true for general G-spaces, \cite[Theorem 1.1]{Gu:g}).
In view of the Poincar\'e conjecture, it suffices to show that $\partial B(p,r)$ is a topological manifold.

By a result of Berestovski\v i--Halverson--Repov\v s \cite[Theorem 1.1]{BHR}, $\partial B(p,r)$ is topologically homogeneous.
Indeed, it is easy to see that any locally BNPC G-space is locally G-homogeneous in the sense of \cite[Definition 4.3]{BHR} (cf.\ \cite[Remark 4.4]{BHR}).
By Theorem \ref{thm:rev}, $\partial B(p,r)$ is a topological manifold except at finitely many points, provided $r$ is small enough.
This completes the proof.
\end{proof}

Thanks to the unique extension property of geodesics, we immediately obtain the following corollary.
Compare with Corollary \ref{cor:bdry}.
See also \cite{DJ} for examples in dimension $\ge 5$.

\begin{cor}\label{cor:gbdry}
Let $X$ be a (globally) BNPC G-space.
Then the ideal boundary $\partial_\infty X$ is homeomorphic to a sphere, and the canonical compactification $X\cup\partial_\infty X$ is homeomorphic to a closed cell.
\end{cor}

Next we prove Theorem \ref{thm:stab}.
For the Gromov--Hausdorff convergence, we refer to \cite[Chapter 7]{BBI}.
We also use some notions and results listed in \cite[Section 4]{LN:top}.

We first recall the following globalization, which is actually the final step of the proof of the Cartan--Hadamard theorem (see Section \ref{sec:bnpc}).
For the convenience of the reader, we include the proof (cf.\ \cite[Theorem 9.3.4]{Pa}).

\begin{lem}\label{lem:ch}
Let $X$ be a locally BNPC, complete geodesic space.
Suppose that a shortest path between given two points in $X$ is unique and that the shortest path depends continuously on its endpoints with respect to the uniform distance.
Then $X$ is a (globally) BNPC space.
\end{lem}

\begin{proof}
It is sufficient to prove the Busemann monotonicity \eqref{eq:mono}.
Let $\gamma,\eta:[0,1]\to X$ be linearly reparametrized shortest paths emanating from $p\in X$.

We first consider the special case.
Let $m$ be a point on the shortest path connecting $\gamma(1)$ and $\eta(1)$ and $\zeta:[0,1]\to X$ be the linearly reparametrized shortest path from $p$ to $m$.
Suppose the Busemann monotonicity \eqref{eq:mono} holds for the pairs $(\gamma,\zeta)$ and $(\zeta,\eta)$.
Then by the triangle inequality, we have
\begin{align*}
|\gamma(t)\eta(t)|&\le|\gamma(t)\zeta(t)|+|\zeta(t)\eta(t)|\\
&\le t|\gamma(1)\zeta(1)|+t|\zeta(1)\eta(1)|\\
&=t|\gamma(1)\eta(1)|,
\end{align*}
that is, the Busemann monotonicity \eqref{eq:mono} holds for the pair $(\gamma,\eta)$.

Now we consider the general case.
In this case, subdivide the shortest path connecting $\gamma(1)$ and $\eta(1)$ into sufficiently small pieces by points $\gamma(1)=m_0,\dots,m_k=\eta(1)$.
For each $0\le i\le k$, let $\zeta_i:[0,1]\to X$ denote the linearly reparametrized shortest path from $p$ to $m_i$.
If we show the Busemann monotonicity \eqref{eq:mono} for every adjacent pair $(\zeta_i,\zeta_{i+1})$, we obtain the desired monotonicity for the pair $(\gamma,\eta)$, from the previous special case.
If the subdivision is sufficiently fine, the shortest paths $\zeta_i$ and $\zeta_{i+1}$ are sufficiently close in the uniform distance by assumption.
Then by the local BNPC condition, the function $|\zeta_i(t)\zeta_{i+1}(t)|$ is locally convex.
Since any local convex function is (globally) convex, $|\zeta_i(t)\zeta_{i+1}(t)|$ is convex, as desired.
\end{proof}

Next we prove the homotopical stability of small metric spheres in GNPC spaces with respect to the Gromov--Hausdorff convergence.
Compare with \cite[Proposition 2.3(2)]{LN:top}.
Strictly speaking, we need the notion of convex geodesic bicombing here, but we defer it to the next subsection.

\begin{lem}\label{lem:stabsph}
Suppose a sequence of closed tiny balls $\bar B_j(p_j,R)$ in GNPC spaces converges to a compact metric space $\bar B$ in the Gromov--Hausdorff topology, where $p_j$ converges to $p\in\bar B$.
Then for any sufficiently small $r>0$, $\partial B(p_j,r)$ is homotopy equivalent to $\partial B(p,r)$, provided $j$ is large enough (depending on $r$).
\end{lem}

\begin{proof}
Note that $\bar B=\bar B(p,R)$.
As we will see later, $B(p,R)$ is a ``weakly'' GNPC space (Definition \ref{dfn:weak}, Proposition \ref{prop:conv}).
Roughly speaking, this means that $B(p,R)$ satisfies the definition of a GNPC space except for the local uniqueness of shortest paths.
For simplicity, here we assume that $B(p,R)$ is a GNPC space.

By Proposition \ref{prop:1str}, for any sufficiently small $r>0$, $p$ is a $(1,\delta)$-strainer on $\partial B(p,r)$.
We claim that the same is true for $p_j$, that is, $p_j$ is a $(1,\delta)$-strainer on $\partial B(p_j,r)$, provided $j$ is large enough.
Since the proof of Proposition \ref{prop:1str} is based on Proposition \ref{prop:bran}, the almost nonbranching of geodesics, we may assume that  $r$ is less than $r_0(p,\delta)$ from Proposition \ref{prop:bran}.

Suppose that Proposition \ref{prop:bran} does not hold for $p_j$ on $\partial B(p_j,r)$, i.e., there exist $x_j\in\partial B(p_j,r)$ and $y_j\in B(x_j,\delta r)$ such that $|2x_j,2y_j|\ge\varkappa(\delta)r$, where $r$ is fixed.
By the Arzela--Ascoli theorem (\cite[Theorem 2.5.14]{BBI}), we find $x\in\partial B(p,r)$ and $y\in\bar B(x,\delta r)$ such that $|2x,2y|\ge\varkappa(\delta)r$.
This means that Proposition \ref{prop:bran} does not hold on $\partial B(p,r)$, which contradicts the choice of $r$ (the possibility that $y$ lies on $\partial B(x,\delta r)$ is not important and can be ignored by adjusting the constants appropriately).

Hence Proposition \ref{prop:bran} holds for $p_j$ on $\partial B(p_j,r)$, provided $j$ is large enough.
By repeating the proof of  Proposition \ref{prop:1str}, we see that $p_j$ is a $(1,\delta)$-strainer on $\partial B(p_j,r)$ (where we replaced $\varkappa(\delta)$ with $\delta$ as in the proof of Proposition \ref{prop:1str}).
Furthermore, as mentioned in Remark \ref{rem:1strrad}, the straining radius of $p_j$ on $\partial B(p_j,r)$ is uniformly bounded below by $\delta'r$, where $\delta'\ll\delta$.

By Corollary \ref{cor:ret} and Remark \ref{rem:rad}, $\partial B(p_j,r)$ and $\partial B(p,r)$ are locally uniformly contractible in the sense of \cite[Definition 4.1]{LN:top}.
In addition, by the $\varepsilon$-openness of a strainer map, Proposition \ref{prop:open}, $\partial B(p_j,r)$ certainly converges to $\partial B(p,r)$ in the Gromov--Hausdorff topology (more precisely, any point of $\partial B(p,r)$ is a limit of points of $\partial B(p_j,r)$).
Therefore, by Petersen's stability theorem \cite[Theorem A]{Pete} (\cite[Theorem 4.2]{LN:top}), we conclude that $\partial B(p_j,r)$ is homotopy equivalent to $\partial B(p,r)$, as desired.

The general case where $B(p,R)$ is a weakly GNPC space is exactly the same.
See the next subsection for details.
\end{proof}

\begin{rem}\label{rem:stabstr}
The above result is actually the special case of the homotopical stability of the fibers of strainer maps.
See \cite[Theorem 13.1]{LN:geo} for more details.
\end{rem}

Now we prove Theorem \ref{thm:stab}.
The proof is almost the same as \cite[Theorem 7.1, Corollary 7.2]{LN:top}.
Recall that a G-space has \textit{injectivity radius $\ge r$} if every shortest path of length $<r$ extends uniquely to length $r$.
In particular, any two points at distance $<r$ are joined by a unique shortest path.

\begin{proof}[Proof of Theorem \ref{thm:stab}]
Let $X_j$ be as in the assumption of Theorem \ref{thm:stab}, i.e., a sequence of compact, locally BNPC G-spaces of dimension $n$ and injectivity radius $\ge r$ that converges to a compact metric space $X$.
Note that by Theorem \ref{thm:g}, $X_j$ is a topological manifold.

From Lemma \ref{lem:ch} (and the fact mentioned before this proof), any point of $X_j$ has a tiny ball of radius $r/100$.
By Corollary \ref{cor:weak} in the next subsection, the limit space $X$ is a weakly GNPC space in the sense of Definition \ref{dfn:weak} such that every point has a tiny ball of radius $r/100$.
For simplicity, here we assume that $X$ is a GNPC space.

Since $X_j$ and $X$ are uniformly locally contractible in the sense of \cite[Definition 4.1]{LN:top} and $X$ is finite-dimensional by Proposition \ref{prop:dim}, it immediately follows from Petersen's stability theorem \cite[Theorem A]{Pete} that $X_j$ is homotopy equivalent to $X$ for any large $j$.
More precisely, for any given $\varepsilon>0$, $X_j$ and $X$ are \textit{$\varepsilon$-equivalent}, provided $j$ is large enough (see \cite[Theorem 4.2]{LN:top}).
In view of the $\alpha$-approximation theorem of Chapman--Ferry \cite{CF} (see also \cite[Theorem 4.7]{LN:top} and references therein), it remains to show that $X$ is a topological manifold.

To this end, we apply the topological regularity theorem, Theorem \ref{thm:reg}, to $X$.
In view of Theorem \ref{thm:link}, it suffices to show that for any $p\in X$, a sufficiently small metric sphere $\partial B(p,r_p)$ is homotopy equivalent to $S^{n-1}$, where $r_p>0$ depends on $p$.
We may assume that $r_p$ is less than the injectivity radius bound $r$.

By Theorem \ref{thm:glink}, there exists $0<r_j<r_p$ depending on $p_j$ such that $\partial B(p_j,r_j)$ is homeomorphic to $S^{n-1}$ (actually a homotopy equivalence is enough).
Since the injectivity radius is bounded below by $r>r_p$, $\partial B(p_j,r_j)$ is homeomorphic to $\partial B(p_j,r_p)$ via the ``exponential map'' defined by the uniquely extendable geodesics.
Therefore $\partial B(p_j,r_p)$ is homeomorphic to $S^{n-1}$.
By Lemma \ref{lem:stabsph}, $\partial B(p,r_p)$ is homotopy equivalent to  $S^{n-1}$, as desired.

The general case where the limit space $X$ is a weakly GNPC space is exactly the same, as we will see in the next subsection.
This completes the proof.
\end{proof}

\begin{rem}\label{rem:4stab}
Except for the fact that $X_j$ is a topological manifold, the assumption that $X_j$ is a G-space is only used to show that $\partial B(p_j,r_j)$ is homeomorphic to $\partial B(p_j,r_p)$ via the exponential map.
Therefore, in dimension $\le 4$, to obtain the same conclusion as Theorem \ref{thm:stab}, it suffices to only assume that $X_j$ is a topological manifold: the G-space assumption, or more specifically, the nonbranching of geodesics, is not necessary.
This is the same as the CAT($0$) case.

More specifically, to carry out the above proof, it is enough to show that $\partial B(p_j,r_p)$ is homotopy equivalent to $S^{n-1}$ for fixed $r_p>0$.
The dimension $4$ case follows from Corollary \ref{cor:4fib}, the dimension $3$ and $2$ cases follow from Proposition \ref{prop:hom} and Theorem \ref{thm:moo}, and the dimension $1$ case is trivial.
\end{rem}

\begin{rem}\label{rem:krak}
In the above proof, we used the fact that $\partial B(p_j,r_p)$ and $\partial B(p,r_p)$ are locally uniformly contractible in the sense of \cite[Definition 4.1]{LN:top}, which was shown by a strainer argument in the proof of Lemma \ref{lem:stabsph}.
This fact also follows from the assumption that $X_j$ is a G-space with injectivity radius $\ge r>r_p$.

Indeed, by using Krakus' retraction inequality \cite{Krak} (see also \cite[Proposition 2.3]{Th:g}), one can show that any point of $\partial B(p_j,r_p)$ has an arbitrarily small contractible neighborhood (see the proof of \cite[Theorem 2.12]{Th:g}).
Hence $\partial B(p_j,r_p)$ are locally uniformly contractible.
Furthermore, by the stability of local uniform contractibility (\cite[Section 16, Main Theorem]{Bo}; see also \cite[Section 5, Theorem]{Pete}, \cite[Theorem 1.2]{GPW}), we see that $\partial B(p,r_p)$ is also locally uniformly contractible.
\end{rem}

As a corollary, we obtain the following finiteness theorem.
For $\varepsilon>0$, the \textit{$\varepsilon$-packing number} of a metric space $X$ is the maximal number of $\varepsilon$-separated points in $X$, i.e., points with pairwise distance $\ge\varepsilon$.

\begin{cor}\label{cor:fin}
Fix $r>0$ and $L>0$.
Let $\mathcal G(r,L)$ denote the family of compact, locally BNPC G-spaces with injectivity radius $\ge r$ and $r/100$-packing number $\le L$.
Then $\mathcal G(r,L)$ contains only finitely many homeomorphism types.
\end{cor}

\begin{proof}
Let $X\in\mathcal G(r,L)$.
By Lemma \ref{lem:ch} and the injectivity radius bound, every point of $X$ has a tiny ball of radius $r/100$. 
By the proof of Proposition \ref{prop:dim} and the packing number bound, every tiny ball in $X$ of radius $r/100$ is $L$-doubling.
In particular, the dimension of $X$ is bounded above by a constant depending only on $L$.
Since $X$ is covered by at most $L$ tiny balls of radius $r/100$, this implies that for any given $\varepsilon>0$, the minimal number of $\varepsilon$-balls needed to cover $X$ is bounded above by a constant $N_{r,L}(\varepsilon)$ depending only on $r$, $L$, and $\varepsilon$.
Therefore, the family $\mathcal G(r,L)$ is precompact in the Gromov--Hausdorff topology (cf.\ \cite[Proposition 5.10]{LN:geo}).

Now we prove the claim by contradiction.
Suppose there exists a sequence $X_j\in\mathcal G(r,L)$ such that every two elements are not homeomorphic to each other.
Since the dimension of $X_j$ is uniformly bounded above, passing to a subsequence, we may assume that $X_j$ have the same dimension.
By precompactness, we may further assume $X_j$ converges to a compact space $X$.
Then by Theorem \ref{thm:stab}, $X_j$ is homeomorphic to $X$ for any large $j$, a contradiction.
\end{proof}

\begin{rem}\label{rem:fin}
In fact, such finiteness extends to a much more general setting without any curvature assumption (except for dimension $3$), since the family of G-spaces with a uniform injectivity radius bound is locally uniformly contractible.
See \cite[Theorem 1]{Fe} (cf.\ \cite[Main Theorem]{GPW}, \cite{GPW:err}).
On the other hand, topological stability like Theorem \ref{thm:stab} is not known in this curvature-free setting.
Ferry--Okun \cite{FO} (cf.\ \cite{Mo}) constructed a sequence of Riemannian metrics on the $5$-sphere, each admitting a uniform contractibility function; this sequence converges to an infinite-dimensional space. Still, it is unclear to the authors what can be said about the convergence under additional geometric assumptions, such as a uniform lower bound on the injectivity radius.
\end{rem}

\subsection{Generalization}\label{sec:prfgen}

In this subsection, we discuss the generalization to spaces with convex geodesic bicombings introduced by Descombes--Lang \cite{DL:conv}.
In particular, we provide more details on Theorem \ref{thm:conv}.

Spaces with convex geodesic bicombings are a natural generalization of BNPC spaces.
They have recently attracted the attention of geometric group theorists, because every hyperbolic group admits a proper and cocompact action by isometries on a space with a convex geodesic bicombing (\cite{La}; more precisely, on an injective metric space), while it is a long-standing problem whether the same holds for a CAT($0$) space.
From our viewpoint, these spaces naturally arise as Gromov--Hausdorff limits of BNPC spaces, as already used in the previous section.

Let $X$ be a geodesic space.
A \textit{geodesic bicombing} on $X$ is a map
\[\sigma:X\times X\times[0,1]\to X\]
that assigns linearly reparametrized shortest paths between two points of $X$, i.e., for any $x,y\in X$, $\sigma_{xy}(t):=\sigma(x,y,t)$ is a constant-speed shortest path from $x$ to $y$.

\begin{dfn}\label{dfn:conv}
A geodesic bicombing $\sigma$ on $X$ is \textit{convex} if for any $x,x',y,y'\in X$, the function
\[t\mapsto d(\sigma_{xy}(t),\sigma_{x'y'}(t))\]
is convex on $[0,1]$.
\end{dfn}

In addition, we always assume the consistency and reversibility of a bicombing.
A geodesic bicombing $\sigma$ on $X$ is \textit{consistent} if for any $x,y\in X$, setting $p:=\sigma_{xy}(s)$ and $q:=\sigma_{xy}(t)$, where $0\le s\le t\le 1$, we have
\[\sigma_{pq}(\lambda)=\sigma_{xy}((1-\lambda)s+\lambda t)\]
for all $0\le\lambda\le1$.
It is \textit{reversible} if for any $x,y\in X$, we have
\[\sigma_{xy}(t)=\sigma_{yx}(1-t)\]
for all $0\le t\le1$.

Clearly, if a geodesic bicombing is convex, then it is continuous.
Any BNPC space admits a unique consistent, reversible, convex geodesic bicombing.
Conversely, a (complete) metric space with a convex geodesic bicombing is a BNPC space if it is uniquely geodesic, i.e., a shortest path between any two points is unique.

Any proper injective metric space of finite topological dimension admits a unique consistent, reversible, convex geodesic bicombing (\cite[Section 1]{DL:conv}).
A combinatorial condition for a simplicial complex equipped with some polyhedral metric to admit a convex geodesic bicombing can be found in \cite{Ha}.
Several theorems known for BNPC spaces are extended to spaces with convex geodesic bicombings (\cite{De}, \cite{DL:flat}, \cite{Mie}).

To make the terminology compatible with ours, let us call $\sigma_{xy}$ the \textit{$\sigma$-shortest path} between $x$ and $y$ (to be more precise, its unit-speed reparametrization).
Then Definition \ref{dfn:conv} is the same as Definition \ref{dfn:bnpc} with ``shortest path'' replaced by ``$\sigma$-shortest path,'' except for the completeness assumption.

For a (consistent, reversible) geodesic bicombing $\sigma$, a \textit{$\sigma$-geodesic} is a curve whose restriction to each sufficiently small interval is a $\sigma$-shortest path.
Note that if $\sigma$ is convex, then any $\sigma$-geodesic is a $\sigma$-shortest path; compare with Section \ref{sec:bnpc}.
We say that a geodesic bicombing $\sigma$ is \textit{locally extendable} (resp.\ (globally) extendable) if any $\sigma$-geodesic is extendable to a $\sigma$-geodesic locally (resp.\ infinitely).
These conditions are equivalent if $X$ is complete; compare with Section \ref{sec:geo} (cf.\ \cite[Proposition 9.1.28]{BBI}).

With this terminology, instead of a GNPC space, we introduce

\begin{dfn}\label{dfn:weak}
A \textit{weakly GNPC space} is a separable, locally compact metric space such that any point has a neighborhood that admits a locally extendable, consistent, reversible, convex geodesic bicombing (with respect to the restriction of the original metric).
\end{dfn}

Then we have

\begin{clm}\label{clm:weak}
All the results in this paper can be generalized from GNPC spaces to weakly GNPC spaces.
More precisely, if we replace
\begin{itemize}
\item ``BNPC space'' by ``complete metric space with a consistent, reversible, convex geodesic bicombing,''
\item ``shortest path'' by ``$\sigma$-shortest path,''
\item ``geodesic completeness'' with ``extendability of a bicombing,''
\end{itemize}
then all the arguments work without any further changes.
Note that the definition of a G-space is weakened in a similar manner.
\end{clm}

This is because our arguments rely only on the contraction along geodesics and the convexity of distance function along geodesics, both of which are available for $\sigma$-geodesics even in the weaker setting.
The uniqueness of shortest paths is never used essentially in our arguments.
Even if there is a place where the uniqueness of shortest paths appears to be used, it can be replaced with the consistency of $\sigma$-shortest paths: for instance, see the proof of Proposition \ref{prop:geo}.
The reversibility is also implicitly used, for example, in the proof of Lemma \ref{lem:sum}.
Lemma \ref{lem:ch} also holds for a continuous, consistent, reversible geodesic bicombing $\sigma$ defined globally on $X$ that is locally convex.

On the other hand, the non-uniqueness of shortest paths typically results in the non-existence of a unique projection to a ``$\sigma$-convex subset'' (compare \cite[Corollary 8.2.6, Proposition 8.4.8]{Pa}).
However, we have never used such a projection.

The advantage of dropping the uniqueness of shortest paths is that this weaker notion is preserved under the Gromov--Hausdorff convergence, which was essential in the previous subsection.
Compare the following with \cite[Example 4.3]{LN:geo}.

\begin{prop}\label{prop:conv}
Let $X_j$ be a sequence of compact metric spaces with consistent, reversible, convex geodesic bicombings $\sigma_j$.
Suppose $X_j$ converges to a compact metric space $X$ in the Gromov--Hausdorff topology.
Then $X$ admits a consistent, reversible, convex geodesic bicombing $\sigma$.

Furthermore, if $\sigma_j$ is locally extendable in $B(p_j,r)$, then $\sigma$ is locally extendable in $B(p,r)$, where $p_j\in X_j$ converges to $p\in X$.
\end{prop}

\begin{proof}
Let $S\subset X$ be any countable dense subset.
By a standard diagonal argument using the Arzela--Ascoli theorem (\cite[Theorem 2.5.14]{BBI}), after passing to a subsequence, one can define the desired convex geodesic bicombing $\sigma$ on $S\times S$ as a limit of $\sigma_j$.
Furthermore, by the convexity of $\sigma$, if $s_i,s_i'\in S$ are Cauchy sequences, then so is $\sigma_{s_is_i'}(t)$ for all $0\le t\le1$.
Therefore $\sigma$ extends uniquely to $X\times X$.
Clearly $\sigma$ is consistent, reversible, and convex, since it is a limit of $\sigma_j$.

Suppose next $\sigma_j$ is locally extendable in $B(p_j,r)$.
Fix $x,y\in B(p,r)$ and let $x_j,y_j\in B(p_j,r)$ be arbitrary sequences converging to $x,y$, respectively.
Since $x_j,y_j$ are uniformly bounded away from the boundary of $B(p_j,r)$, the $\sigma_j$-shortest path between $x_j$ and $y_j$ is extendable by a fixed length independent of $j$ (where we used the local-to-global argument for extendability as in \cite[Proposition 9.1.28]{BBI} and the fact that any $\sigma_j$-geodesic is a $\sigma_j$-shortest path).
Therefore the $\sigma$-shortest path between $x$ and $y$ is also extendable as a limit of the extendable $\sigma_j$-shortest paths.
This completes the proof.
\end{proof}

Applying the above proposition to tiny balls, we obtain

\begin{cor}\label{cor:weak}
Let $X_j$ be a sequence of compact weakly GNPC spaces with a uniform positive lower bound on the radii of tiny balls.
Suppose $X_j$ converges to a compact metric space $X$ in the Gromov--Hausdorff topology.
Then $X$ is a weakly GNPC space with the same lower bound on the radii of tiny balls.
\end{cor}

\section{Problems}\label{sec:prob}

We conclude this paper by raising some remaining problems.
Most of them are related to the geometric structure of GNPC spaces, which we intentionally did not explore in depth here.
See also Remark \ref{rem:prob} for the latest information.

The first is the generalization of upper curvature bounds mentioned in Section \ref{sec:gen}.
By modifying the BNPC condition in terms of comparison angles, one can define \textit{Busemann spaces} with curvature $\le\kappa$ for arbitrary $\kappa\in\mathbb R$ (see Remark \ref{rem:curv}).

\begin{prob}\label{prob:curv}
Generalize our arguments and results to the setting of Busemann spaces with curvature $\le\kappa$, where $\kappa\in\mathbb R$.
\end{prob}

The next problem is the one mentioned in Remark \ref{rem:dim}, which is true for GCBA spaces.
We remark that the original proof in \cite{LN:geo} is based on a strainer argument.

\begin{prob}\label{prob:dim}
Show that the Hausdorff dimension is equal to the topological dimension for GNPC spaces.
\end{prob}

The next three problems are all related to the tangent cones of GNPC spaces, which also generalize previous results for GCBA spaces by Lytchak--Nagano \cite{LN:geo}.
Before discussing the problems, we briefly review the original results in \cite{LN:geo} (cf.\ \cite{OT}, \cite{Ot:diff}, \cite{Be:up}).
For analogous results for Alexandrov spaces with curvature bounded below, see \cite{OS}, \cite{Ot:sec}, \cite{Ot:diff}, \cite{Per:dc}, \cite{Be:one}.
We also refer to \cite{Ly:diff} for tangent cones in more general settings.

Let $X$ be a GCBA space and $p\in X$.
The \textit{tangent cone} $T_pX$ of $X$ at $p$ is defined as the Gromov--Hausdorff blow-up of $X$ at $p$, which in turn is isometric to the Euclidean cone over the \textit{space of directions} $\Sigma_pX$ of $X$ at $p$.
$T_pX$ is a locally compact, geodesically complete CAT($0$) space and $\Sigma_pX$ is a compact, geodesically complete CAT($1$) space.

We say that $p$ is a \textit{$k$-regular point} if $T_pX$ is isometric to $\mathbb R^k$ (with Euclidean metric).
The set $R^k$ of $k$-regular points in $X$ is dense and has full measure with respect to the $k$-dimensional Hausdorff measure in the $k$-dimensional part $X^k$ of $X$.
Here, the \textit{$k$-dimensional part $X^k$} is the set of points in $X$ with $\dim T_pX=k$ (in the sense of the Hausdorff or topological dimension).

Furthermore, there exists a Lipschitz manifold $M^k\subset X^k$, which is open in $X$ and contains $R^k$, such that $M^k$ admits a Riemannian metric that is continuous on $R^k$ and is locally compatible with the original distance.

We want to develop GNPC analogs of the above results.
For simplicity, we ignore the dimensional inhomogeneity of a GNPC space and formulate very rough problems.
Note that these are not rigorous conjectures.

The starting point would be the structure of the tangent cone (cf.\ \cite{Ly:diff}).

\begin{prob}\label{prob:tan}
Describe the structure of the tangent cone of a GNPC space.
Define regular points of a GNPC space and show its almost everywhere existence.
\end{prob}

Considering that the tangent space of a Finsler manifold is a strictly convex normed space, a regular point of a GNPC space should be defined as a point at which the tangent cone is isometric to a strictly convex normed space.
Note that, unlike a Riemannian metric modeled on the standard Euclidean metric, a Finsler norm generally varies pointwise.
However, if a Finsler manifold satisfies the BNPC condition, then the Finsler metric belongs to a special class called a \textit{Berwald metric}, as shown by Ivanov--Lytchak \cite{IL}.
One of the properties of a Berwald metric is that the tangent norms are all mutually isometric.
From our viewpoint that GNPC spaces are a generalization of Finsler manifolds with nonpositive flag curvature, one might ask the following.

\begin{prob}\label{prob:ber}
Does a GNPC space satisfy any kind of Berwaldness?
For example, is it possible to ``decompose" a GNPC space so that in each piece the tangent cone is isometric to a fixed strictly convex normed space at almost every point?
\end{prob}

Related problems and partial counterexamples can be found in \cite[Sections 1, 6]{IL}.
For example, one can glue two different strictly convex normed spaces at one point to obtain a GNPC space such that the tangent cones at ``regular points" are not isometric.
However, the authors do not know more complicated counterexamples to the above problem.

Note that, unlike Riemannian manifolds, Finsler manifolds do not generally have a ``canonical'' measure, but for (reversible) Berwald spaces the Hausdorff measure can be regarded as a canonical measure (see \cite[Section 10.2]{Oh}).
Therefore it would be natural to consider the Hausdorff measure for our GNPC spaces.

Furthermore, based on the notion of a regular point, one could ask the following.
Compare with Theorem \ref{thm:bilip} and Remark \ref{rem:isom}.

\begin{prob}\label{prob:fin}
Does a GNPC space admit any kind of Finsler structure almost everywhere?
Does a regular point have a neighborhood almost isometric to an open subset of its tangent cone?
\end{prob}

To attack the above problems, it would be reasonable to first consider the special case of locally BNPC G-spaces, i.e., (complete) GNPC spaces without branching geodesics, which could be viewed as the ``regular part'' of general GNPC spaces.
Andreev \cite{An:norm} proved that the tangent cone of such a space is indeed a strictly convex normed space.

Another problem we would like to mention here is the diffeomorphism version of Theorem \ref{thm:four} in the cocompact setting.
As mentioned in \cite[Section 6]{DJL}, this is open even in the CAT($0$) case.
See also \cite{St} and \cite{Sa} for related results.

\begin{prob}\label{prob:diff}
Is the universal cover of a closed smooth $4$-manifold with a locally CAT($0$)/BNPC metric diffeomorphic to the standard $\mathbb R^4$?
\end{prob}

Finally, we discuss Busemann spaces with lower curvature bounds.
Recall that Theorems \ref{thm:reg} and \ref{thm:sing} were originally proved by Wu \cite{Wu:alex} for Alexandrov spaces with curvature bounded below.
Reversing the inequality for comparison angles in Remark \ref{rem:curv}, one can also define \textit{Busemann spaces} with curvature $\ge\kappa$, where $\kappa\in\mathbb R$ (cf.\ \cite{Gu:ab}; however, note that there is no equivalent definition in terms of the concavity of the distance function, similar to Definition \ref{dfn:bnpc}, even for nonnegative curvature).
Clearly, every Alexandrov space is a Busemann space with the same lower curvature bound.
Therefore, it is natural to ask

\begin{prob}\label{prob:alex}
Is it possible to extend Theorems \ref{thm:reg} and \ref{thm:sing} to Busemann spaces with curvature bounded below?
\end{prob}

Note that Wu's proof \cite{Wu:alex} (cf.\ \cite{Wu:ed}) relies heavily on the structure theorem for Alexandrov spaces by Perelman \cite{Per:alex}, \cite{Per:mor}, which seems very hard to generalize to Busemann spaces.
We point out that Lytchak--Nagano's proof \cite{LN:top} can also be applied to Alexandrov spaces, by using the gradient flow of the distance function (see \cite{Petr}) instead of the geodesic contraction we frequently used.
However, to the best of our knowledge, it is not even known whether a Busemann space with a lower curvature bound is locally contractible.
We remark that there is work of Kell \cite{Ke} on Busemann spaces with nonnegative curvature.

\begin{rem}\label{rem:prob}
After the first version of this paper appeared on arXiv, there have been several developments related to the above problems.
Regarding Problem \ref{prob:alex}, Han--Yin \cite{HY} developed a structure theory of Busemann nonnegatively curved spaces with some additional conditions.
For Problems \ref{prob:tan} and \ref{prob:fin}, some progress was made in \cite{FT} and \cite{FG:fin}.
A general theory of the geometric structure of GNPC spaces will be discussed in a forthcoming paper by the first author \cite{Fu:geo}, which in particular addresses Problem \ref{prob:dim}.
\end{rem}

\end{document}